\documentclass[a4paper]{amsart}
\usepackage{amscd,amsmath,amssymb,amsfonts,verbatim}
\usepackage[cmtip, all]{xy}
\usepackage{pgf,tikz}
\usepackage{mathrsfs}
\usetikzlibrary{arrows}

\newtheorem{thm}{Theorem}[subsection]
\newtheorem{prop}[thm]{Proposition}
\newtheorem{lem}[thm]{Lemma}
\newtheorem{cor}[thm]{Corollary}

\theoremstyle{definition}

\newtheorem{defn}[thm]{Definition}
\theoremstyle{remark}
\newtheorem{remk}[thm]{Remark}
\newtheorem{remks}[thm]{Remarks}

\newtheorem{exm}[thm]{Example}
\newtheorem{exms}[thm]{Examples}
\newtheorem{notat}[thm]{Notation}
\numberwithin{equation}{subsection}

{\hfill$\square$\end{defn}}
{\hfill$\square$\end{remk}}
{\hfill$\square$\end{remks}}
{\hfill$\square$\end{exm}}
{\hfill$\square$\end{exms}}
{\hfill$\square$\end{notat}}

\newcommand{\CH}{{\rm CH}}

\newcommand{\Spec}{{\rm Spec \,}}

\newcommand{\ds}{{/\kern-3pt/}}

\renewcommand{\TH}{{\operatorname{TCH}}}
\newcommand{\un}{\underline}

\renewcommand{\dim}{\text{\rm dim}}

\newcommand{\tuborg}{\left\{\begin{array}{ll}}
\newcommand{\sluttuborg}{\end{array}\right.}

\renewcommand{\mod}{ {\rm \ mod \ } }

\begin{document}
\title[Presentations of Milnor $K$-groups and norms]{Geometric presentations of Milnor $K$-groups of certain Artin algebras and Bass-Tate-Kato norms}
\author{Jinhyun Park}
\address{Department of Mathematical Sciences, KAIST, 291 Daehak-ro Yuseong-gu, Daejeon, 34141, Republic of Korea (South)}
\email{jinhyun@mathsci.kaist.ac.kr; jinhyun@kaist.edu}




\begin{abstract}
For an arbitrary field $k$, and an arbitrary regular henselian local $k$-scheme $X$ of dimension $1$ with the residue field $k$, we introduce two subcomplexes of the higher Chow complexes of $X$ using certain extended face intersection conditions. We define suitable equivalence relations on them, and prove that their Milnor range cycle class groups offer geometric presentations of the improved (Gabber-Kerz) Milnor $K$-groups of Artin local $k$-algebras of the embedding dimension $1$, and their relative groups, generalizing the theorem of Nesterenko-Suslin and Totaro.

Using these, we prove the existence of the norm and trace maps for the Milnor $K$-groups of the Artin local algebras associated to arbitrary finite extensions of fields, generalizing the Bass-Tate and Kato norms on the Milnor $K$-theory of fields.
\end{abstract}

\maketitle

\setcounter{tocdepth}{1}

\tableofcontents

\section{Introduction}

\subsection{Motivation and the results}
Computing and finding presentations of certain (relative) $K$-groups have been two interconnected and important problems since 1970s (see, e.g. W. van der Kallen \cite{vdK}, \cite{vdK2}, and also the 1978 Helsinki ICM lecture \cite{vdK ICM}, as well as S. Bloch \cite{Bloch 1973}, \cite{Bloch 1975} and F. Keune \cite{Keune} for some of the earliest results).

Around the year 1990, an interesting new way that offers a presentation of the Milnor $K$-group of a field $k$ emerged: by Nesterenko-Suslin \cite{NS} and B. Totaro \cite{Totaro}, there is an isomorphism $K_n ^M (k) \simeq \CH^n (k, n)$ with the higher Chow group of $\Spec (k)$ (see \S \ref{sec:Milnor K}, \S \ref{sec:HC} and Theorem \ref{thm:NST}). Since $\CH^n (k, n)= z^n (k, n) / \partial z^n (k, n+1)$ for two free abelian groups, namely higher Chow cycles, coming from some $0$-cycles on an affine $n$-space and the ``boundaries" of some $1$-cycles on an affine $(n+1)$-space, it offers a presentation of geometric origin, which we would like to call a \emph{geometric} (or motivic, if one prefers) presentation of $K_n ^M (k)$. This isomorphism is functorial on the category of fields, while it also respects the so-called the transfer structure, that there is the commutative diagram
$$
\xymatrix{ K_n ^M (k')\ar[d] ^{ {\rm N}_{k'/k}} \ar[r] ^{\simeq} & \CH^n (k', n) \ar[d] ^{\pi_*} \\
K_n ^M (k) \ar[r] ^{\simeq} & \CH^n (k, n),}
$$
for each finite extension $k \hookrightarrow k'$ and its associated morphism $\pi: \Spec (k') \to \Spec (k)$, where ${\rm N}_{k'/k}$ is the norm of Bass-Tate-Kato norm (\cite{BassTate} and \cite{Kato}) and $\pi_*$ is the push-forward of higher Chow groups in S. Bloch \cite{Bloch HC}. 

This result is now almost 35 years old, but it seems to the author that its philosophical implications and lessons haven't been explored enough. Here is one problem that surprisingly hasn't yet been answered (and will be answered in this article after a correction): suppose we have a finite extension $k \hookrightarrow k'$ of fields. Then do we have a norm operation ${\rm N}_{k'/k}: K^M_n (k'_{m+1}) \to K^M _n (k_{m+1})$ satisfying various requirements, such as the transitivity, where $R_{m+1}: = R[t]/(t^{m+1})$ for any ring $R$?

Recent decades, we learned that the original Milnor $K$-theory may have rather erratic behaviors when the residue fields have small cardinalities, but the improved Milnor $K$-theory $\widehat{K}^M_n (-)$ of Gabber and Kerz \cite{Kerz finite}, onto which the original $K^M_n (-)$ maps surjectively, offers a better platform. For instance, for a finite \emph{separable} extension $k \hookrightarrow k'$, using \cite{Kerz finite} we do have the push-forward ${\rm N}_{k'/k}: \widehat{K}_n ^M (k'_{m+1}) \to \widehat{K}_n ^M (k_{m+1})$. But we knew little about its existence if the extension $k \hookrightarrow k'$ is not separable.

The main question behind this article grew out of attempts to understand this point via the geometry of algebraic cycles: \emph{do we possibly have a good geometric presentation of $\widehat{K}_n ^M (k_{m+1})$ in terms of certain cycle class group?} It is not always the case, but often algebraic cycles behave nicely with respect to push-forwards, so if the first question is answered affirmatively, we may be able to \emph{define} the norms ${\rm N}_{k'/k}$ for any finite extension via the push-forwards of cycles.

In the simplest class of Artin local $k$-algebras $k_{m+1}$ over an arbitrary field $k$, we prove in \S \ref{sec:local case} that these are indeed possible, where the relevant new Chow groups are defined in \S \ref{sec:our cycle}. Here, one can say that finding the right definition is half of the problem:

\begin{thm}\label{thm:intro main v d}
Let $k$ be an arbitrary field and let $m, n \geq 1$ be integers. Let $X$ be any regular henselian local $k$-scheme of dimension $1$ with the residue field $k$. 
\begin{enumerate}
\item We have an isomorphism
$$
 \widehat{K}_n ^M (k_{m+1}) \simeq \CH_{{\rm d}} ^n (X / (m+1), n),
 $$
to the Chow group of ${\rm d}$-cycles of $X$ (Definition \ref{defn:dsf}). This gives a presentation of $\widehat{K}_n ^M (k_{m+1})$.
\item If $r$ is an integer such that $1 \leq r \leq m+1$, we have an isomorphism
$$
 \widehat{K}_n ^M (k_{m+1}, (t^r)) \simeq \CH_{{\rm v}, \geq r}^n (X/ (m+1), n),
 $$
to the Chow group of ${\rm v}$-cycles of order $\geq r$ of $X$ (Definition \ref{defn:M_v r=1}). This gives a presentation of $\widehat{K}_n ^M (k_{m+1}, (t^r))$.
\end{enumerate}
\end{thm}
These are generalizations to the Artin algebra $k_{m+1}$ of Nesterenko-Suslin \cite{NS} and B. Totaro \cite{Totaro} proven for fields (see Theorem \ref{thm:NST}). The reader may also wonder if analogues of Theorems \ref{thm:intro main v d} and \ref{thm:intro main norm} may hold for more general $k$-algebras $R$. The author guesses that it may work at least when $R$ is a regular local $k$-algebra essentially of finite type over $k$, and this is being studied in a separable project. In the process, the results of this article plays a significant role as well.

\medskip

The geometric presentations of Theorem \ref{thm:intro main v d} and the push-forward structures (Lemma \ref{lem:fpf}) on the cycle class groups imply the existence of the following norm / trace maps, which we would like to call still by the name Bass-Tate-Kato norms, as they generalize to $k_{m+1}$ the classical Bass-Tate-Kato norm of \cite{BassTate} and \cite{Kato} for fields:

\begin{thm}\label{thm:intro main norm}
Let $k$ be an arbitrary field and let $m, n \geq 1$ be integers. Let $1 \leq r \leq m+1$ be an integer. Then for each finite extension $k \hookrightarrow k'$ of fields, there exist operations
$$
{\rm N}_{k'/k}: \widehat{K}^M_n (k'_{m+1}) \to \widehat{K}^M_n (k_{m+1}), \mbox{ and}
$$
$$
{\rm Tr}_{k/k} : \widehat{K}^M_n (k'_{m+1}, (t^r)) \to \widehat{K}^M _n (k_{m+1}, (t^r))
$$
such that these are transitive for a tower of finite extensions of fields.
\end{thm}

There is another very nontrivial application of the above theorems, which proves that there is an isomorphism
$$
\mathbb{W}_m \Omega_k ^{n-1} \simeq \widehat{K}_n ^M (k_{m+1}, (t))
$$
between the group of the big de Rham-Witt forms of Hesselholt-Madsen \cite{HeMa} over an arbitrary field $k$, and the relative Milnor $K$-group. This has been known for ${\rm char} (k) =0$ for years, but it has yet been open for ${\rm char} (k) =p>0$ in general.

This is discussed in a separate article \cite{Park vanishing}, and the result answers an old question on the computation of the relative Milnor $K$-group of Artin local $k$-algebras of the embedding dimension $1$ whose residue field is $k$, without any conditions on $k$.

\subsection{Sketch of the main ideas}\label{sec:intro sketch}

To find a geometric presentation of the groups $\widehat{K}_n ^M (k_{m+1})$ as well as the relative groups $\widehat{K}_n ^M (k_{m+1}, (t^r))$, we borrow some geometric intuitions from the deformation theory: observing that $\Spec (k_{m+1})$ is a fat point of the embedding dimension $1$ whose underlying reduced point is $0 \in \mathbb{A}_k ^1$, take any integral regular henselian local $k$-scheme $X= \Spec (A)$ of dimension $1$ whose residue field is $k=A/I$. This gives a $1$-dimensional ``perturbative" neighborhood of $0$ in a suitable sense. We define two subcomplexes of the cubical version of the higher Chow complexes of $X$ of S. Bloch \cite{Bloch HC},
\begin{equation}\label{eqn:subcomplexes intro}
z^q_{{\rm v}}(X, \bullet) \hookrightarrow z^q_{{\rm d}} (X, \bullet) \hookrightarrow z^q (X, \bullet),
\end{equation}
consisting of the cycles called the \emph{strict vanishing cycles} (in short, ${\rm v}$-cycles) and the \emph{dominant cycles} (in short ${\rm d}$-cycles) (see Definitions \ref{defn:HC}, \ref{defn:dsf} and \ref{defn:vanishing cycle}). Here, the ${\rm d}$-cycles are algebraic cycles over $X$  dominant over $X$ with some additional intersection-theoretic properties, while the ${\rm v}$-cycles are ${\rm d}$-cycles that have the empty (i.e. vanishing) special fibers over the unique closed point of $X$. 

To give a further intuition behind the reason why we use these strict vanishing cycles for the relative Milnor $K$-groups, recall first that to define the cubical version of higher Chow groups, we use the affine space $\square^n = (\mathbb{P}^1 \setminus \{ 1 \})^n$, and for each coordinate $y_i$, the ``infinity" is given by the equations $y_i = 1$. Geometrically, the strict vanishing cycles have the closures whose reductions to the special fiber give $y_i = 1$ for some $ 1 \leq i \leq n$. This is analogous to the behavior of a Milnor symbol $\{ a_1, \cdots, a_n \} \in K_n ^M (A)$ with $a_i \equiv 1 \mod I$ for some $i$, which gives $\{ \bar{a}_1, \cdots, \bar{a}_n\} = 0$ in $K_n ^M (A/I)$. See Examples \ref{exm:good eg} and \ref{exm:bad eg}, for concrete examples and non-examples of such ${\rm v}$-cycles.

This definition of the strict vanishing cycles is reminiscent of the following Lemma \ref{lem:KS} for generators of the relative Milnor $K$-theory from Kato-Saito \cite[Lemma 1.3.1, p.26]{KS}: 

\begin{lem}\label{lem:KS} Let $R$ be a local ring and let $I \subset R$ be an ideal. Then the relative Milnor $K$-group $K_n ^M (R, I)=\ker (K_n ^M (R) \to K_n ^M (R/I))$ is generated by the elements of the form $\{ a_1, \cdots, a_n \}$ with $a_i \in R^{\times}$ for all $1 \leq i \leq n$, but for some $1 \leq i_0\leq n$, we have $a_{i_0} \in \ker (R^{\times} \to (R/I)^{\times})$.
\end{lem}

\medskip

For the subcomplexes in \eqref{eqn:subcomplexes intro}, for each integer $m \geq 1$, we define the \emph{mod $I^{m+1}$-equivalence} (see \S \ref{sec:mod t^{m+1}}) to obtain the morphism of complexes
\begin{equation}\label{eqn:subcomplexes intro m}
z_{{\rm v}} ^q (X/ (m+1), \bullet) \hookrightarrow z_{{\rm d}} ^q (X/ (m+1), \bullet).
\end{equation} 
Similar versions were considered in some earlier articles of the author, e.g. Park-\"Unver \cite{PU Milnor}, Park \cite{Park MZ}, but the definitions in this article are different, and could be regarded as improvements.

For each $1 \leq m \leq \infty$, we then have the short exact sequence of complexes of abelian groups
\begin{equation}\label{eqn:main ses intro}
0 \to z^q _{{\rm v}} (X/(m+1), \bullet) \to z^q_{{\rm d}} (X/(m+1), \bullet) \overset{{\rm ev}}{\to} z^q (k, \bullet) \to 0,
\end{equation}
where ${\rm ev}$ is the reduction modulo $I$. When $n=q$, we define $\CH^n_{\rm v} (X/ (m+1), n)$ (resp. $\CH^n _{\rm d} (X/ (m+1), n)$) to be the $n$-th homology groups of the complexes in \eqref{eqn:subcomplexes intro m}, called the Chow group of the \emph{strict vanishing} (resp. \emph{dominant}) cycles mod $I^{m+1}$.

More generally, we define the groups $\CH^n_{{\rm v}, \geq r}(X, n)$ and $\CH^n_{{\rm v}, \geq r}(X/ (m+1), n)$ of the strict vanishing cycles of order $\geq r$, extending $r=1$ to all $r \geq 1$. 

\medskip

In \S \ref{sec:local case}, especially in Theorems \ref{thm:local main 1} and \ref{thm:local main 2}, we give the presentations of the Milnor $K$-groups by the cycle class groups, and this proves Theorem \ref{thm:intro main v d}.

\medskip

The proof of Theorem \ref{thm:intro main v d} is obtained in terms of an inductive moving lemma, and it occupies most of \S \ref{sec:local case}. A similar argument was used in J. Park \cite{Park MZ}, but since our definitions of the cycle class groups are somewhat different from \emph{ibid.}, and since our argument, via induction on the set $\mathbb{N}^n$ with the lexicographic order, improves that of \emph{ibid.}, we present a full proof. It is interesting to remark that the implications of this moving lemma turn out to be \emph{stronger} than the ones deduced from some $t$-adic analytic arguments, used in an earlier version of this article with the function field $k((t))$ of $X= \Spec (k[[t]])$, similar to the one used in Park-\"Unver \cite{PU Milnor}. This seemed to have rendered the previous $t$-adic arguments rather obsolete, so we no longer keep them in this version.

\medskip

Theorem \ref{thm:intro main v d} and \ref{thm:intro main norm} improve our understanding of the Milnor $K$-groups via cycle class groups. Conversely, the Milnor $K$-groups could say something nontrivial about the cycle class groups as well:

\begin{cor}
Let $k$ be a field. Let $X$ be a regular henselian local $k$-scheme of dimension $1$ with the residue field $k$. Let $m, n, r\geq 1$ be integers such that $1 \leq r \leq m+1$. Then we have the short exact sequence
\begin{equation}\label{eqn:ses 1 intro}
0 \to \CH_{{\rm v}, \geq r} ^n (X/ (m+1), n) \overset{\iota}{ \to} \CH_{{\rm d}} ^n (X/ (m+1), n) \to \CH ^n_{{\rm d}} (X/r, n) \to 0.
\end{equation}
The sequence splits when $r=1$.
\end{cor}

The injectivity of $\iota$ in \eqref{eqn:ses 1 intro} is nontrivial to establish purely in terms of cycles. However, via the identifications in Theorem \ref{thm:intro main v d} and the exact sequence 
\begin{equation}\label{eqn:ses 0 intro}
 0 \to \widehat{K}^M_n (k_{m+1}, (t^r)) \to \widehat{K}^M_n (k_{m+1})     \to  \widehat{K}^M_n (k^r) \to 0,
\end{equation}
 one deduces \eqref{eqn:ses 1 intro} for the cycle class groups.
 
 \medskip
 
 Finally, recall that we can regard $\widehat{K}_n ^M (k_{m+1})$ as the motivic cohomology group ${\rm H}_{\mathcal{M}} ^n (k_{m+1}, \mathbb{Z} (n))$ of Elmanto-Morrow  \cite[Theorem 1.8]{EM}. One notes that Theorem \ref{thm:intro main v d} can be rephrased from this perspective as:

\begin{thm}\label{thm:main interpretation}
Let $k$ be an arbitrary field and $m, n \geq 1$ be integers. 

Then for any integral regular henselian local $k$-scheme $X$ of dimension $1$ with the residue field $k$, we have isomorphisms between the motivic cohomology and the Chow group of ${\rm d}$-cycles
\begin{equation}\label{eqn:7 interpretation}
{\rm H}_{\mathcal{M}} ^n (k_{m+1}, \mathbb{Z} (n)) \simeq \widehat{K}_n ^M (k_{m+1}) \simeq \CH_{{\rm d}} ^n (X/ (m+1), n), \ \ \ \mbox{ and }
\end{equation}
and between the group of the vanishing cycles of the motivic cohomology and the Chow group of ${\rm v}$-cycles.
$$
\ker ({\rm ev}_{t=0}: {\rm H}_{\mathcal{M}} ^n (k_{m+1}, \mathbb{Z} (n)) \to {\rm H}_{\mathcal{M}} ^n (k, \mathbb{Z} (n)) \simeq \CH_{{\rm v}} ^n (X/ (m+1), n).
$$
\end{thm}

This theorem shows in particular that when we write $X= \Spec (A)$, we have various options in choosing any such $k$-algebras $A$ between the henselization $\mathcal{O}_{\mathbb{A}^1, 0} ^h$ and the completion $\widehat{\mathcal{O}}_{\mathbb{A}^1, 0} \simeq k[[t]]$ of $\mathcal{O}_{\mathbb{A}^1, 0} $. Between these two extremities, probably there are infinitely many such regular henselian local rings where the above assumptions hold. This answers a question posed to the author by Mark Spivakovski during the 2nd PRIMA Congress at Oaxaca, Mexico, in August 2017, where the author presented a version of the isomorphism similar to \eqref{eqn:7 interpretation} in  \cite{PU Milnor} with $k[[t]]$.

On the other hand, in a private communication years ago with Matthew Morrow, concerning construction of the motivic cohomology of $\Spec (k_{m+1})$, he suggested an idea of using algebraic cycles over Ind-smooth schemes such as $\Spec (\mathcal{O}_{\mathbb{A}^1, 0} ^h)$ (see e.g. \cite[Lemma 04GV]{stacks}) instead of the cycles over $\Spec (k[[t]])$. The reason behind this might be that the higher Chow group functor behaves nicely on smooth schemes, and it commutes with filtered colimits, so the construction with the henselization, instead of the completion, could potentially work better in general.

The author agrees to the point, and the above Theorem \ref{thm:main interpretation} shows that, at least in the Milnor range, it is possible to use \emph{both} $\Spec (\mathcal{O}_{\mathbb{A}^1, 0} ^h)$ and $\Spec (k[[t]])$. However, there indeed is a possibility that this coincidence was just a fluke because we concentrated on the Milnor range in this article. At least in \cite{Park vanishing}, we will see that this flexibility helps in establishing that $\mathbb{W}_m \Omega_k ^{n-1} \simeq \widehat{K}_n ^M (k_{m+1}, (t))$, as it is a bit harder to work with the elements of the henselization than the completion. The situation may become clearer when we compute a cycle class group in an off-Milnor range, for which the author is working on in another project with Sinan \"Unver at this moment.

 \medskip

 \bigskip
 
 \textbf{Conventions:} In this paper, unless said otherwise $k$ is an arbitrary field of arbitrary characteristic, a $k$-scheme is a noetherian separated $k$-scheme of finite Krull dimension, but it is not necessarily of finite type over $k$. The fiber product $\times$ means $\times_k$ unless said otherwise.

\section{Recollection of certain definitions and results}\label{sec:2 recollection}

 We recall a few basic definitions and notations we use. 

\subsection{Recollection on Milnor $K$-theory}\label{sec:Milnor K}

Let $R$ be a commutative ring with $1$. In \S \ref{sec:Milnor K}, we recall the classical definition of the Milnor $K$-groups of rings, e.g. from J. Milnor \cite{Milnor IM}, and discuss a few basic relevant notions.

\subsubsection{The Milnor $K$-groups}

For the multiplicative group $R^{\times}$, regarded as a $\mathbb{Z}$-module, consider the tensor $\mathbb{Z}$-algebra
$$
T_{\mathbb{Z}} R^{\times} := \bigoplus_{n\geq 0} T_n R^{\times},
$$
where $T_0 R^{\times} = \mathbb{Z}$, $T_1 R^{\times} = R^{\times}$, and $T_n R^{\times}= R^{\times} \otimes \cdots \otimes R^{\times}$ is the $n$-fold self tensor product over $\mathbb{Z}$ for $n \geq 2$. The juxtaposition homomorphism
$$ 
T_{n_1} R^{\times} \otimes_{\mathbb{Z}} T_{n_2} R^{\times} \to T_{n_1 + n_2} R^{\times}
$$
for $n_1, n_2 \geq 0$ naturally defines a structure of a $\mathbb{Z}$-graded noncommutative ring structure on $T_{\mathbb{Z}} R^{\times}$. For the two-sided ideal ${\rm St} \subset T_{\mathbb{Z}} R^{\times}$ generated by all elements of the form $a \otimes (1-a)$ over all $a \in R^{\times}$ such that $1-a \in R^{\times}$, the $\mathbb{Z}$-graded ring
$$ 
K_* ^M (R):= T_{\mathbb{Z}} R^{\times} / {\rm St},
$$
is the Milnor $K$-ring of $R$. Its $n$-th graded piece $K_n ^M (R)$ is the $n$-th Milnor $K$-group of $R$. It is known that each $K_n ^M (R)$ is abelian and $K_* ^M (R)$ is a graded-commutative ring. This $K_n ^M (-)$ defines a covariant functor
$$ 
K_n ^M : ({\rm Rings}) \to ({\rm Ab})
$$
from the category of commutative rings with $1$ to the category of abelian groups. 

\medskip

Suppose we are given an ideal $I \subset R$. Define the relative $K$-group to be
$$
K_n ^M (R, I):= \ker \left( K_n ^M (R) \to K_n ^M (R/I) \right).
$$

In case there is a section $s: R/I \to R$ of the homomorphism $R \to R/I$, by the functoriality we deduce the decomposition
$$
 K_n ^M (R) = K_n ^M (R/ I) \oplus K_n ^M (R , I).
 $$

\subsubsection{The improved Milnor $K$-groups of Gabber-Kerz}\label{sec:Milnor Kerz}

Recall the improved Milnor $K$-theory of  Gabber and M. Kerz \cite{Kerz finite}. This is a functor $\widehat{K}_n ^M$ with a surjective transform $K_n ^M ( -) \to \widehat{K}_n ^M (-)$ satisfying a universal property, which roughly says the following: whenever there is a natural transform $\psi : K_n ^M (-) \to F$ on a suitable subcategory $\mathcal{C}$ of the category of rings, where $F$ has the \'etale transfer property, i.e. for each finite \'etale morphism $A \to B$ of local rings in $\mathcal{C}$, we have the norm (transfer) map $F (B) \to F (A)$, then $\psi$ factors via $\widehat{K}_n ^M(-)$. 

We recall some other essential properties of $\widehat{K}_n ^M$ needed:

\begin{thm}[{\cite[Proposition 10]{Kerz finite}}]\label{thm:Khat_univ} For local rings $A$, we have:
\begin{enumerate}
\item The natural homomorphism $K_n ^M (A) \to \widehat{K}_n ^M (A)$ is surjective.
\item The above is the identity map in the following cases:
\begin{enumerate}
\item When $n=1$, for all local rings $A$.
\item When $A$ is a field, for all $n \geq 1$.
\item There exists a universal integer $M_n>1$ (call it the $n$-th Kerz number), such that when the residue field of $A$ has the cardinality $> M_n$. In particular, when $A$ is a $k$-algebra over an infinite field $k$.
\end{enumerate}
\item For each finite \'etale homomorphism $A \to B$, there exists the norm map ${\rm N}_{B/A} : \widehat{K}^M _n (B) \to \widehat{K}^M_n (A)$, and the norm maps are subject to a few properties including the transitivity.
\item Gersten conjecture holds for $\widehat{K}_n ^M$, i.e. when $A$ is a local domain and $F={\rm Frac} (A)$, the natural homomorphism $\widehat{K}^M_n (A) \to \widehat{K}^M_n (F)$ is injective.
\end{enumerate}

\end{thm}

We have the following version of Lemma \ref{lem:KS} for $\widehat{K}_n ^M (k_{m+1}, (t))$:

\begin{lem}\label{lem:KS hat}
Suppose $1 \leq r \leq m+1$. Then the group $\widehat{K}_n ^M (k_{m+1}, (t^r))$ is generated by the symbols $\{ a_1, a_2, \cdots, a_n \}$, where $a_i \in k_{m+1} ^{\times}$ and $a_1 \equiv 1 \mod t^r$. 
\end{lem}
\begin{proof}
This immediately follows from Lemma \ref{lem:KS} by the anti-symmetricity and the surjectivity of $K_n ^M (k_{m+1} ) \to \widehat{K}_n ^M (k_{m+1})$. 
\end{proof}

We leave the following easy but important point:

\begin{lem}\label{lem:higher rel Milnor K}
Let $m, n, r \geq 1$ be integers, and suppose that $1 \leq r \leq m+1$. 

Let $\mathfrak{R}: \widehat{K} ^M_n (k_{m+1}, (t)) \to \widehat{K}^M _n (k_{r}, (t))$ be the canonical homomorphism induced from the surjection $k_{m+1} \to k_r$. Then we have a canonical isomorphism
$$
 \ker \mathfrak{R}  \simeq \widehat{K}_n ^M (k_{m+1}, (t^r)).
 $$
\end{lem}

\begin{proof}
This follows immediately from the snake lemma applied to the following commutative diagram with the exact rows
$$
\xymatrix{
0 \ar[r] & \ker \mathfrak{R} \ar[r] \ar[d]^{i_1} & \widehat{K}^M_n (k_{m+1}, (t)) \ar[r] ^{\mathfrak{R}} \ar[d]^{i_2} & \widehat{K}_n ^M (k_r, (t)) \ar[r] \ar[d]^{i_3} & 0 \\
0 \ar[r] & \widehat{K}_n ^M (k_{m+1}, (t^r)) \ar[r] & \widehat{K}_n ^M (k_{m+1}) \ar[r] & \widehat{K}_n ^M (k_r) \ar[r] & 0}
$$
because $\ker (i_2) = \ker (i_3) = 0$ and ${\rm coker} (i_2) = {\rm coker} (i_3) = \widehat{K}_n ^M (k)$. 
\end{proof}

\begin{remk}
Later in Theorem \ref{thm:local main 2}, we will show that $\widehat{K}_n ^M (k_{m+1} , (t^r))$ for $1 \leq r \leq m+1$ is isomorphic to a cycle class group $ \CH^n_{{\rm v}, \geq r } (X/ (m+1), n),$ to be defined in Definition \ref{defn:M_v r=1}. Assuming a result $\widehat{K}^M (k_{m+1}, (t)) \simeq \mathbb{W}_m\Omega_k ^{n-1}$ proven in a following-up paper \cite{Park vanishing}, where the group on the right is the big de Rham-Witt forms of Hesselholt-Madsen \cite{HeMa}, Lemma \ref{lem:higher rel Milnor K} offers a description of the kernels of the restriction maps 
$$
\mathfrak{R}: \mathbb{W}_m \Omega_k ^{n-1} \to \mathbb{W}_{m'} \Omega_k ^{n-1}
$$
for $1 \leq m' \leq m$ in terms of algebraic cycles. Such a cycle-theoretic description of $\ker \mathfrak{R}$ is new, and it is not immediate from the cycle-theoretic description of $\mathbb{W}_m \Omega_k ^{n-1}$ as the additive higher Chow group $\TH^n (k, n;m)$ in K. R\"ulling \cite{R}.
\qed
\end{remk}

\subsection{Higher Chow cycles}\label{sec:HC}

Recall the definition of the cubical version of the higher Chow cycles together with some relevant notations we need. The original simplicial version was defined by S. Bloch \cite{Bloch HC}. (For a bit of general discussions on the relation between the simplicial and the cubical versions, one can also read the introduction of \cite{Park localization}). Here, we suppose that $X$ is an equidimensional scheme of finite Krull dimension.

\medskip

Let $\overline{\square}:= \mathbb{P}_k ^1$ and $\square = \overline{\square}  \setminus \{ 1 \}$. For $n \geq 1$, we let $\square^n$ (resp. $\overline{\square}^n$) be the $n$-fold self-fiber product of $\square$ (resp. $\overline{\square}$) over $k$, while we let $\square^0=\overline{\square}^0:= \Spec (k)$. We use $(y_1, \cdots, y_n) \in \square^n$ (resp. $\in \overline{\square}^n$) for the coordinates.

Let $X$ be an equidimensional $k$-scheme. Write $\square_X ^n:= X \times \square_k ^n$ (resp. $\overline{\square}_X ^n:= X \times \overline{\square} ^n$). If $X= \Spec (A)$, then we often write $\square_A^n$ (resp. $\overline{\square}_A^n$). 

\medskip

When $Z \subset X \times \square_k ^n$ is a closed subscheme, we denote its closure in $X \times \overline{\square}_k ^n$ by $\overline{Z}$.

Suppose $n \geq 1$. Using $\{ 0, \infty \} \subset \square_k $, we define faces (resp. extended faces): a \emph{face} of $\square_X ^n$ (resp. an \emph{extended face} of $\overline{\square}_X ^n$) is a closed subscheme defined by a finite set of equations of the form $\{ y_{i_1} = \epsilon_1, \cdots, y_{i_s} = \epsilon _s\}$, where $1 \leq i_1 < i_2 < \cdots< i_s \leq n$ with $\epsilon_j \in \{ 0, \infty\}$. The case $s=0$ is allowed and it means empty set of equations, in which case the corresponding face (resp. extended face) is the entire space $\square_X^n$ (resp. $\overline{\square}_X ^n$). When $s=1$, we have the codimension $1$ faces (resp. extended faces), and the one given by $\{y_i = \epsilon\}$ is denoted by $F_{i} ^{\epsilon}$ or $F_{i, X} ^{\epsilon}$ (resp. $\overline{F}_i ^{\epsilon}$ or $\overline{F}_{i, X} ^{\epsilon}$). 

\begin{remk}
The extended faces are not needed for higher Chow groups, but they will be crucial in \S \ref{sec:our cycle} and beyond, especially in defining the extended general position condition, and the extended special fiber condition, respectively, in \S \ref{sec:SF* condition}. This is one of new geometric ingredients exploited in this article.
\qed
\end{remk}

\begin{defn}[S. Bloch]\label{defn:HC}
Let $X$ be an equidimensional $k$-scheme. Let $ \un{z}^q (X, 0):= z^q (X)$, the group of codimension $q$-cycles on $X$. 

\begin{enumerate}
\item (GP) For $n \geq 1$, let $\un{z}^q (X, n)$ be the subgroup of $z^q (\square_X ^n)$ consisting of the irreducible closed subschemes on $\square_X ^n$ that intersect all faces of $\square_X^n$ properly.

The acronym (GP) stands for ``general position", and the members of $\un{z}^q (X, n)$ are called higher Chow cycles.

\item The subgroup $\un{z}^q (X, n)_{\rm degn} \subset \un{z}^q (X, n)$ of degenerate cycles is generated by the cycles obtained as pull-backs of cycles via various projections $\square_X ^n \to \square_X ^i$, $0 \leq i < n$, defined by omitting some coordinates. We let
$$
 z^q (X, n):= \frac{ \un{z}^q (X, n)}{ \un{z} ^q (X, n)_{\rm degn}}.
$$

\item For each face $F_{i, X} ^{\epsilon}$ of codimension $1$, taking the intersection with the closed immersion $\iota_{i} ^{\epsilon} : \square_X ^{n-1}  \hookrightarrow \square_X ^n$ of the face $F_{i, X} ^{\epsilon}$ induces a face homomorphism 
$$
 \partial_i ^{\epsilon} : z^q (X, n) \to z^q (X, n-1).
 $$
The maps $\{ \partial _i ^{\epsilon} \}$ satisfy the cubical identities, giving a cubical abelian group $\{ \un{n} \mapsto z^q (X, n) \}$. If we let $\partial:= \sum_{i=1} ^n (-1)^i (\partial_i ^{\infty} - \partial_i ^0)$, we check that $\partial \circ \partial = 0$ using the cubical formalism. Thus $(z^q (X, \bullet), \partial)$ is a complex of abelian groups. Its $n$-th homology is denoted by $\CH^q (X, n)$ and is called the $n$-th higher Chow group of $X$ of codimension $q$.\qed
\end{enumerate}
\end{defn}

We use the following result by Nesterenko-Suslin \cite{NS} and B. Totaro \cite{Totaro}:
\begin{thm}\label{thm:NST}
Let $k$ be a field. Then we have an isomorphism $K^M_n (k) \simeq \CH^n (k, n)$ such that the Milnor symbol $\{ a_1, \cdots, a_n \}$ with $a_i \in k^{\times}$ is mapped to the closed point on $\square^n_k$ given by the system $\{ y_1 = a_1, \cdots, y_n = a_n \}$. 
\end{thm}

\section{${\rm d}$-cycles, ${\rm v}$-cycles and mod $I^{m+1}$-equivalence}\label{sec:our cycle}

In \S \ref{sec:our cycle}, we define the cycles of our interest step-by-step, and an equivalence relation on them. We study their essential properties as well.

From \S \ref{sec:DF condition} to \S \ref{sec:vanishing}, we introduce certain conditions to impose on algebraic cycles so as to define the main objects called the ${\rm d}$-cycles and the ${\rm v}$-cycles. The following diagram summarizes  the conditions to be defined and how they are related:

$$
{\rm v}\mbox{-cycles} \ \tuborg {\rm d}\mbox{-cycles} \ \tuborg (GP)_* \Rightarrow  \ \ (GP) : \ \ \mbox{ higher Chow cycles} \\
											(SF)_* \Rightarrow \tuborg (DO) \\ (SF) \sluttuborg
											 \sluttuborg \\
							\mbox{pre-vanishing} \ \ \ \   \Rightarrow \ \ \ \   (SF) \sluttuborg
$$

In \S \ref{sec:higher vanishing} and \S \ref{sec:mod t^{m+1}}, we discuss some additional notions, such as the mod $I^{m+1}$-equivalence, and several properties.

\subsection{Dominance / flatness conditions $(DO)$ / $(DF)$}\label{sec:DF condition}
The cycle groups of our primary interest to be defined later in \S \ref{sec:dsf} and \S \ref{sec:vanishing} are given by certain subcomplexes of the higher Chow complexes of henselian local $k$-schemes, subject to a few extra requirements discussed from \S \ref{sec:DF condition} to \S \ref{sec:SF* condition}. We start with the following basic lemma:

\begin{lem}\label{lem:basic Milnor 1}
Let $X$ and $B$ be integral $k$-scheme of finite dimension. Let $Z \subset X \times B$ be an integral closed subscheme. Consider the morphism $Z \hookrightarrow X \times B \to X$ composed with the projection.

\begin{enumerate}
\item Suppose $X$ is regular of dimension $1$. If $Z \to X$ is dominant, then it is flat.
\item Suppose $B$ is proper over $k$. If $Z \to X$ is dominant, then it is proper and surjective.
\item Suppose $B$ is proper over $k$. Then the following are equivalent:
\begin{enumerate}
\item $Z \to X$ is affine.
\item $Z \to X$ is finite.
\item $Z \to X$ is quasi-finite.
\end{enumerate}
\end{enumerate}
\end{lem}

\begin{proof}
The proofs are all easy, but let us give sketches with references.

(1) This is elementary e.g. \cite[Ch.III Proposition 9.7, p.257]{Hartshorne}.

\medskip

(2) Since the closed immersion $Z \hookrightarrow X$ and the projection $X \times B \to X$ are both proper, the composite $Z \to X$ is also proper. Since $Z \to X$ is dominant and closed, the image of $Z$ in $X$ is all of $X$, thus $Z \to X$ is surjective. 

\medskip

(3) Since $Z \to X$ is proper, the equivalence (a) $\Leftrightarrow$ (b) is obvious. The implication (b) $\Rightarrow$ (c) is trivial.

Suppose (c). Then a proper quasi-finite morphism is finite by e.g. \cite[Ch. III Exercise 11.2, p.280]{Hartshorne} or by an application of Zariski's Main Theorem. Thus (b) and (c) are equivalent.
\end{proof}

By taking $B= \square^n$ or $\overline{\square}^n$ for $n \geq 0$ in the above, we obtain:

\begin{cor}\label{cor:proper int t=0}
Let $X$ be an integral $k$-scheme of finite dimension. Let $n \geq 0$ be an integer, and let $Z \subset X \times \square^n$ be an integral closed subscheme. Let $\overline{Z}$ be its Zariski closure in $X \times \overline{\square}^n$.

Suppose that
\begin{enumerate}
\item 
 $Z \subset X \times \square^n$ is of codimension $n$,
\item  
$\overline{Z} \to X$ is dominant and flat.
\end{enumerate}

Then $\overline{Z} \to X$ is finite and faithfully flat.
\end{cor}

\begin{proof}
Since $\overline{Z} \to X$ is dominant, by Lemma \ref{lem:basic Milnor 1}-(2), $\overline{Z} \to X$ is surjective. Since this is assumed to be flat, it is faithfully flat.

\medskip

Since $\dim \ Z = \dim \ \overline{Z} = \dim \ X $, and $\overline{Z} \to X$ is faithfully flat of finite type, it is quasi-finite of relative dimension $0$. Thus by Lemma \ref{lem:basic Milnor 1}-(3), the morphism $\overline{Z} \to X$ is finite. 
\end{proof}

\begin{remk}
If $n=0$ in Corollary \ref{cor:proper int t=0}, then that $Z \subset X$ is of codimension $0$ means precisely that $Z=X$. So, it is generally harmless to assume $n \geq 1$ to omit this trivial case.
\qed
\end{remk}

Let's define the following terminologies:

\begin{defn}\label{defn:DF}
Let $X$ be an integral $k$-scheme of finite dimension. Let $n \geq 0$ be an integer. For a closed subscheme $Z \subset X \times \square^n$, we let $\overline{Z}$ be its Zariski closure in $X \times \overline{\square}^n$. 
\begin{enumerate}
\item We say that $Z$ satisfies the property $(DO)$ if for each face $F \subset \square^n$ such that $Z \cap (X \times F ) \not = \emptyset$, each component $W \subset Z \cap (X \times F)$ is dominant over $X$. 

\item If each $\overline{W} \to X$ is dominant and flat, then we say $Z$ satisfies the property $(DF)$.

\item If $Z$ is a cycle on $X \times \square^n$, we say that $Z$ satisfies $(DO)$ (resp. $(DF)$) if each component of $Z$ satisfies the property $(DO)$ (resp. $(DF)$).
\qed
\end{enumerate}
\end{defn}

\begin{remk}\label{remk:0 DO}
When $\dim \ X = 0$, both of the conditions $(DO)$ and $(DF)$ hold automatically for all nonempty cycles.
\qed
\end{remk}

\begin{remk}
When $X$ is regular of dimension $1$, the condition $(DO)$ and $(DF)$ are equivalent by Lemma \ref{lem:basic Milnor 1}-(1). Here, we bothered to define $(DF)$ still because in the future when we need to work with higher dimensional schemes, the author expects that it is probably necessary to impose the flatness.
\qed
\end{remk}

\subsection{Special fiber condition $(SF)$}\label{sec:SF condition}

For various versions of Chow groups and some moving lemmas on them, one often considered a condition of proper intersection with respect to a given set $\mathcal{C}$ of subschemes in  $X$ (see e.g. \cite{KP moving}, \cite{KP Jussieu}, \cite{KP sfs}). The special fiber condition below is along this line, with respect to a given fixed closed point $p \in X$:

\begin{defn}\label{defn:SF condition}
Let $X$ be an integral $k$-scheme of finite dimension. Let $n \geq 0$ be an integer. Let $Z \subset X \times \square^n$ be an integral closed subscheme and let $\overline{Z}$ be its Zariski closure in $X \times \overline{\square}^n$. Let $p \in X$ be a given fixed closed point. 

\begin{enumerate}
\item We say that $Z$ satisfies the condition $(SF)$ with respect to $p$, if the intersection $Z \cap (\{ p \} \times F)$ is proper on $X \times \square^n$  for each face $F \subset \square^n$.

\item In case the reference to $p$ is not needed (e.g. when $X$ is local), we simply say $Z$ satisfies the condition $(SF)$.

\item If $Z$ is a cycle on $X \times \square^n$, we say that $Z$ satisfies $(SF)$, if each component satisfies the property $(SF)$.\qed
\end{enumerate}
\end{defn}

\begin{remk}\label{remk:0 SF}
When $\dim \ X = 0$ we have $\{ p \} = X$, so the condition $(SF)$ always holds.
\qed
\end{remk}

\begin{remk}\label{remk:SF n=0 silly}
Suppose $n=0$. Then $\square^0 = \Spec (k)$, and $F= \Spec (k)$ is the only face. Hence $Z \cap ( \{ p \} \times F) = Z \cap \{ p \}$. This is either empty (when $p \not \in Z$) or exactly $\{p \}$ (when $p \in Z$), in which case its codimension in $X$ is $d=\dim \ X$. So, when we let $c\geq 0$ to be the codimension of $Z$ in $X$, the intersection is proper on $X$ if and only if either $p \not \in Z$, or $c=0$, in which case $Z=X$. 

If $X$ is local, then $Z$ necessarily contains $p$. So, for $n=0$ and $X$ local, the case $Z=X$ is the only possibility to have $(SF)$.
\qed
\end{remk}

When $X$ is local, the closed subschemes satisfying $(SF)$ have the following useful properties in Lemmas \ref{lem:proper int face} and \ref{lem:codim>n}:

\begin{lem}\label{lem:proper int face}

Let $X$ be an integral local $k$-scheme of finite dimension with the unique closed point $p$. Let $n \geq 1$ be an integer. Let $Z \subset X \times \square^n$ be an integral closed subscheme, and let $\overline{Z}$ be its closure in $X \times \overline{\square}^n$.

Suppose that
\begin{enumerate}
\item $Z \subset X \times \square^n$ is of codimension $n$, and
\item $Z$ satisfies $(SF)$.
\end{enumerate}

Then for each proper face $F \subsetneq \square^n$, we have
$$
Z \cap (X \times F) = \emptyset.
$$
\end{lem}

\begin{proof}

Suppose that $Z \cap (X \times F) \not = \emptyset$ for a codimension $c \geq 1$ face $F \subset \square^n$. For each closed point $x \in Z \cap (X \times F)$, consider the composite 
$$
 \{x \} \hookrightarrow \overline{Z} \to X.
 $$

Since $\overline{Z} \to X$ is proper, this composite is also proper. So its image is a closed subset of $X$. Hence the closed point $x$ is mapped to the unique closed point $p$ of $X$, i.e. $x \in Z \cap (\{ p \} \times F)$. In particular, $Z \cap (\{ p \} \times F) \not = \emptyset$. However, the intersection of $Z$ with $\{ p \} \times F$ is proper by the condition $(SF)$ in (2). Thus together with the condition (1), the intersection $Z \cap ( \{ p \} \times F)$ is nonempty and of codimension $\geq n$ in $\{ p \} \times F$. But, the space $\{ p \} \times F$ is of dimension $n-c< n$, so it has no nonempty subset of codimension $\geq n$, which is a contradiction. Thus $Z \cap (X \times F) = \emptyset$.
\end{proof}

\begin{remk}
In Lemma \ref{lem:proper int face}, we did not consider the case $n=0$ because there is no proper face for $\square^0 = \Spec (k)$. If we want to regard the empty set $\emptyset$ as a proper face of $\square^0$, then the conclusion of the lemma still holds trivially. So we may include this case in Lemma \ref{lem:proper int face} as well, if one wishes to insist.
\qed
\end{remk}

\begin{remk}
In Lemma \ref{lem:proper int face *}, we will prove an analogue of Lemma \ref{lem:proper int face} for the condition $(SF)_*$ to be defined in \S \ref{sec:SF* condition}.
\qed
\end{remk}

\begin{lem}\label{lem:codim>n}
Let $X$ be an integral local $k$-scheme of finite dimension with the unique closed point $p$. Let $n \geq 0$ be an integer. Let $Z \subset X \times \square^n$ be an integral closed subscheme, and let $\overline{Z}$ be its closure in $X \times \overline{\square}^n$.

Suppose that $Z$ satisfies $(SF)$. Then the codimension $c$ of $Z$ in $X \times \square^n$ is $\leq n$. 
\end{lem}

\begin{proof}
It uses a similar reasoning. Suppose $c \geq n+1$. Let $x \in Z$ be a closed point. Via the proper composite morphism
$$
 \{ x \} \hookrightarrow \overline{Z} \to X,
 $$
its image in $X$ is the closed point $p$ of $X$. Hence $x \in Z \cap ( \{ p \} \times \square^n)$.

Since $Z$ satisfies $(SF)$, the intersection $Z \cap ( \{ p \} \times \square^n)$ is proper, so that it has the codimension $c  \geq n+1$ in the space $\{ p  \} \times \square^n$ of dimension $n$, which is impossible. Hence $c \leq n$.
\end{proof}

\subsection{Extended conditions $(GP)_*$ and $(SF)_*$}\label{sec:SF* condition}

The following conditions that take into account some information ``at infinity" (which is actually $1\in \overline{\square} \setminus \square$), called the conditions $(GP)_*$ and $(SF)_*$, are certain requirements similar to the conditions $(GP)$ and $(SF)$, but with respect to the Zariski closures of the cycles and the extended faces:

 \begin{defn}\label{defn:GP*}
 Let $X$ be an integral $k$-scheme of finite dimension. Let $n \geq 0$ be an integer. Let $Z \subset X \times \square^n$ be an integral closed subscheme and let $\overline{Z}$ be its Zariski closure in $X \times \overline{\square}^n$. 
 
 We say that $Z$ satisfies the condition $(GP)_*$ if for each face $F \subset \square^n$, the intersection $\overline{Z} \cap (X \times \overline{F})$ is proper on $X \times \overline{\square}^n$.
  \qed
 \end{defn}

\begin{defn}\label{defn:SF* condition}
Let $X$ be an integral $k$-scheme of finite dimension. Let $n \geq 0$ be an integer. Let $Z \subset X \times \square^n$ be an integral closed subscheme and let $\overline{Z}$ be its Zariski closure in $X \times \overline{\square}^n$. Let $p \in X$ be a given fixed closed point.

\begin{enumerate}
\item We say that $Z$ satisfies $(SF)_*$ with respect to $p$, if for each face $F \subset \square^n$, the intersection $\overline{Z} \cap ( \{ p \} \times \overline{F})$ is proper on $X \times \overline{\square}^n$. 

In particular, if $Z$ satisfies $(SF)_*$ and $\overline{Z} \cap (X \times \overline{F}) \not = \emptyset$, then for each component $\overline{W} \subset \overline{Z} \cap (X \times \overline{F})$, the intersection $\overline{W} \cap ( \{ p \} \times \overline{F})$ is proper on $X \times \overline{\square}^n$.

\item If $Z$ is a cycle on $X \times \square^n$, we say that $Z$ satisfies $(SF)_*$ if each component of $Z$ satisfies $(SF)_*$.

\item If $X$ is local, then we drop the reference to the point $p$.
\qed

\end{enumerate}
\end{defn}

\begin{remk}\label{remk:SF* SF}
Apparently, if $Z$ satisfies $(GP)_*$ (resp. $(SF)_*$), then it satisfies $(GP)$ (resp. $(SF)$).
\qed
\end{remk}

\begin{remk}\label{remk:0 SF*}
When $\dim \ X = 0$, write $X= \Spec (k)$. Here, the condition $(GP)$ and $(SF)$ are equivalent, while $(GP)_*$ and $(SF)_*$ are equivalent by definition. Let's inspect the condition $(SF)_*$ for an integral closed cycle $Z \subset \square^n$. This means that the intersection $\overline{Z} \cap \overline{F}$
is proper on $\overline{\square}^n$ for each face $F \subset \square^n$. 

\medskip

Suppose $Z$ is a closed point in $\square^n$. Then we have $Z= \overline{Z}$ so that $\overline{Z} \cap \overline{F} = Z \cap F$. Hence the above condition $(SF)_*$ for $Z$ is equivalent to that $Z \cap F$ is proper, i.e. the conditions $(GP)$, $(SF)$, $(GP)_*$ and $(SF)_*$ are all equivalent.

\medskip

This time suppose that $\dim \ Z  = 1$ and it satisfies the condition $(GP)$. 
When $F$ is a codimension $\geq 2$ face in $\square^n$, then the intersection $Z \cap F$ is empty by $(GP)$, but it is possible to have $\overline{Z} \cap \overline{F}\not = \emptyset$, in which case the intersection may not be proper. In particular, the conditions $(GP)_*$ and $(SF)_*$ do not hold automatically even if $\dim \ X = 0$ in general. 

This situation is a bit different from those of the conditions $(DO)$, $(DF)$, and $(SF)$, all of which hold automatically when $\dim \ X = 0$. See  Remarks \ref{remk:0 DO} and \ref{remk:0 SF}.
\qed
\end{remk}

\begin{remk}\label{remk:SF* n=0 silly}
In case $n=0$, we have $\square^0 = \overline{\square}^0 = \Spec (k)$ and $Z = \overline{Z}$. This is similar to Remark \ref{remk:SF n=0 silly}. We have $\overline{Z} \cap (\{ p \} \times \overline{\square}^0) = \overline{Z} \cap \{ p \}$. This is either empty (when $p \not \in {Z}$) or exactly $\{ p \}$ (when $p \in {Z}$), whose codimension in $X$ is $d= \dim \ X$. 

So, as in Remark \ref{remk:SF n=0 silly}, when we let $c\geq 0$ be the codimension of $Z$ in $X$, the intersection is proper on $X$ if and only if either $p \not \in Z$, or $c=0$, in which case $Z=X$. 

If $X$ is local (which we will later assume), then $Z$ necessarily contains $p$. So, for $n=0$ and $X$ local, the case $Z=X$ is the only possibility to have the property $(SF)_*$
\qed
\end{remk}

We have the following analogue of Lemma \ref{lem:proper int face} for the condition $(SF)_*$:

\begin{lem}\label{lem:proper int face *}
Let $X$ be an integral local $k$-scheme of finite dimension with the unique closed point $p$. Let $n \geq 1$ be an integer. Let $Z \subset X \times \square^n$ be an integral closed subscheme, and let $\overline{Z}$ be its closure in $X \times \overline{\square}^n$.

Suppose that
\begin{enumerate}
\item $Z \subset X \times \square^n$ is of codimension $n$, and
\item $Z$ satisfies $(SF)_*$.
\end{enumerate}
Then for each proper face $F \subsetneq \square^n$, we have
$$
\overline{Z} \cap (X \times \overline{F}) = \emptyset.
$$
\end{lem}

\begin{proof}
One notes that $\dim \ Z = \dim \ \overline{Z}$ so that the codimension of $\overline{Z}$ in $X \times \overline{\square}^n$ is also $n$. The rest of the proof is identical to that of Lemma \ref{lem:proper int face}, except that we replace $Z$ and $F$ by $\overline{Z}$ and $\overline{F}$ respectively, and use $(SF)_*$ instead of $(SF)$. We shrink details.
\end{proof}

We also have the following analogue of Lemma \ref{lem:codim>n} for the condition $(SF)_*$:

\begin{lem}\label{lem:codim>*}
Let $X$ be an integral local $k$-scheme of finite dimension with the unique closed point $p$. Let $n \geq 0$ be an integer. Let $Z \subset X \times \square^n$ be an integral closed subscheme, and let $\overline{Z}$ be its closure in $X \times \overline{\square}^n$.

Suppose $Z$ satisfies $(SF)_*$. Then the codimension $c$ of $Z$ in $X \times \square^n$ is $\leq n$.
\end{lem}

\begin{proof}
The condition $(SF)_*$ for $Z$ implies the condition $(SF)$ for $Z$. Thus by Lemma \ref{lem:codim>n} we deduce the lemma.
\end{proof}

As said in Remark \ref{remk:SF* SF}, the condition $(SF)_*$ implies the condition $(SF)$. If $\dim \ X \leq 1$, we note that $(SF)_*$ implies $(DO)$ as well:

\begin{lem}\label{lem:SF DF}
Let $X$ be an integral local $k$-scheme of dimension $ \leq 1$ with the unique closed point $p $. Let $n \geq 0$ be an integer. Let $Z \subset X \times \square^n$ be a nonempty integral closed subscheme and let $\overline{Z}$ be the Zariski closure in $X \times \overline{\square}^n$. 

\begin{enumerate}
\item If $Z$ satisfies $(SF)_*$, then $Z$ satisfies $(DO)$.
\item In addition, if $X$ is regular, then $(SF)_*$ for $Z$ implies $(DF)$ for $Z$.
\end{enumerate}
\end{lem}

\begin{proof}
If $\dim \ X = 0$, both (1) and (2) hold trivially (see Remark \ref{remk:0 DO}). So, we suppose that $\dim \ X = 1$. 

\medskip

(1) When $n=0$, the condition $(SF)_*$ holds only when $Z = \overline{Z} = X$ (see Remark \ref{remk:SF* n=0 silly}). So the map $\overline{Z} \to X$ is the identity map, and $(DO)$ holds trivially.

\medskip

Suppose $n \geq 1$. We prove a statement that is slightly stronger than $(DO)$:  for each face $F \subset \square^n$ such that $\overline{Z} \cap (X \times \overline{F}) \not = \emptyset$,  and each component $\overline{W} \subset \overline{Z} \cap (X \times \overline{F})$, the proper morphism $\overline{W} \to X$ is dominant. 

If $\overline{W} \to X$ is not dominant, then the image of the proper morphism is a proper closed subset of $X$, so it is the unique closed point $p\in X$. Hence $\overline{W} \subset \{ p \} \times \overline{F}$. This contradicts the condition $(SF)_*$ that the intersection $\overline{W} \cap ( \{ p \} \times \overline{F})$ is proper on $X \times \overline{\square}^n$. Thus $\overline{W} \to X$ must be dominant. Hence $W \to X$ is also dominant. Hence $Z$ satisfies $(DO)$.

\medskip

(2) When $X$ is regular of dimension $1$, by Lemma \ref{lem:basic Milnor 1}-(1), a dominant morphism $\overline{W} \to X$ is flat. Thus $(SF)_*$ for $Z$ implies $(DF)$ for $Z$.
\end{proof}

\begin{cor}\label{cor:SF* DO SF}

Let $X$ be an integral local $k$-scheme of dimension $\leq 1$ with the unique closed point $p$. Let $n \geq 0$ be an integer. Let $Z \subset X \times \square^n$ be an integral closed subscheme.

Suppose $Z$ satisfies $(SF)_*$. Then $Z$ satisfies both $(DO)$ and $(SF)$.
\end{cor}

\begin{proof}
Certainly $(SF)_* \Rightarrow (SF)$, while $(SF)_* \Rightarrow (DO)$ by Lemma \ref{lem:SF DF}.
\end{proof}

One may ask whether the converse of Corollary \ref{cor:SF* DO SF} holds. This is false in general even when $\dim \ X = 0$ by Remark \ref{remk:0 SF*}. There, we saw that for the $0$-cycles $(SF)_*$ does hold, but for a higher dimensional cycles, the properties $(DO)$ and $(SF)$ hold automatically, while $(SF)_*$ may not always hold. Here is a partial converse of Lemma \ref{lem:SF DF} and Corollary \ref{cor:SF* DO SF}:

\begin{lem} \label{lem:SF* DF}
Let $X$ is an integral local $k$-scheme of dimension $\leq 1$, with the unique closed point $p $. Let $n \geq 1$ be an integer. Let $Z \subset X \times \square^n$ be an integral closed subscheme, and let $\overline{Z} \subset X \times \overline{\square}^n$ be its Zariski closure.

If $Z$ satisfies $(DO)$, then $Z$ satisfies part of the $(SF)_*$ when the face $F= \square^n$.
\end{lem}

\begin{proof}
If $\dim \ X =0$, then $X= \{p \}$, so the condition of $(SF)_*$ for $F= \square^n$ holds trivially. So, suppose that $\dim \ X = 1$.

If $F= \square^n$, then $\overline{Z} \cap (X \times \overline{\square}^n ) = \overline{Z} \not = \emptyset$. Here, we want to show that the intersection $\overline{Z} \cap  ( \{ p \} \times \overline{\square}^n)$ is proper on $X \times \overline{\square}^n$.

Since $\{ p \} \times \overline{\square}^n$ is an integral effective divisor of $X \times \overline{\square}^n$, it is enough to check that $\overline{Z} \not \subset \{ p \} \times \overline{\square}^n$.

Toward contradiction, suppose that $\overline{Z} \subset \{ p \} \times \overline{\square}^n$. Then the image of the projection $\overline{Z} \to X$ is concentrated at the closed point $\{ p \}$. This contradicts the assumption $(DO)$ that $Z \to X$ is dominant. 
\end{proof}

\subsection{The ${\rm d}$-cycles over $X$}\label{sec:dsf}

Using the discussions in \S \ref{sec:DF condition}, \S \ref{sec:SF condition}, and \S \ref{sec:SF* condition}, 
we define the first main objects of the article:

\begin{defn}\label{defn:dsf}
Let $X$ be an integral local $k$-scheme of dimension $ \leq 1$ and let $n,q \geq 0$ be integers.

\begin{enumerate}
\item A cycle $Z \in z^q (X, n)$ is called a \emph{dominant cycle} or a ${\rm d}$\emph{-cycle} if each component of $Z$ satisfies the conditions $(GP)_*$ and $(SF)_*$. 

By Corollary \ref{cor:SF* DO SF}, such $Z$ satisfies both $(DO)$ and $(SF)$. In particular each component of $Z$ is dominant over $X$,  thus the name a ``dominant" cycle.

\item Let $z^q _{{\rm d}} (X, n) \subset z^q (X, n)$ be the subgroup generated by the ${\rm d}$-cycles in $z^q (X, n)$.

When $n=0$, in the light of the condition $(DO)$, we note that $z^q_{{\rm d}} (X, 0)=0$ if $q > 0$, while $z^0 _{{\rm d}} (X, 0)= \mathbb{Z} \cdot [ X]$, the free abelian group of rank $1$ generated by the symbol $X$.

\item We see that the boundary maps $\partial_{i} ^{\epsilon} : z^q (X, n+1) \to z^q (X, n)$ for $1 \leq i \leq n+1$ and $\epsilon \in \{ 0, \infty \}$ induce the corresponding boundary maps
$$
\partial_i ^{\epsilon}: z^q _{{\rm d}} (X, n+1) \to z^q _{{\rm d}} (X, n)
$$
so that $z^q_{{\rm d}} (X, \bullet)$ forms a subcomplex of $z^q (X, \bullet)$ with respect to the boundary operator $\partial:= \sum_{i=1}^{n+1} (-1)^i (\partial_i ^{\infty} - \partial_i ^0)$. We call it the (higher) Chow complex of ${\rm d}$-cycles, and we define the (higher) Chow group $\CH^q_{{\rm d}} (X, n)$ of ${\rm d}$-cycles to be the $n$-th homology of $z^q_{{\rm d}} (X, \bullet)$.
\qed
\end{enumerate}
\end{defn}

\begin{remk}
Recall from Remark \ref{remk:0 SF*} that even when $X= \Spec (k)$, the conditions $(GP)_*$ and $(SF)_*$ may not always hold, except for $0$-cycles. Hence $z_{{\rm d}} ^n (\Spec (k), n) = z^n (\Spec (k), n)$, but $z_{{\rm d}} ^q (\Spec (k), n) \subset z ^q (\Spec (k), n)$ in general. 

When $\dim \ X = 1$, the inclusion $z_{{\rm d}} ^q (X, n) \subset z^q (X, n)$ is always proper, as the group $z^q (X, n)$ does have cycles that belong to the special fiber, while the group $z_{{\rm d}} ^q (X, n) $ does not.
\qed
\end{remk}

From Lemma \ref{lem:codim>*}, we deduce:

\begin{cor}
Let $X$ be an integral local $k$-scheme of dimension $ \leq 1$. Let $n\geq 0$ be an integer. Then for $q >n$ $z^q_{{\rm d}} (X, n) = 0$. In particular, $\CH^q _{{\rm d}} (X, n) = 0$.
\end{cor}

One aspect immediately coming from the property $(SF)_*$ is:

\begin{lem}\label{lem:specialization}
Let $X$ be an integral local $k$-scheme of dimension $1$ with the unique closed point $p$. Let $q \geq 0$ be an integer. 

The specialization at $p$ defines the morphism of complexes
$$
{\rm ev}_{p} : z^q _{ {\rm d}} (X, \bullet) \to z^q ( p, \bullet).
$$
In particular, we have the homomorphism $\CH_{{\rm d}} ^q (X, n) \to \CH^q (p, n)$.
\end{lem}

\begin{proof}
$(SF)_*$ implies $(SF)$. (See Remark \ref{remk:SF* SF}).
\end{proof}

\begin{remk}
In a future project, to give a right definition of the ${\rm d}$-cycles for $\dim \ X > 1$, the author believes that one needs to include the condition $(DO)$ or $(DF)$ as part of the axioms. It was not necessary in this article for $\dim \ X = 1$ by Lemma \ref{lem:SF DF}.
\qed
\end{remk}

\subsection{The ${\rm v}$-cycles over $X$}\label{sec:vanishing}

In \S \ref{sec:dsf}, when $\dim \ X \leq 1$, we defined the subcomplex $z^q_{{\rm d}} (X, \bullet)$ of $z^q (X, \bullet)$ consisting of the ${\rm d}$-cycles over $X$. 

In \S \ref{sec:vanishing}, we define a smaller subcomplex with an additional requirement when $\dim \ X = 1$; there is no analogue of it when $\dim \ X = 0$:

\begin{defn}\label{defn:vanishing cycle}
Let $X$ be an integral local $k$-scheme of dimension $1$ with the unique closed point $p$. Let $n, q \geq 1$ be integers.

\begin{enumerate}
\item We say that an integral cycle $Z \in z^q (X, n)$ is a \emph{pre-vanishing cycle}, if the special fiber over $p$ is empty, i.e.  $Z \cap ( \{ p \} \times \square^n) = \emptyset$. 

\item An integral cycle $Z$ is called a \emph{strict vanishing cycle}, or simply a ${\rm v}$-\emph{cycle}, if it is a pre-vanishing cycle which is also a ${\rm d}$-cycle (Definition \ref{defn:dsf}).

\item A cycle $Z \in z^q _{{\rm d}} (X, n)$ is called a \emph{strict vanishing cycle} or a ${\rm v}$-\emph{cycle}, if each component is a ${\rm v}$-\emph{cycle}. 
We let $z^q _{\rm v} (X, n) \subset z^q _{{\rm d}} (X, n)$ be the subgroup of the ${\rm v}$-cycles.
\qed
\end{enumerate}
\end{defn}

\begin{remk}
If $n=0$ or $q=0$, we do not have any nonempty pre-vanishing cycle by definition. So, it is reasonable to assume that both $n, q \geq 1$ to study pre-vanishing cycles.
\qed
\end{remk}

Here is one equivalent formulation for pre-vanishing cycles, which resembles one of the properties in the theory of additive higher Chow cycles (see, e.g. \cite{KP moving}), especially a consequence of the modulus condition:

\begin{lem}\label{lem:y_i=1}
Let $X$ be an integral local $k$-scheme of dimension $1$ with the unique closed point $p$. Let $n, q  \geq 1$ be integers. Let $Z \in z^q (X, n)$ be an integral cycle and let $\overline{Z} \subset X \times \overline{\square}^n$ be its Zariski closure.

Then $Z$ is a pre-vanishing cycle if and only if we have
\begin{equation}\label{eqn:y_i=1 1}
\overline{Z} \cap (\{ p \} \times \overline{\square}^n ) \subset \bigcup _{i=1} ^n (  \{ p \} \times \{ y_i = 1 \}) =  \{ p \} \times \bigcup_{i=1} ^n \{ y_i = 1 \} ,
\end{equation}
where $\{ y_i = 1 \} \subset \overline{\square}^n$ denotes the divisor defined by the equation $y_i = 1$.
\end{lem}

\begin{proof}
We note that $Z \cap ( \{ p \} \times \square ^n) = \emptyset$ $\Leftrightarrow$ $\overline{Z} \cap ( \{ p \} \times \overline{\square}^n) \subset \{p \} \times ( \overline{\square}^n \setminus \square ^n)$. The lemma now follows because $ \overline{\square}^n \setminus \square^n=  \bigcup _{i=1}^n \{ y_i = 1 \}$.
\end{proof}

\begin{defn}\label{defn:type i}
Let $Z \in z^q _{{\rm v}} (X, n)$ be an integral ${\rm v}$-cycle.

For $1 \leq i \leq n$, we say that $y_i$ is a vanishing coordinate for $Z$ if there exists a point $x \in \overline{Z} \cap (\{ p \} \times \overline{\square}^n)$ such that $x \in \{ p \} \times \{ y_i = 1 \}$. There could be more than one vanishing coordinate.
\qed
\end{defn}

\begin{remk}\label{remk:type i}
By definition an integral ${\rm v}$-cycle is a ${\rm d}$-cycle which is a pre-vanishing cycle. In particular, the special fiber of $Z \to X$ is empty, while the special fiber of $\overline{Z} \to X$ is nonempty, because it is a projective dominant (thus surjective) morphism by Lemma \ref{lem:SF DF}. Hence for an integral ${\rm v}$-cycle, there exists at least one vanishing coordinate for $Z$.

This aspect distinguishes ${\rm v}$-cycles from pre-vanishing cycles for which there may be no vanishing coordinate.
\qed
\end{remk}

The property of being pre-vanishing is stronger than $(SF)$:

\begin{lem}\label{lem:empty face}
Let $X$ be an integral local $k$-scheme of dimension $1$ with the unique closed point $p$. Let $n, q \geq 1$ be integers. Let $Z \in z^q (X, n)$ be an integral pre-vanishing cycle.

Then for each face $F \subset \square^n$, we have
$$ Z \cap (\{ p \} \times F) = \emptyset.$$
In particular, it satisfies $(SF)$, and each component of $Z \cap (X \times F)$ is a pre-vanishing cycle on $X \times F \simeq X \times \square^d$, where $d= \dim \ F$.
\end{lem}

\begin{proof}
$Z \cap ( \{ p \} \times F) \subset Z \cap ( \{ p \} \times \square^n) = \emptyset$.
\end{proof}

\begin{cor}
Let $X$ be an integral local $k$-scheme of dimension $1$ with the unique closed point $p$. Let $n, q \geq 1$ be integers. 

Let $W \in z^q _{{\rm v}} (X, n+1)$ be an integral ${\rm v}$-cycle. Then for each $1 \leq i \leq n+1$ and $\epsilon \in \{ 0, \infty\}$, we have $\partial _i ^{\epsilon} W \in z^q _{{\rm v}} (X, n)$.

In particular, $z^q _{{\rm v}} (X, \bullet)$ forms a subcomplex of $z^q _{{\rm d}} (X, \bullet)$.
\end{cor}

\begin{proof}
The property $(SF)_*$ and being pre-vanishing are both respected by taking faces by definition, so the corollary holds. 
\end{proof}

\begin{defn}
Let $X$ be an integral local $k$-scheme of dimension $1$ with the unique closed point $p$. 

The $n$-th homology of the complex $z^q _{{\rm v}} (X, \bullet)$ of the ${\rm v}$-cycles over $X$ is denoted by $\CH^q _{{\rm v}}(X, n)= {\rm H}_n (z^q _{{\rm v}} (X, \bullet))$, and called the (higher) Chow group of the strict vanishing cycles, or ${\rm v}$-cycles.
\qed
\end{defn}

\begin{remk} 
The natural inclusion $z^q _{{\rm v}} (X, \bullet) \hookrightarrow z^q_{{\rm d}} (X, \bullet)$
induces the homomorphism of groups
\begin{equation}\label{eqn:v to d}
\CH^q_{{\rm v}} (X, n) \to \CH^q _{{\rm d}} (X, n)
\end{equation}
for each $n \geq 0$. This is not in general injective, but in certain special cases, it is. See Corollary \ref{cor:ses final}.
\qed
\end{remk}

We introduce the following related notion:

\begin{defn}
Let $X$ be an integral local $k$-scheme of dimension $1$ with the unique closed point $p$.
A cycle class, or a cycle that represents a cycle class in $\CH_{{\rm d}} ^q (X, n)$ is called a \emph{vanishing cycle} if it belongs to the kernel of the specialization map of Lemma \ref{lem:specialization}
$$
\CH_{{\rm d}} ^q (X, n) \to \CH^q (p, n).
$$

When $q=n$, we will see that this is equivalent to saying that it belongs to the image of \eqref{eqn:v to d}. See Corollary \ref{cor:ses final}. There exists an example of a vanishing cycle that is not itself a strict vanishing cycle. See Example \ref{exm:non-strict van}.
\qed
\end{defn}

\begin{exm}\label{exm:good eg}
Here are important examples of ${\rm v}$-cycles. Let $X= \Spec (k[[t]])$ with the unique closed point $p$ given by $t=0$. Let $m \geq 1$ and $n \geq 2$ be integers. Let $a_1 \in k[[t]]\setminus \{ 0 \}$ and $b_2, \cdots, b_n \in k[[t]]^{\times}$. Let $\bar{a}, \bar{b}_2, \cdots, \bar{b}_n$ be their images in $k$ modulo $t$.
Consider the codimension $n$ integral cycle
$$
\Gamma :=  \left( 1 + t^m a_1, b_2, \cdots, b_n \right) \subset X \times \square ^n
$$
given by the equations $\{ y_1 = 1+ t^m a_1, y_2= b_2, \cdots, y_n = b_n\}$. One can check that this is a cycle satisfying $(GP)_*$ and $(SF)_*$, so it is a ${\rm d}$-cycle. Since it is proven more generally later in Lemma \ref{lem:graph adm 0}, here we just accept this fact and move on.

One notes that $\Gamma$ is a pre-vanishing cycle i.e. $\Gamma \cap ( \{ p \} \times \square^n) = \emptyset$, because modulo $t=0$, it becomes $ \left( 1, \bar{b}_2, \cdots, \bar{b}_n \right)  \not \in \{ p \} \times \square ^n = \square^n$, as $1 \not \in \square = \mathbb{P}^1 \setminus \{ 1 \}$. 

\medskip

Being a pre-vanishing ${\rm d}$-cycle, it is a ${\rm v}$-cycle, i.e. $\Gamma \in z_{{\rm v}} ^n (X, n)$. 
\qed
\end{exm}

 Here is an example of a pre-vanishing cycle in $z^3 (X, 3)$ that is \emph{not} a ${\rm v}$-cycle. For the rest of the article, such cycles are not needed, and this example is just for an illustration. The author would like to acknowledge that an example of this sort originates from discussions with Sinan \"Unver around the years 2017 - 2018 while working on Park-\"Unver \cite{PU Milnor}:

\begin{exm}\label{exm:bad eg}
For $X= \Spec (k[[t]])$ with the unique closed point $p\in X$ given by $t=0$, consider the codimension $3$ integral cycle
$$
\Gamma:=  (1 + t, t, 2+t) \subset  X \times \square ^3
$$
given by the equations $\{ y_1 = 1+t , \ y_2 = t, \ y_3 = 2 + t \}$. Let $\overline{\Gamma}$ be its closure in $X \times \overline{\square}^3$.

This is a pre-vanishing cycle because after reduction mod $t$, it reduces to the $k$-rational point $(1,0,2) \in \{ p \} \times \overline{\square}^3 = \overline{\square}^3$, which is not in $\square^3$, because $1 \not \in \square =\mathbb{P}^1 \setminus \{ 1 \}$, thus $\Gamma \cap (\{ p \} \times \square^3) = \emptyset$.

However, for the codimension $1$ face $F \subset \square^3$ given by $\{ y_2 = 0 \}$, we have $\overline{\Gamma} \cap ( \{ p \} \times \overline{F}) = \{ (1, 0, 2 ) \}$, which is a closed point in $\{ p \} \times \overline{\square}^3= \overline{\square}^3$. Its codimension in $X \times \overline{\square}^3$ is $4$, which is smaller than the sum $3 + 2= 5$ of the codimensions of $\overline{\Gamma}$ and $\{ p \} \times \overline{F}$ in $X \times \overline{\square}^3$. Hence the intersection $\overline{\Gamma} \cap ( \{ p \} \times \overline{F}) $ is not proper on $X \times \overline{\square}^3$, and the condition $(SF)_*$ is violated.

Hence $\Gamma$ is a pre-vanishing cycle, but not a ${\rm v}$-cycle. 
\qed
\end{exm}

\begin{exm}\label{exm:non-strict van}
There exists a cycle that represents a vanishing cycle in $\CH^n_{{\rm d}} (X, n)$, which is \emph{not} itself a strict vanishing cycle. Here is an example for $n=2$.

Let $X= \Spec (k[[t]])$. Recall that in $K_2 ^M (k[[t]])$, we have the relation
$$
 \left\{ a, b \right\} = \left\{ a + b, - \frac{b}{a} \right\}
 $$
deduced from the Steinberg relations.

Motivated by this relation, consider the closed subscheme $Z \subset \square_X ^2$ defined by
$$
\left\{ y_1 = 1 + c - \alpha t, \ \  y_2 = - \frac{c}{ 1 - \alpha t }\right\},
 $$
where $c \in k[[t]]^{\times}, c \not = -1$, $\alpha \in k^{\times}$. One can check directly that $Z \in z^2_{{\rm d}} (X, 2)$ (more generally, see Lemma \ref{lem:graph adm 0}), thus it represents a cycle class in $\CH_{{\rm d}} ^2 (X, 2)$. Modulo $t=0$, it reduces to the closed subscheme $Z'$ of $\square_k^2$ given by
$$
 \left\{ y_1= 1 + c, \ \ y_2 = -c \right\},
 $$
which is not empty. Thus $Z$ is not a strict vanishing cycle. 

However, under the isomorphism $\CH^2 (k, 2) \simeq K^M_2 (k)$ of Theorem \ref{thm:NST}, the class of $Z'$ is zero in $\CH^2 (k, 2)$ by the vanishing of the Steinberg relation $\{ 1+c, -c\} = 0$ in $K^M_2 (k)$. Thus $Z$ is a vanishing cycle in $\CH_{{\rm d}} ^2 (X, 2)$, being in the kernel of ${\rm ev}_{t=0}: \CH_{{\rm d}} ^2 (X, 2) \to \CH^2 (k, 2)$. 

\medskip

Alternatively, if we assume Theorem \ref{thm:gr_n iso summary} to be proven later, we can also deduce the equivalence
$$
W = \{ y_1 = 1 - \alpha t , y_2 = c \} \equiv Z \ \ \ \mbox{ in } \CH_{{\rm d}} ^2 (X, 2),
$$
where $W$ is apparently a strict vanishing cycle as $1- \alpha t \equiv 1 \mod t$.
\qed
\end{exm}

From Lemma \ref{lem:codim>*}, we deduce:

\begin{cor}
Let $X$ be an integral local $k$-scheme of dimension $1$ with the unique closed point $p$. Then $z^q _{{\rm v}} (X, n) = 0$ for $q > n$. In particular, $\CH^q _{{\rm v}} (X, n) = 0$ for $q > n$ as well.
\end{cor}

\begin{remk}
The cycle in Example \ref{exm:bad eg} represents one of the main villains, namely pre-vanishing but not strict vanishing cycles, that the author wanted to avoid for a long time.

 In the paper \cite{PU Milnor} with S. \"Unver, we were able to avoid such cycles by imposing a rather strong condition, namely that the morphism $Z \to X$ is finite surjective. However it came at the cost of being a bit too restrictive in that the inductive definition of the off-Milnor range cycles given there appeared somewhat unnatural to the author. 
 
 In this article, our cycles in the Milnor range satisfy the property that the closure $\overline{Z} \to X$ is finite (see Corollary \ref{cor:proper int t=0}), while we do not suppose that $Z \to X$ itself is finite.
 
 On the other hand, in the paper \cite{Park MZ}, again as an attempt to avoid cycles like $\Gamma$ in Example \ref{exm:bad eg}, we used a formal scheme and a restricted formal series model, which indeed did the job. However in late 2023, Junnosuke Koizumi of U. of Tokyo (now at RIKEN, Japan) informed the author that this approach also comes with a cost elsewhere that there seems to be a minor mistake in the graph morphism, and to repair it one can take the admissible blow-ups of the ambient formal schemes. The author hopes the corrigendum for this article \cite{Park MZ} may become available soon in a near future. 

A virtue of the present article, especially compared to the previous articles \cite{PU Milnor} and \cite{Park MZ}, is that we work with the new conditions $(GP)_*$ and $(SF)_*$, 
which offer consistent definitions for both the Milnor and the off-Milnor ranges, compatible with the face operations, while it also eliminates certain undesirable cycles, like the one in Example \ref{exm:bad eg}..
\qed
\end{remk}

\subsection{The ${\rm v}$-cycles of higher vanishing order}\label{sec:higher vanishing}

We generalize the Definition \ref{defn:vanishing cycle} as follows. 

\begin{defn}\label{defn:higher v-cycle}
Let $X= \Spec (A)$ be an integral local $k$-scheme of dimension $1$ with the unique closed point $p$. Let $I \subset A$ be the maximal ideal of $A$. Let $n, q, r \geq 1$ be integers. Let $Z \in z^q _{{\rm d}} (X, n)$ be an integral cycle.

\begin{enumerate}
\item We say that $Z$ is a strict vanishing cycle of \emph{order \emph{$ \geq r$}}, if 
$$ 
Z \times_{\Spec (A)} \Spec (A/ I^r) = \emptyset.
$$

\item We say that $Z$ is a strict vanishing cycle of \emph{order $r$}, if it is of order $\geq r$, but not of order $\geq r+1$. 
\item Let $z^q_{{\rm v}, \geq r} (X, n) \subset z^q_{{\rm d}} ( X, n)$ be the subgroup generated by strict vanishing cycles of order $\geq r$. By convention, we define $z^q_{{\rm v}, \geq 0} (X, n) := z^q_{{\rm d}} (X, n)$. Since fibre products are associative, one deduces immediately that if $Z \in z_{{\rm v}, \geq r} ^q (X, n)$, then $\partial_i ^{\epsilon} Z \in z_{{\rm v}, \geq r} ^q (X, n-1)$ for all $1 \leq i \leq n$ and $\epsilon \in \{ 0, \infty \}$. In particular, one deduces that we have a subcomplex $z_{{\rm v}, \geq r}^{q} (X, \bullet)$, and in fact, we have the obvious decreasing filtration
$$
 \cdots \subset z^q _{{\rm v}, \geq 2} (X, \bullet) \subset z^q_{{\rm v}, \geq 1} (X, \bullet) \subset z^q _{ {\rm v}, \geq 0} (X, \bullet) = z^q_{{\rm d}} (X, \bullet)
 $$
on the complex $z^q_{{\rm d}} (X, \bullet)$.
 \qed
\end{enumerate}
\end{defn}

\begin{exm}
By definition $Z$ is strict vanishing cycle in the sense of Definition \ref{defn:vanishing cycle} if and only if it is a strict vanishing cycle of order $\geq 1$.
\qed
\end{exm}

\begin{exm}
Let $X= \Spec (k[[t]])$. Consider the cycle $Z \subset X \times \square^3$ given by 
$$\left\{ y_1 = 1 - \frac{t^r}{a}, y_2 = b_1, y_3 = b_2 \right\},$$
where $a, b_1, b_2 \in k[[t]]^{\times}$. This $Z$ is a strict vanishing cycle of order $r$.
\qed
\end{exm}

\subsection{The mod $I^{m+1}$-equivalence}\label{sec:mod t^{m+1}}

We now suppose that $X= \Spec (A)$ is an integral \emph{henselian} local $k$-scheme of dimension $1$ with the unique closed point $p$. Let $I \subset A$ be the maximal ideal. For each integer $m \geq 1$, let $X_{m+1}$ be the closed subscheme defined by the ideal $I^{m+1}$. 

Under the above assumptions and the notations, we define two notions of equivalence relations on cycles, and in the Milnor range we prove that they coincide. 

First, we define the following equivalence relation on the ${\rm d}$-cycles $z_{{\rm d}} ^q (X, n)$ in the Milnor range, originating from an analogous notion in \cite{PU Milnor}, with some subtle differences: 

\begin{defn}[{cf. \cite[Definition 2.3.1]{PU Milnor}}]\label{defn:mod t^{m+1}}
Let $m, n \geq 1$ be integers. 

Let $Z_1, Z_2 \in z^q _{{\rm d}} (X, n)$ be two integral ${\rm d}$-cycles.
\begin{enumerate}
\item We say that $Z_1$ and $Z_2$ are \emph{naively mod $I^{m+1}$-equivalent} and write $Z_1 \sim_{I^{m+1}} Z_2$ if the Zariski closures $\overline{Z}_1, \overline{Z}_2$ in $X \times \overline{\square}^n$ satisfy the \emph{equality} of the closed subschemes of $X_{m+1} \times \overline{\square}^n$
$$
\overline{Z}_1 \times_X X_{m+1} = \overline{Z}_2 \times_X X_{m+1}.
$$
We emphasize that this notion is defined using the closures. 

In case $I= (t)$ is a principal ideal (e.g. when $X$ is regular), we may say ``mod $t^{m+1}$-equivalent" synonymously.
\item Let $\mathcal{N}^q (m+1) \subset z^q_{{\rm d}} (X, n)$ be the subgroup generated by the differences $Z_1 - Z_2$ over all mod $I^{m+1}$-equivalent pairs $(Z_1, Z_2)$ of integral cycles $Z_i \in z_{{\rm d}} ^n (X, n)$.
\qed
\end{enumerate}
\end{defn}

While the above definition is easy to give and convenient for computations, it poses a couple of technical inconveniences. For instance, it is difficult to show that this equivalence relation is preserved under flat pull-backs or finite push-forwards.

This can be overcome by considering the following definition which will be shown to be equivalent to the above in the Milnor range in Lemma \ref{lem:comparison mod}, which requires the additional assumption that $A$ is henselian. This notion is analogous to the one from \cite{Park MZ}, with some subtle differences as well.

\begin{defn}[{cf. \cite[Definition 2.4.1]{Park MZ}}]\label{defn:mod t^{m+1} 2}
Let $\mathcal{A}$ be a coherent $\mathcal{O}_{\overline{\square}^n_X}$-algebra. 
\begin{enumerate}
\item We say that $\mathcal{A}$ is $(q,n)$-\emph{admissible over $X$} if the associated cycle $[\mathcal{A}|_{\square^n_X}]$ is in $z^q_{{\rm d}} (X, n)$.
\item A pair $(\mathcal{A}_1, \mathcal{A}_2)$ of $(q,n)$-admissible coherent $\mathcal{O}_{\overline{\square}^n_X}$-algebras is said to be \emph{mod $I^{m+1}$-equivalent} if there is an isomorphism
$$\mathcal{A}_1 \otimes_{\mathcal{O}_{\overline{\square}^n_{X}}} \mathcal{O}_{\overline{\square}^n_{X_{m+1}}}  \simeq \mathcal{A}_2 \otimes_{\mathcal{O}_{\overline{\square}^n_{X}}} \mathcal{O}_{\overline{\square}^n_{X_{m+1}}}$$
of $\mathcal{O}_{\overline{\square}^n_{X_{m+1}}}$-algebras.
\item Let $\mathcal{M}^q (m+1, n) \subset z^q_{{\rm d}} (X, n)$ be the subgroup generated by the cycles of the form 
$[\mathcal{A}_1|_{\square_X^n}] - [\mathcal{A}_2|_{\square_X^n}] $ over all pairs $(\mathcal{A}_1, \mathcal{A}_2)$ as in (2).
\qed
\end{enumerate}
\end{defn}

Here, when $q=n$, we have two subgroups $\mathcal{N}^n (m+1)$ and $\mathcal{M}^n (m+1, n)$ of $z^n_{{\rm d}} (X, n)$ that define equivalence relations on the group $z^n_{{\rm d}} (X, n)$. We show in Lemma \ref{lem:comparison mod} that they coincide. This requires our additional assumption ``henselian" made from \S \ref{sec:mod t^{m+1}}.

Recall the following from \cite[Proposition (18.5.9), (18.5.10), p.130]{EGA4-4}, or \cite[Lemma 04GG]{stacks}): 

\begin{lem}\label{lem:henselian}
Let $R$ be a henselian local ring and let $R \hookrightarrow B$ be a finite extension of rings.

Then $B$ is a finite direct product of henselian local rings. Furthermore, there is a bijection between the factors and the maximal ideals of $B$. In particular, if $B$ is an integral domain in addition to the above assumptions, then $B$ is a henselian local domain with a unique maximal ideal.
\end{lem}

\begin{lem}[{cf. \cite[Lemma 4.3.1]{Park MZ}}]\label{lem:comparison mod}
We have the equality
$$
\mathcal{M}^n (m+1, n) = \mathcal{N}^n (m+1)
$$
of the subgroups of $z^n_{{\rm d}} (X, n).$ In particular, on the group $z_{{\rm d}} ^n (X, n)$ of ${\rm d}$-cycles in the Milnor range, the naive mod $I^{m+1}$-equivalence coincides with the mod $I^{m+1}$-equivalence.
\end{lem}

\begin{proof}
While the idea of the proof is essentially identical to the one in \emph{loc.cit.}, since our objects are not identical to the ones there, we give details for self-containedness. 

Each generator $[Z_1] - [Z_2] \in \mathcal{N}^n (m+1)$ is given by two integral cycles $Z_1, Z_2 \in z_{{\rm d}} ^n (X, n)$ such that
$$
\overline{Z}_1 \times_X X_{m+1} = \overline{Z}_2 \times_X X_{m+1}
$$ 
as closed subschemes in $\overline{\square}^n_{X_{m+1}}$. This is equivalent to say that we have an isomorphism of $\mathcal{O}_{{ \overline{\square}^n_{X_{m+1}}}}$-algebras
$$
\mathcal{O}_{\overline{Z}_1} \otimes _{\mathcal{O}_{\overline{\square}_{X} ^n}} \mathcal{O}_{ \overline{\square}^n_{X_{m+1}}} \simeq \mathcal{O}_{\overline{Z}_1} \otimes _{\mathcal{O}_{\overline{\square}_{X} ^n}} \mathcal{O}_{ \overline{\square}^n_{X_{m+1}}}.
$$
Hence we have
$$
[Z_1] - [Z_2] = [ \mathcal{O}_{\overline{Z}_1}|_{ \square_X ^n}] -  [ \mathcal{O}_{\overline{Z}_2}|_{ \square_X ^n}]  \in \mathcal{M} ^n (m+1, n),
$$
so we deduce that $\mathcal{M}^n (m+1, n) \supset \mathcal{N}^n (m+1).$

\medskip

We now prove the opposite inclusion $\mathcal{M}^n (m+1, n) \subset \mathcal{N}^n (m+1).$ Let $(\mathcal{A}_1, \mathcal{A}_2)$ be a pair of mod $I^{m+1}$-equivalent $(n,n)$-admissible coherent $\mathcal{O}_{\overline{\square}_X^n}$-algebras. Recall that for each integral cycle $Z \in z_{{\rm d}} ^n (X, n)$, we know the morphism $\overline{Z} \to X$ is finite (Corollary \ref{cor:proper int t=0}), thus $\overline{Z}$ is affine. Hence each $\mathcal{A}_j$ is determined by the ring $B_j$ of the global sections. Furthermore, the mod $I^{m+1}$-equivalence of $\mathcal{A}_1$ and $\mathcal{A}_2$ implies the existence of an isomorphism of the Artin $k$-algebras
\begin{equation}\label{eqn:artin 1}
B_1 / I^{m+1} B_1 \simeq B_2 / I^{m+1} B_2.
\end{equation}

Since $B_j$ is finite over the henselian ring $A$, by Lemma \ref{lem:henselian}, we can write
\begin{equation}\label{eqn:henselian 0}
B_j = \prod_{i=1} ^{r_j} B_{ji}
\end{equation}
for some henselian local domains $B_{ij}$, corresponding to the maximal ideals of $B_j$. 

By the structure theorem of Artin rings (see Atiyah-Macdonald \cite[Theorem 8.7, p.90]{AM} or D. Eisenbud \cite[Corollary 2.16, p.76]{Eisenbud}), each $B_j / I^{m+1} B_j$ decomposes uniquely into a product of Artin local rings, and the isomorphism \eqref{eqn:artin 1} gives a correspondence between the Artin local factors of $B_1/ I^{m+1} B_1$ and $B_2/ I^{m+1} B_2$. 

The Artin local factors in the second step and the henselian local factors in \eqref{eqn:henselian 0} correspond naturally. 

Hence, the numbers of factors for $B_1$ and $B_2$ are equal, i.e. $r_1 = r_2 =:r$.

\medskip

Since each irreducible component in $[ \Spec (B_j)]$ has the empty intersection with the extended faces $X \times \overline{F}$ of $X \times \overline{\square}^n$ (Lemma \ref{lem:proper int face *}), we have the proper morphisms $\Spec (B_j) \to X \times \mathbb{A}^n$. We let
$$
f_{ji} : \Spec ( B_{ji}) \to X \times \mathbb{A}^n
$$
be its restriction to $\Spec (B_{ji})$. We let $\overline{Z}_{ji} := f_{ji} (\Spec (B_{ji})) \subset X \times \mathbb{A}^n$. It is closed in $X \times \overline{\square}^n$. Let $\mathcal{A}_{ji}$ be the $\mathcal{O}_{\overline{\square}^n_X}$-algebra corresponding to $B_{ji}$, so that $A_{ji} \simeq f_{ji*} \mathcal{O}_{\Spec (B_{ji})}$ and 
$$
[ \mathcal{A}_{ji} |_{\square_X ^n} ] = f_{ji*} [ \Spec (B_{ji})|_{\square_X^n}] = d_{ji} [ Z_{ji}]
$$
where $d_{ji} := \deg (f_{ji})$ and $Z_{ji}:= \overline{Z}_{ji} |_{\square_X ^n}$.

Let $C_{ji}$ be the coordinate ring of $\overline{Z}_{ji}$. This is the image of the natural homomorphism 
$$
f_{ji} ^{\sharp}: A[ y_1, \cdots, y_n] \to B_{ji},$$
and since we have the commutative diagram with the vertical isomorphism coming from \eqref{eqn:artin 1}
$$
\xymatrix{ (A/ I^{m+1}) [ y_1, \cdots, y_n ] \ar[rr]^{ f_{1i} ^{\sharp}} \ar[drr]^{ f_{2i} ^{\sharp}} & & B_{1i}/ I^{m+1} B_{1i} \ar[d] ^{\simeq} \\
 & & B_{2i} / I^{m+1} B_{2i},}
 $$
 taking the images of the horizontal maps, we deduce the induced isomorphism
 $$
 C_{1i}/ I^{m+1} C_{1i} \simeq C_{2i}/ I^{m+1} C_{2i}.
 $$
 Hence $[ Z_{1i} ] - [Z_{2i}] \in \mathcal{N}^n (m+1)$. Since $d_{ji} = [ B_{ji}/ I : C_{ji}/ I]$ and we have the induced isomorphisms $B_{1i}/ IB_{1i} \simeq B_{2i} / I B_{2i}$ and $C_{1i}/ I C_{1i} \simeq C_{2i} / I C_{2i}$, we deduce $d_{1i} = d_{2i}$, call it $d_i$. Thus 
 $$
 [\mathcal{A}_1 |_{\square_X ^n}] - [ \mathcal{A}_2 |_{\square_X ^n}] = \sum_{i=1} ^r ( [ \mathcal{A}_{1i}|_{\square_X^n} ] - [ \mathcal{A}_{2i}|_{\square_X^n}]) = \sum_{i=1} ^r d_i ( [ Z_{1i} ] - [ Z_{2i}]) \in \mathcal{N}^n (m+1),
 $$
 which shows $\mathcal{M}^n (m+1, n) \subset \mathcal{N}^n (m+1)$, and completes the proof.
\end{proof}

The group $\mathcal{M}^q (m+1, n)$ works well away from the Milnor range as well:
\begin{lem}\label{lem:mod t face}
For $1 \leq i \leq n$ and $\epsilon \in \{ 0, \infty \}$, we have
$$ 
\partial_i ^{\epsilon} \mathcal{M}^q (m+1, n) \subset \mathcal{M}^q (m+1, n-1).
$$
In particular, $\mathcal{M}^q (m+1, \bullet)$ is a subcomplex of $z^q_{{\rm d}} (X, \bullet)$.
\end{lem}

\begin{proof}
The argument is essentially equal to \cite[Lemma 2.4.4]{Park MZ}. We follow its outline with minor adjustments for our situation.

For $1 \leq i \leq n$ and $\epsilon \in \{ 0, \infty \}$, let $\iota_i ^{\epsilon} : \overline{\square}_X ^{n-1} \hookrightarrow \overline{\square}_X ^n$ be the closed immersion given by $\{ y_i = \epsilon \}$. 

When $\mathcal{A}$ is a $(q,n)$-admissible coherent $\mathcal{O}_{\overline{\square}_X ^n}$-algebra, the sheaf pull-back $(\iota_i ^{\epsilon})^* \mathcal{A}$ is a $(q, n-1)$-admissible coherent $\mathcal{O}_{\overline{\square}_X ^{n-1}}$-algebra, and it satisfies $\partial _i ^{\epsilon} [ \mathcal{A} |_{\square_X ^n}] = [ ((\iota_i ^{\epsilon})^* \mathcal{A} ) |_{\square_X ^{n-1}}]$. 

If $(\mathcal{A}_1, \mathcal{A}_2)$ is a pair of such sheaves that are mod $I^{m+1}$-equivalent, then we have an isomorphism of $\mathcal{O}_{\overline{\square}_{X_{m+1}} ^n}$-algebras
$$
\mathcal{A}_1 \otimes_{\mathcal{O}_{\overline{\square}_X ^n}} \mathcal{O}_{\overline{\square}_{X_{m+1}} ^n} \simeq  \mathcal{A}_2 \otimes_{\mathcal{O}_{\overline{\square}_X ^n}} \mathcal{O}_{\overline{\square}_{X_{m+1}} ^n}.
$$
Pulling back via $\iota_i ^{\epsilon}$, we deduce an isomorphism of $\mathcal{O}_{\overline{\square}_{X_{m+1}} ^{n-1}}$-algebras
$$
(\iota_i ^{\epsilon})^* \mathcal{A}_1 \otimes_{\mathcal{O}_{\overline{\square}_X ^{n-1} }} \mathcal{O}_{\overline{\square}_{X_{m+1}} ^{n-1}} \simeq  
(\iota_i ^{\epsilon})^* \mathcal{A}_2 \otimes_{\mathcal{O}_{\overline{\square}_X ^{n-1} }} \mathcal{O}_{\overline{\square}_{X_{m+1}} ^{n-1}}.
$$
Hence $((\iota_i ^{\epsilon})^* \mathcal{A}_1 , (\iota_i ^{\epsilon})^* \mathcal{A}_2)$ is a mod $I^{m+1}$-equivalent pair and we have
$$
\partial_i ^{\epsilon} ( [ \mathcal{A}_1 |_{\square_X^n}] - [ \mathcal{A}_2 |_{\square_X^n}] ) =  [  ((\iota_i ^{\epsilon})^* \mathcal{A}_1) |_{\square_X^{n-1}}] - [ ((\iota_i ^{\epsilon})^* \mathcal{A}_2)|_{\square_X^{n-1}}] \in \mathcal{M}^q (m+1, n-1).
$$
This proves the lemma.
\end{proof}

\begin{defn}
Let $X$ be an integral henselian local $k$-scheme of dimension $1$. Let $m, n \geq 1$ be integers. Define
$$ 
z^q_{{\rm d}} (X/ (m+1), n):= z^q_{{\rm d}} (X, n) / \mathcal{M}^q (m+1, n),
$$
and the homology of the complex is by definition
$$
\CH^q _{{\rm d}} (X/ (m+1), n) :={\rm H}_n ( z^q_{{\rm d}} (X/ (m+1), \bullet)).
$$

In case $q=n$, we may replace the group $\mathcal{M}^n (m+1, n)$ by $\mathcal{N} ^n (m+1)$ by Lemma \ref{lem:comparison mod}, so that we have
\begin{equation}\label{eqn:CHd N}
\CH^n_{{\rm d}} (X/ (m+1), n)= \frac{ z^n _{{\rm d}} (X, n)}{ \partial z^n_{{\rm d}} (X, n+1) + \mathcal{N}^n (m+1)},
\end{equation}
which we could also have been taken as a definition if one wishes to ignore the groups $\mathcal{M}^q (m+1, n)$ and just stick to the Milnor range when $q=n$ for simplicity. However, we mention that with \eqref{eqn:CHd N} as the definition, it is harder to prove the functoriality properties as in Lemmas \ref{lem:fpb} and \ref{lem:fpf}. 
\qed
\end{defn}

\begin{remk}
We do not claim that the above offers a right definition of the motivic cohomology of $X_{m+1}$ in all ranges of $q$. We will see in Theorem \ref{thm:local main 1} that when $q=n$ this group is indeed isomorphic to the motivic cohomology of Elmanto-Morrow \cite{EM} (cf. \cite{Park general}).

 For $q \not = n$, a follow-up work in progress will clarify more about some extra tasks required to describe the motivic cohomology of $X_{m+1}$ in terms of cycles. 
\qed
\end{remk}

We can induce the mod $I^{m+1}$-equivalence on the $({\rm v}, \geq r)$-cycles as well:

\begin{defn}[{cf. Definition \ref{defn:mod t^{m+1}}}]\label{defn:mod t^{m+1} v}
Let $\mathcal{N}_{{\rm v}, \geq r} ^n (m+1) \subset z_{{\rm v}, \geq r} ^n (X, n)$ be the subgroup generated by $Z_1 - Z_2$ over all pairs $(Z_1, Z_2)$ of naively mod $I^{m+1}$-equivalent integral cycles $Z_1, Z_2 \in z_{{\rm v}, \geq r} ^n (X, n)$.
\qed
\end{defn}

We have the following analogue of Definition \ref{defn:mod t^{m+1} 2} for Definition \ref{defn:mod t^{m+1} v}:

\begin{defn}\label{defn:M_v}
Let $X$ be an integral henselian local $k$-scheme of dimension $1$. Let $m, n \geq 1$ be integers. Define $ \mathcal{M}^q_{{\rm v}, \geq r} (m+1, n)$ to be the subgroup of $z_{{\rm v}, \geq r} ^q (X, n)$ generated by the cycles of the form $[ \mathcal{A}_1 |_{\square_X^n}] - [ \mathcal{A}_2 |_{\square_X^n}] \in \mathcal{M}^q _{{\rm d}} (m+1, n)$ for mod $I^{m+1}$-equivalent pairs of $\mathcal{O}_{\overline{\square}_X ^n}$-algebras $(\mathcal{A}_1, \mathcal{A}_2)$ such that each component belongs to $z^q_{{\rm v}, \geq r} (X, n)$. 

In particular, $\mathcal{M}^q _{{\rm v}, \geq r} (m+1, n) \subset \mathcal{M}^q _{{\rm d}} (m+1, n) \cap z_{{\rm v}, \geq r} ^q (X, n)$. This may not be an equality in general.
\qed
\end{defn}

We have the following analogue of Lemma \ref{lem:comparison mod} whose proof is essentially identical:
\begin{lem}\label{lem:comparison mod v}
We have the equality
$$ \mathcal{M}_{{\rm v}, \geq r} (m+1, n) = \mathcal{N}_{{\rm v}, \geq r} ^n (m+1)$$
in the group $z_{{\rm v}, \geq r}^n (X, n)$. 
\end{lem}

\begin{lem}\label{lem:M_v}
$\mathcal{M}^q _{{\rm v}, \geq r} (m+1, \bullet)$ forms a subcomplex of $z^q _{{\rm v}, \geq r} (X, \bullet)$.
\end{lem}

\begin{proof}
It is essentially a repetition of Lemma \ref{lem:mod t face}: suppose that $[ \mathcal{A}_1|_{\square_X^{n+1}}] - [ \mathcal{A}_2 |_{\square_X^{n+1}}] \in \mathcal{M}_{{\rm v}, \geq r} ^q (m+1, n+1)$ is a generator. When $\iota_i ^{\epsilon}: \overline{ \square}_X ^n \hookrightarrow \overline{\square}_X ^{n+1}$ is the closed immersion given by $y_i = \epsilon$, we have
$$
\partial_i ^{\epsilon} \left( [ \mathcal{A}_1|_{\square_X^{n+1}}] - [ \mathcal{A}_2 |_{\square_X^{n+1}}]  \right) = [ ( (\iota_i ^{\epsilon})^* \mathcal{A}_1) |_{\square_X^n}] - [  ((\iota_i ^{\epsilon})^* \mathcal{A}_2) |_{\square_X^n}] \in z_{{\rm v}, \geq r} ^q (X, n)
$$
because $\partial_i ^{\epsilon} ( z_{{\rm v}, \geq r} ^q (X, n+1)) \subset z_{{\rm v}, \geq r} ^q (X, n)$ already. Hence $\partial_i ^{\epsilon} \mathcal{M}_{{\rm v}, \geq r} ^q (m+1, n+1) \subset \mathcal{M}_{{\rm v}, \geq r} ^q (m+1., n)$. Thus, $\mathcal{M}_{{\rm v}, \geq r} ^q (m+1, \bullet)$ is a complex with the boundary map $\partial$, thus a subcomplex of $z_{{\rm v}, \geq r} ^q (X, \bullet)$. 
\end{proof}

\begin{defn}\label{defn:M_v r=1}
Define the quotient complex $z^q_{{\rm v}, \geq r} (X/ (m+1), \bullet)$ by
$$ 
z^q _{{\rm v}, \geq r} (X/ (m+1), n) := z^q _{{\rm v}, \geq r} (X, n) / \mathcal{M}^q_{{\rm v}, \geq r} (m+1, n).
$$

The homology of the complex is by definition
$$
\CH^q _{{\rm v}, \geq r} (X/ (m+1), n):= {\rm H}_n (z^q _{{\rm v}, \geq r} (X/ (m+1), \bullet)).$$

As before, in case $q=n$, we can describe $\CH^n _{{\rm v}, \geq r} (X/ (m+1), n)$ also as
$$
\CH_{{\rm v}, \geq r} ^n (X/ (m+1), n) = \frac{ z_{{\rm v}, \geq r}^n (X, n) }{ \partial z_{{\rm v}, \geq r} ^n (X, n+1) + \mathcal{N}_{{\rm v}, \geq r} ^n (m+1)}
$$
similar to \eqref{eqn:CHd N}, using the group $\mathcal{N}_{{\rm v}, \geq r} ^n (m+1)$ of Definition \ref{defn:mod t^{m+1} v} by Lemma \ref{lem:comparison mod v}. When $r=1$, we drop ``$\geq r$" and simply write ${\rm v}$.
\qed
\end{defn}

\begin{lem}\label{lem:reduction mod m}
For each integer $m \geq 1$, there exist natural homomorphisms
$$
 s_m ^*: \CH^q _{{\rm d}} (X, n) \to \CH^q _{{\rm d}} (X/ (m+1), n),
 $$
 $$
 s_m: \CH_{{\rm v}, \geq r} ^q (X, n) \to \CH_{{\rm v}, \geq r} ^q (X/ (m+1), n), \ \ \mbox{ for } 1 \leq r \leq m+1.
 $$

More generally, for each pair $m'  \geq m \geq 1$, where $m' = \infty$ is also allowed, we have natural homomorphisms
$$
s_m ^{m' *}:  \CH^q _{{\rm d}} (X/ (m'+1), n) \to \CH^q _{{\rm d}} (X/ (m+1), n), 
$$
$$
s_m ^{m' *}: \CH_{{\rm v}, \geq r} ^q (X/ (m'+1), n) \to \CH_{{\rm v}, \geq r}^q (X/ (m+1), n), \ \ \mbox{ for } 1 \leq r \leq m+1,
$$
where if $m'=\infty$, we regard $ \CH^q _{? } (X/ (m'+1), n)  =  \CH^q _{?} (X, n)$ for $?=\{{\rm d}\}, \{{\rm v}\}$, or $\{{\rm v}, \geq r\}$.
\end{lem}

\begin{proof}
By construction, we have the natural homomorphism of complexes
$$
z^q _{{\rm d}} (X/ (m'+1), \bullet) \to z^q_{{\rm d}} ( X/ (m+1), \bullet).
$$
This induces the desired homomorphisms. The other cases are similar.
\end{proof}

\subsection{Some pull-backs and push-forwards}

We discuss some cases where we can apply pull-backs and push-forwards of cycles needed for the purpose of this article: 

\begin{lem}\label{lem:fpb}
Let $X, Y$ be integral local $k$-schemes of dimension $\leq 1$ with the closed points $p_1, p_2$, respectively. Let $f: X \to Y$ be a flat surjective morphism of $k$-schemes. 
\begin{enumerate}
\item We have the induced flat pull-back morphism
$$
f^*: z_{{\rm d}} ^q (Y, \bullet) \to z_{{\rm d}} ^q (X, \bullet).
$$
\item Suppose $\dim \ X = \dim \ Y = 1$. Then the restriction of the above $f^*$ on the subcomplex $z_{{\rm v}} ^q (Y, \bullet)$ induces
$$
f^*: z_{{\rm v}} ^q (Y, \bullet) \to z_{{\rm v}} ^q (X, \bullet).
$$
\item Suppose $Y= \Spec (A)$ is an integral local $k$-scheme of dimension $1$ with the residue field $k=A/I$, and and $f: X= Y_{k'} \to Y$ is the base change map for a finite extension $k \hookrightarrow k'$ of fields. Under this we have:
$$
f^*: z_{{\rm v}, \geq r} ^q (Y, \bullet) \to z_{{\rm v}, \geq r}^q (X, \bullet).
$$
\item In addition to the assumptions of (3), suppose that $Y$ is henselian. Then we have: 
$$
f^*: z_{{\rm d}} ^q (Y/ (m+1), \bullet) \to z_{{\rm d}} ^q (X/ (m+1), \bullet),
$$
$$
f^*: z_{{\rm v}, \geq r} ^q (Y/ (m+1), \bullet) \to z_{{\rm v}, \geq r} ^q (X/ (m+1), \bullet).
$$
\end{enumerate}
\end{lem}

\begin{proof}
(1) We already have the natural flat pull-backs on higher Chow complexes 
\begin{equation}\label{eqn:flat pb 0}
f^*: z^q (Y, \bullet) \to z^q (X, \bullet).
\end{equation}
Hence we just need to check that $f^*$ maps $z^q_{{\rm d}} (Y, n)$ into $z^q _{{\rm d}} (X, n)$ for each $n \geq 0$. For this, we check that the conditions $(GP)_*$ and  $(SF)_*$ are preserved under $f^*$.

Let $Z \in z^q_{{\rm d}} (Y, n)$ be an integral cycle. So for each face $F \subset \square^n$, the intersections
$$
\overline{Z} \cap (Y \times \overline{F}) \ \ \ \mbox{ and } \ \ \ \overline{Z} \cap (\{ p_2\}  \times \overline{F})
$$
are proper on $Y \times \overline{\square}^n$. Note that $\overline{ f^{-1} (Z)} = f^{-1} (\overline{Z})$. By taking the inverse images via the flat morphism $f$, we deduce that the intersections
$$
\overline{f^{-1} (Z)} \cap (X \times \overline{F}) \ \ \ \mbox{ and } \ \ \ \overline{f^{-1} (Z)} \cap (\{ p_1 \} \times \overline{F})
$$
are proper on $X \times \overline{\square}^n$, as desired. Thus $f^* (Z)=[ f^{-1} (Z)]$ satisfies $(GP)_*$ and $(SF)_*$ so that  $f^* (Z) \in z_{{\rm d}} ^q (X, n)$. This proves (1).

\medskip

(2) We need to check that $f^* ( z_{{\rm v}} ^q (Y, n)) \subset z^q_{{\rm v}} (X, n)$. Let $Z \in z_{{\rm v}} ^q (Y, n)$ be an integral cycle. This means that the special fiber over $p_2 \in Y$ is
$$
Z_{p_2} = Z \times _Y \{ p_2 \} = \emptyset.
$$
Then taking $f^{-1}$,
$$
f^{-1} (Z) \times_X \{ p_1 \} = f^{-1} (Z) \times_{f^{-1} (Y)} f^{-1} (p_2) =  f^{-1} (Z_{p_2}) = f^{-1} (\emptyset) = \emptyset.
$$

In particular, the associated cycle $f^* (Z)= [f^{-1} (Z)]$ is a pre-vanishing cycle. Hence together with (1), we have $f^* (Z) \in z_{{\rm v}} ^q (X, n)$ as desired.

\medskip

(3) It is enough to show that $f^* (z^q _{{\rm v}, \geq r} (Y, n)) \subset z^q _{{\rm v}, \geq r} (X, n)$. Here $X= \Spec (A_{k'})$. 

Let $Z \in z^q _{{\rm v}, \geq r} (Y, n)$ be an integral cycle so that
$$
{Z} \times_Y Y_r = {Z} \times_{\Spec (A)} \Spec (A/ I^r) = \emptyset.
$$
Applying $f^{-1}$ to the above,
$$
f^{-1} ({Z}) \times _{\Spec (A_{k'})} \Spec (A_{k'}/ (I A_{k'} )^r) = f^{-1} ({Z}) \times_X X_r = \emptyset.$$
Combined with (1), this implies that $f^* (Z) = [ f^{-1}(Z)] \in z_{{\rm v}, \geq r} ^q (X, n)$.

\medskip

(4) We show that the pull-back respects the mod $I^{m+1}$-equivalence. For $Y$, write $\mathcal{M}_{{\rm d}} ^q (Y, m+1, n)$ for $\mathcal{M}_{{\rm d}} ^q (m+1, n)$, and similarly for $X$.

Let $(\mathcal{A}_1, \mathcal{A}_2)$ be a pair of mod $I^{m+1}$-equivalent $(q,n)$-admissible $\mathcal{O}_{\overline{\square}_Y^n}$-algebras, i.e. there is an isomorphism of $\mathcal{O}_{\overline{\square}_{Y_{m+1}} ^n}$-algebras
$$
\mathcal{A}_1 \otimes_{\mathcal{O}_{\overline{\square}_Y^n}} \mathcal{O}_{\overline{\square}_{Y_{m+1}} ^n} \simeq \mathcal{A}_2 \otimes_{\mathcal{O}_{\overline{\square}_Y^n}} \mathcal{O}_{\overline{\square}_{Y_{m+1}} ^n}.$$
Applying $f^*$ to the above, we obtain the isomorphism of $\mathcal{O}_{\overline{\square}_{X_{m+1}}^n}$-algebras
$$
f^* (\mathcal{A}_1) \otimes_{\mathcal{O}_{\overline{\square}_X^n}} \mathcal{O}_{\overline{\square}_{X_{m+1}} ^n} \simeq f^* (\mathcal{A}_2) \otimes_{\mathcal{O}_{\overline{\square}_X^n}} \mathcal{O}_{\overline{\square}_{X_{m+1}} ^n}
$$
so that $(f^* (\mathcal{A}_1) , f^* (\mathcal{A}_2))$ is a pair of mod $(I A_{k'})^{m+1}$-equivalent $(q,n)$-admissible $\mathcal{O}_{\overline{\square}_X^n}$-algebras. Hence we deduce $f^* ( \mathcal{M}_{{\rm d}} ^q (Y, m+1, n)) \subset \mathcal{M}_{{\rm d}} ^q (X, m+1, n)$.

For $\mathcal{M}_{{\rm v}, \geq r} ^q ( m+1, n)$, a similar argument proves the assertion. 
\end{proof}

For push-forwards as well, we consider some needed cases only. 

\begin{lem}\label{lem:fpf}

Let $X, Y$ be integral local $k$-schemes of dimension $\leq 1$, and let $p_1, p_2$ be their closed points, respectively. Let $f: X \to Y$ be a finite surjective morphism of $k$-schemes.
\begin{enumerate}
\item We have the induced finite push-forward morphism
$$ 
f_* : z_{{\rm d}} ^q (X, \bullet) \to z_{{\rm d}} ^q (Y, \bullet).
$$

\item Suppose $\dim \ X = \dim \ Y = 1$. If we restrict the above $f_*$ to the subcomplex $z_{{\rm v}} ^q (X, \bullet)$, then it induces
$$
f_*: z_{{\rm v}} ^q (X, \bullet) \to z_{{\rm v}} ^q (Y, \bullet).
$$

\item Suppose that $Y= \Spec (A)$ is an integral local $k$-scheme of dimension $1$ with the residue field $k= A/I$. Suppose that $f: X= Y_{k'} \to Y$ is the base change for a finite extension $k\hookrightarrow k'$ of fields. Then for each $r \geq 1$, we have
$$
f_*: z^q _{{\rm v}, \geq r} (X, \bullet) \to z^q _{{\rm v}, \geq r} (Y, \bullet).
$$

\item In addition to the assumptions of (3), suppose that $Y$ is henselian. Then we have : 
$$
f_*: z^q _{{\rm d}}  (X/ (m+1), \bullet) \to z^q _{{\rm d}} (Y/ (m+1), \bullet),
$$
$$
f_* : z^q _{{\rm v}, \geq r} (X/ (m+1), \bullet) \to z^q _{{\rm v}, \geq r} (Y/ (m+1), \bullet).
$$
\end{enumerate}

\end{lem}

\begin{proof}
(1) We have the finite push-forward for the higher Chow complexes
\begin{equation}\label{eqn:pushforward 0}
f_*: z^q (X, \bullet) \to z^q (Y, \bullet).
\end{equation}

Thus it is enough to check that $f_* (z^q _{{\rm d}} (X, n)) \subset z^q _{{\rm d}} (Y, n)$ for each $n \geq 0$. We only need to check that conditions $(GP)_*$ and $(SF)*$ are respected under $f_*$.

\medskip

Let $Z \in z_{{\rm d}} ^q (X, n)$ be an integral cycle. Let $F \subset \square^n$ be a face. Since $Z$ satisfies the conditions $(GP)_*$ and $(SF)_*$, the intersections
$$
\overline{Z} \cap (X \times \overline{F}) \ \ \ \mbox{ and } \ \ \ \overline{Z} \cap (\{ p_1 \} \times \overline{F})
$$
are proper on $X \times \overline{\square}^n$. Since $f$ is finite, we have $\overline{f (Z)} = f (\overline{Z})$. Applying the morphism $f$ to the above, we have
$$
\tuborg \dim \ \overline{ f(Z)} \cap (Y \times \overline{F}) \leq \dim \ Z \cap ( X \times \overline{F}) , \\
\dim \ \overline{f (Z)}  \cap ( \{ p_2 \} \times \overline{F}) \leq \ \dim \ \overline{Z} \cap ( \{ p_1 \} \times \overline{F}),\sluttuborg
$$
while $\dim \ Y \times \overline{\square}^n = \dim \ X \times \overline{\square}^n$. Hence we deduce that the intersections $\overline{f(Z)} \cap (Y \times \overline{F})$ and $\overline{f(Z)} \cap ( \{ p_2 \} \times \overline{F})$ are proper on $Y \times \overline{\square}^n$, proving the conditions $(GP)_*$ and $(SF)_*$ for $f(Z)$, respectively. Hence $f_* (Z)$, which is an integer multiple of $f(Z)$, is in $z_{{\rm d}} ^q (Y, n)$. This proves (1).

\medskip

(2) In case $Z \in z_{{\rm v}} ^n (X, n)$ is an integral ${\rm v}$-cycle, consider the Cartesian diagram
\begin{equation}\label{eqn:fpf cartesian}
\xymatrix{ 
\{ p_1 \} \ar[r] ^{\iota_1} \ar[d] ^{f_s}  & X \ar[d] ^f  \\
\{ p_2 \} \ar[r] ^{\iota_2} & Y, }
\end{equation}
where $f_s$ is the restriction of $f$. This is Cartesian because $f^{-1} (p_2) = p_1$. We are given that $\iota_1 ^* (Z) = Z \cap ( \{ p_1 \} \times \square^n) = \emptyset$. From the diagram \eqref{eqn:fpf cartesian}, we deduce that $\iota_2 ^* (f_* (Z)) = f_{s*} \iota_1 ^* (Z) = f_{s*} (\emptyset) = \emptyset$. Hence $f_* (z_{{\rm v}} ^q (X, n) ) \subset z_{{\rm v}} ^q (Y, n)$. This proves (2).

\medskip

(3) Consider the Cartesian diagram that generalizes \eqref{eqn:fpf cartesian}:
\begin{equation}\label{eqn:fpf cartesian r}
\xymatrix{
X_r \times {\square}^n \ar[d] ^{f_r} \ar[r] ^{ \iota_{1, r}} & X \times {\square}^n \ar[d] ^f\\
Y_r \times {\square}^n \ar[r] ^{\iota_{2, r}} & Y \times {\square}^n.}
\end{equation}

Here, if $Z \in z^q _{{\rm v}, \geq r} (X, n)$ is an integral cycle, then it means $\iota_{1, r} ^* (Z) = \emptyset$. Hence from the diagram \eqref{eqn:fpf cartesian r}, we deduce
$$
 \iota_{2,r} ^* (f_* (Z)) = f_{r*} \iota_{1, r} ^* (Z) = f_{r*} (\emptyset) = \emptyset.
 $$
Hence $f_*  (z^q _{{\rm v}, \geq r} (X, n)) \subset z^q _{{\rm v}, \geq r} (Y, n).$ This proves (3).

\medskip

(4) We need to show that the mod $I^{m+1}$-equivalence is preserved under $f_*$. For $X$, we write $\mathcal{M}_{{\rm d}} ^q (X, m+1, n)$ for $\mathcal{M}_{{\rm d}} ^q (m+1, n)$, and similarly for $Y$.

Let $(\mathcal{A}_1, \mathcal{A}_2)$ be a pair of mod $(I A_{k'})^{m+1}$-equivalent $(q,n)$-admissible $\mathcal{O}_{\overline{\square}_X^n}$-algebras, i.e. there is an isomorphism of $\mathcal{O}_{\overline{\square}_{X_{m+1}} ^n}$-algebras
$$
\mathcal{A}_1 \otimes_{\mathcal{O}_{\overline{\square}_X^n}} \mathcal{O}_{\overline{\square}_{X_{m+1}} ^n} \simeq \mathcal{A}_2 \otimes_{\mathcal{O}_{\overline{\square}_X^n}} \mathcal{O}_{\overline{\square}_{X_{m+1}} ^n}.$$
Applying $f_*$ to the above, we obtain the isomorphism of $\mathcal{O}_{\overline{\square}_{Y_{m+1}}^n}$-algebras
$$
f_* (\mathcal{A}_1) \otimes_{\mathcal{O}_{\overline{\square}_Y^n}} \mathcal{O}_{\overline{\square}_{Y_{m+1}} ^n} \simeq f_* (\mathcal{A}_2) \otimes_{\mathcal{O}_{\overline{\square}_Y^n}} \mathcal{O}_{\overline{\square}_{Y_{m+1}} ^n}
$$
so that $(f_* (\mathcal{A}_1) , f_* (\mathcal{A}_2))$ is a pair of mod $I^{m+1}$-equivalent $(q,n)$-admissible $\mathcal{O}_{\overline{\square}_Y^n}$-algebras. Hence we have $f_*( \mathcal{M}_{{\rm d}} ^q (X, m+1, n)) \subset \mathcal{M}_{{\rm d}} ^q (Y, m+1, n)$.

For $\mathcal{M}_{{\rm v}, \geq r} ^q ( m+1, n)$, the identical argument proves the assertion. 
\end{proof}

\section{The graph maps and the case $n=1$}\label{sec:the graph map}

Let $X= \Spec (A)$ be an integral henselian local $k$-scheme of dimension $1$ with the unique closed point $p$. Suppose its residue field is $k= A/ I$, for the maximal ideal $I \subset A$. Let $\mathbb{F}= {\rm Frac}(A)$ be the fraction field. For $r \geq 1$, we let $X_r:= \Spec (A/I^r)$. 

In \S \ref{sec:4 graph}, we define the graph maps from the (resp. relative) Milnor $K$-groups to the Chow groups of ${\rm d}$-cycles (resp. $({\rm v}, \geq r)$-cycles). In \S \ref{sec:regulator n=1}, under the additional assumption that $A$ is regular, we prove the main theorem for $n=1$. 

\subsection{The graph map}\label{sec:4 graph}

Let $n \geq 1$ be an integer. Define the set map
\begin{equation}\label{eqn:gr mult n set}
gr: (A ^{\times} )^{ n} \to z^n (X \times \square^n)
\end{equation}
given by sending $(a_1, \cdots, a_n)$ for $a_i \in A^{\times}$ to the integral closed subscheme $ \Gamma_{(a_1, \cdots, a_n)}$ of $X \times \square^n$ defined by the equations $\{ y_1= a_1, \cdots, a_n = a_n \}$. We first check:

\begin{lem}\label{lem:graph adm 0}
Let $a_i \in A^{\times}$ for $1 \leq i \leq n$. Then:
\begin{enumerate}
\item $\Gamma _{(a_1, \cdots, a_n)} \in z^n_{{\rm d}} (X, n)$.
\item In addition, if $a_{i_0} \equiv 1 \mod I^r$ for some $1 \leq i_0\leq n$ and an integer $r \geq 1$, then we have $\Gamma_{(a_1, \cdots, a_n)} \in z^n_{{\rm v}, \geq r} (X, n)$.
\end{enumerate}
\end{lem}

\begin{proof}
(1) Let $\Gamma:= \Gamma_{(a_1, \cdots, a_n)}$. Its closure in $X \times \overline{\square}^n$ is denoted by $\overline{\Gamma}$, and given by the same equations. This is a codimension $n$ integral closed subscheme of $X \times \overline{\square}^n$. (In fact $\Gamma= \overline{\Gamma}$.)

\medskip

For each face $F \subset \square^n$, let's first compute $\overline{\Gamma} \cap (X \times \overline{\square}^n)$. 

When $F= \square^n$, we have $\overline{\Gamma} \cap (X \times \overline{\square}^n) =\overline{\Gamma}$.

When $F \subset \square^n$ is a codimension $1$ face given by $\{ y_i = \epsilon\}$ for some $1 \leq i \leq n$ and $\epsilon \in \{ 0, \infty \}$, since $a_i \not = \epsilon$, we have $\overline{\Gamma} \cap (X \times \overline{F}) = \emptyset$. In particular, for all proper face $F \subset \square^n$, we also have $\overline{\Gamma} \cap (X \times \overline{F}) = \emptyset$ as well, because each proper face is contained in a codimension $1$ face.

\medskip

The above computations show that the intersection $\overline{\Gamma} \cap (X \times \overline{F})$ is proper on $X \times \overline{ \square}^n$, so that the condition $(GP)_*$ holds for $\Gamma$.

\medskip

Using the above extended face computations, we also deduce that when $F= \square^n$ we have $\overline{\Gamma}' :=\overline{\Gamma} \cap ( \{ p \} \times \overline{\square}^n)= (\bar{a}_1, \cdots, \bar{a}_n) \in \square_k ^n$, where $\bar{a}_i \in k$ is the image of $a_i$. It has the codimension $(n+1)$ in $X \times \overline{\square}^n$, so the intersection is proper on $X \times \overline{\square}^n$. For all proper faces $F \subset \square^n$, we have $ \overline{\Gamma} \cap ( \{ p \} \times \overline{F} ) \subset \overline{\Gamma} \cap (X \times \overline{F}) =\emptyset$. Hence for all faces $F \subset \square^n$, the intersection $\overline{\Gamma} \cap ( \{ p \} \times \overline{\square}^n)$ is proper on $X \times \overline{\square}^n$, proving that $\Gamma$ satisfies the condition $(SF)_*$. Hence we checked that $\Gamma \in z_{{\rm d}} ^n (X, n)$.

\medskip

(2) If we have $a_{i_0} \equiv 1 \mod I^r$ for some $ 1 \leq i_0\leq n$, then by definition $\Gamma$ is a cycle such that $\Gamma \times _X X_r = \emptyset$. Hence $\Gamma \in z^n_{{\rm v}, \geq r} (X, n)$.
\end{proof}

\begin{prop}\label{prop:gr mult n infty}
For integers $n, r \geq 1$, the set map \eqref{eqn:gr mult n set} induces the group homomorphisms
\begin{equation}\label{eqn:gr mult n gp}
gr_{{\rm d}} :K_n ^M (A) \to \CH^n_{{\rm d}} (X, n),
\end{equation}
and
\begin{equation}\label{eqn:gr mult n gp v}
gr_{{\rm v}, \geq r} :  K_n ^M (A, I^r) \to \CH^n_{{\rm v}, \geq r} (X, n).
\end{equation}
In case $r=1$, we write $gr_{{\rm v}} = gr_{{\rm v}, \geq 1}$.
\end{prop}

\begin{proof}
 The basic idea is similar to the one given in \cite[Lemma 2.1]{EVMS}, though we need to do some extra works. We give the arguments with some needed changes.

\medskip

First consider the case $n=1$. For each $f \in A^{\times}$, let $\Gamma_f$ be the closed subscheme given by $\{y_1 = f\}$ in $ \square^1_X$.

Let $f_1, f_2 \in A^{\times}$. We show that there exists a cycle ${W}_{f_1, f_2} \in z^1 _{{\rm d}} (X, 2)$ such that
$$
\partial {W}_{f_1, f_2} = - {\Gamma}_{f_1} - {\Gamma}_{f_2} + {\Gamma}_{f_1f_2}.
$$

\medskip

Consider the closed subscheme ${W}_{f_1, f_2} \subset \square_X ^2$ defined by the polynomial in $y_1, y_2$
\begin{equation}\label{eqn:graph 1 *}
 (f_1 y_1 - f_1 f_2 ) - (y_1 - f_1 f_2) y_2 \in A [ y_1, y_2].
 \end{equation}
One may also describe $W_{f_1, f_2}$ as the closure of the graph of the rational map $\square_X ^1 \dashrightarrow \square_X ^{2}
$, $x \mapsto \left ( x, \frac{ f_1 x - f_1 f_2}{ x - f_1 f_2}\right)$ as done in, for example, Totaro \cite[p.182]{Totaro}. The equation \eqref{eqn:graph 1 *} defines the closure $\overline{W}_{f_1, f_2} \subset X \times \overline{\square}^2$ as well.

\medskip

To check that $W_{f_1, f_2} \in z_{{\rm d}} ^1 (X, 2)$, let's inspect the extended proper faces $\overline{W}_{f_1, f_2} \cap (X \times \overline{F})$ of $\overline{W}_{f_1, f_2}$. When $F \subset \square^2$ is of codimension $1$:
\begin{enumerate}
\item [(a)] For $F= \{ y_1 = 0 \}$: the equation \eqref{eqn:graph 1 *} becomes $ - f_1 f_2 + f_1 f_2 y_2 = 0$ so that we deduce $y_2 = 1$. Hence 
$$
\overline{W}_{f_1, f_2} \cap (X \times \overline{F}) = \{ y_1 = 0, y_2 = 1 \} \subset X \times \overline{\square}^2.
$$
Identifying $F \simeq \square^1_X$, we deduce $\partial_1 ^0 W_{f_1, f_2} =\{ y= 1 \} = \emptyset$ because $1 \not \in \square = \mathbb{P}^1 \setminus \{ 1 \}$.

\item [(b)] For $F= \{ y_1 = \infty \}$: the equation \eqref{eqn:graph 1 *} gives $f_1 - y_2 = 0$. Hence
$$
\overline{W}_{f_1, f_2} \cap (X \times \overline{F})  = \{ y_1 = \infty, y_2 = f_1 \} \subset X \times \overline{\square}^2.
$$ 
Here, $\partial _1 ^{\infty} W_{f_1, f_2} = \{ y = f_1 \} = \Gamma_{f_1}$.
\item [(c)] For $F= \{ y_2 = 0 \}$: the equation \eqref{eqn:graph 1 *} gives $ f_1 y_1 - f_1 f_2 = 0$, which in turn gives $y_1 = f_2$. Hence
$$
\overline{W}_{f_1, f_2} \cap (X \times \overline{F})  = \{ y_1 = f_2, y_2 = 0 \} \subset X \times \overline{\square}^2.
$$
Here, $\partial_2 ^0 W_{f_1, f_2} = \{ y = f_2 \} = \Gamma_{f_2}$.
\item [(d)] For $F= \{ y_2 = \infty \}$: the equation \eqref{eqn:graph 1 *} gives $y_1 - f_1 f_2 = 0$. Hence
$$
\overline{W}_{f_1, f_2} \cap (X \times \overline{F})  =  \{ y_1 = f_1 f_2, y_2 = \infty \} \subset X \times \overline{\square}^2.
$$
Here, $\partial_2 ^{\infty} W_{f_1, f_2} = \{ y= f_1 f_2 \} = \Gamma_{f_1f_2}$.
\end{enumerate}

\medskip

Since all of the nonempty codimension $1$ extended faces of $W_{f_1, f_2}$ are graph cycles, the codimension $2$ extended faces of $W_{f_1, f_2}$ are the codimension $1$ faces of the graph cycles. Hence by the arguments in the proof of Lemma \ref{lem:graph adm 0}-(1), we have $\overline{W}_{f_1, f_2} \cap (X \times \overline{F}) = \emptyset$ for any faces $F \subset \square^n$ of the codimension $\geq 2$. 

\medskip

To check the condition $(GP)_*$ for $W_{f_1, f_2}$, we can use the above extended face computations to deduce the face computations, and by direct inspections, one notes that all of them are in the right codimensions. Hence $W_{f_1, f_2}$ satisfies the condition $(GP)_*$.

\medskip

To check the condition $(SF)_*$, we can again use the above extended face computations to compute $\overline{W}_{f_1, f_2} \cap (\{ p \} \times \overline{F})$. For proper faces of codimension $\geq 2$, we have $\overline{W}_{f_1, f_2} \cap ( \{ p \} \times \overline{F}) \subset \overline{W}_{f_1, f_2} \cap (X \times \overline{F}) = \emptyset$, so the intersection $\overline{W}_{f_1, f_2} \cap ( \{ p \} \times \overline{F})= \emptyset$ is automatically proper. For the codimension $1$ faces $F$, the intersections $\overline{W}_{f_1, f_2} \cap (\{ p \} \times \overline{F})$ are the cycles in (a) $\sim$ (d) modulo $I$, and they give $0$-cycles on $\overline{\square} ^2$. So, their codimensions in $X \times \overline{\square}^2 $ are $3$, and the intersections $\overline{W}_{f_1,f_2} \cap ( \{ p \} \times \overline{F})$ are indeed proper on $X \times \overline{\square}^3$. When $F= \overline{\square}^2$, the intersection $\overline{W}_{f_1, f_2} \cap ( \{ p \} \times \overline{\square}^2)$ is given by a single equation \eqref{eqn:graph 1 *} mod $I$ in $k[ y_1, y_2]$, so its codimension in $\overline{\square}^2$ is $1$. Hence its codimension in $X \times \overline{\square}^2$ is $2$ so that the intersection is proper. We checked that $W_{f_1, f_2}$ satisfies the condition $(SF)_*$, and $W_{f_1, f_2} \in z_{{\rm d}} ^1 (X, 2)$.

\medskip

The calculations of the codimension $1$ faces of ${W}_{f_1, f_2}$ also imply
\begin{equation}\label{eqn:graph 1 **}
\partial {W}_{f_1, f_2} = - {\Gamma}_{f_1} - {\Gamma}_{f_2} + {\Gamma}_{f_1 f_2},
\end{equation}
so that $\Gamma_{f_1f_2} \equiv \Gamma_{f_1} + \Gamma_{f_2}$ in $\CH_{{\rm d}} ^1 (X, 1)$. This implies that $gr_{{\rm d}}$ of \eqref{eqn:gr mult n gp} is a group homomorphism.

\medskip

Now suppose $n \geq 2$. Let $a \in A^{\times}$ be an element such that $1-a \in A^{\times}$. In case $n >2$, then we also let $a_i \in A^{\times}$ for $3 \leq i \leq n$. In what follows, if $n=2$, simply ignore all these elements with $a_i$ for $i >2$. The difference from the case $n=1$ is that we need to check that the image of the Steinberg relation $\{ a, 1-a \} = 0$ vanishes in the cycle class group.

Consider the closed subscheme $W_1 \subset X \times \square^{n+1}$ defined to be the closure in $X \times \square^{n+1}$ of the graph of the rational map 
$$
\square_X ^1 \dashrightarrow \square_X ^{n+1}
$$
\begin{equation}\label{eqn:Steinberg W1}
x \mapsto \left( x, 1-x, \frac{ a-x}{1-x}, a_3, \cdots, a_n \right).
\end{equation}

\medskip

Let's first check that $W_1 \in z^n _{{\rm d}} (X, n+1)$. Let $\overline{W}_1$ be the closure of $W_1$ in $X \times \overline{\square}^{n+1}$. As before, we compute the extended faces of it. When $F \subset \square^{n+1}$ is a codimension $1$ face, we note:
\begin{enumerate}
\item [(a)] For $F=\{ y_1 = 0 \}$: \eqref{eqn:Steinberg W1} gives
$$
\overline{W}_1 \cap  (X \times \overline{F}) = \{ y_1 = 0, y_2 = 1, y_3 = a, y_{i+1} = a_i \ \ \mbox{ for } 3 \leq i \leq n \}.
$$
It has codimension $n+1$ in $X \times \overline{\square}^{n+1}$, while identifying $F = \square_X ^n$, we deduce that $\partial_1 ^0 W_1 = 0$ because $y_2 = 1$ gives no point in $\square^n$.
\item [(b)] For $F=\{ y_1 = \infty \}$: similarly we have
$$
\overline{W}_1 \cap  (X \times \overline{F})  = \{ y_1 = \infty, y_2 = \infty, y_3 = 1, y_{i+1} = a_i \ \ \mbox{ for } 3 \leq i \leq n \}.
$$
It has codimension $n+1$ in $X \times \overline{\square}^{n+1}$, while identifying $F= \square_X ^n$, we deduce that $\partial_1 ^{\infty} W_1 = 0$ because $y_3 = 1$ gives no point in $\square^n$.
\item [(c)] For $F=\{ y_2 = 0 \}$: similarly we have
$$
\overline{W}_1 \cap  (X \times \overline{F})= \{ y_1 = 1, y_2 = 0, y_3 = \infty, y_{i+1} = a_i \ \ \mbox{ for } 3 \leq i \leq n \}.
$$
It has codimension $n+1$ in $X \times \overline{\square}^{n+1}$, and we have $\partial_2 ^0 W_1 =0 $ because $y_1 =1$ gives no point in $\square^n$.
\item [(d)] For $F= \{ y_2 = \infty \}$: similarly we have
$$
\overline{W}_1 \cap (X \times \overline{F}) = \{ y_1 = \infty, y_2 = \infty, y_3 = 1, y_{i+1} = a_i \ \ \mbox{ for } 3 \leq i \leq n \}.
$$
It has codimension $n+1$ in $X \times \overline{\square}^{n+1}$, and we have $\partial_2 ^{\infty} W_1 = 0$ because $y_3 =1$ gives no point in $\square^n$.
\item [(e)] For $F= \{ y_3 = 0 \}$: similarly we have
$$
\overline{W}_1 \cap (X \times \overline{F})  = \{ y_1 = a, y_1 = 1-a, y_3 = 0, y_{i+1} = a_i \ \ \mbox{ for } 3 \leq i \leq n \}.
$$
It has codimension $n+1$ in $X \times \overline{\square}^{n+1}$, and we have $\partial_3 ^0 W_1 = \Gamma_{(a, 1-a, a_3, \cdots, a_n ) }$.
\item [(f)] For $F= \{ y_3 = \infty \}$: similarly we have
$$
\overline{W}_1 \cap (X \times \overline{F})  = \{ y_1 = 1, y_2 = 0, y_3 = \infty, y_{i+1} = a_i \ \ \mbox{ for } 3 \leq i \leq n\}.
$$
It has codimension $n+1$ in $X \times \overline{\square}^{n+1}$, and we have $\partial_3 ^{\infty} W_1 = 0$ because $y_1 = 1$ gives no point in $\square^n$.
\item [(g)] For $F= \{ y_i = \epsilon \}$, where $ 4 \leq i \leq n+1$ and $\epsilon \in \{ 0, \infty \}$: here we have
$$
\overline{W}_1 \cap (X \times \overline{F}) = \emptyset
$$
because $a_i$ for $i \geq 3$ is neither $0$ nor $\infty$.
\end{enumerate}

\medskip

For the extended faces $\overline{F}$ of codimension $\geq 2$, we can simply intersect the above extended codimension $1$ faces and we find by inspection that $\overline{W}_1 \cap (X \times \overline{F}) = \emptyset$. Thus $W_1$ satisfies the condition $(GP)_*$. 

\medskip

To check the condition $(SF)_*$, we need to inspect $\overline{W}_1 \cap ( \{ p \} \times \overline{F})$ for all faces $F \subset \square^{n+1}$. 

When $F= \square^{n+1}$, one notes that by definition $\overline{W}_1 \not \subset \{ p \}\times \overline{\square}^{n+1}$ so that the intersection $\overline{W}_1 \cap ( \{ p \} \times \overline{F})$ is proper on $X \times \overline{\square}^{n+1}$. 

When $F \subset \square^{n+1}$ is a codimension $1$ face, one checks from the above that all of $\overline{W}_1 \cap ( \{ p \} \times \overline{F})$ are of codimension $\geq n+1$ in $X \times \overline{\square}^{n+1}$ because none of $a_i$ is divisible by $t$, thus the intersection is proper. 

When the face $F \subset \square^{n+1}$ is of codimension $\geq 2$, by inspection we note that $\overline{W}_1 \cap ( \{ p \} \times \overline{F}) = \emptyset$, thus the intersection is proper. Hence $W_1$ satisfies the condition $(SF)_*$, and we checked that $W_1 \in z_{{\rm d}} ^n (X, n+1)$. 

\medskip

From the above codimension $1$ face computations, we deduce that 
$$
\partial W_1 = \Gamma_{(a, 1-a, a_3, \cdots, a_n)}.
$$
 In particular, for the Milnor symbol $x=\{ a, 1-a, a_3, \cdots, a_n \} \in K_n ^M (A)$, we have $gr_{{\rm d}} ( x) = 0$ in $\CH^n _{{\rm d}} (X, n)$.

\medskip

The multi-linearity is essentially identical to the proof given for $n=1$ in the above; namely for $f_1, f_2 \in A^{\times}$ and $a_i \in A^{\times}$ for $2 \leq i \leq n$, by considering $W_2 := W_{f_1, f_2} \times \Gamma_{(a_2, \cdots, a_n )}$ for the cycle $W_{f_1, f_2}$ there, we have
$$
\Gamma_{(f_1f_2, a_2, \cdots, a_n)} \equiv \Gamma_{(f_1, a_2, \cdots, a_n)} + \Gamma_{(f_2, a_2, \cdots, a_n)} \ \ \ \mbox{ in } \CH^n_{{\rm d}} (X, n).
$$

\medskip

Now, applying various appropriate permutations of the coordinates $y_1, \cdots, y_{n+1}$ of $W_1$, we deduce that $gr_{{\rm d}}$ induces \eqref{eqn:gr mult n gp} as desired. We shrink details.

\bigskip

To obtain \eqref{eqn:gr mult n gp v}, it is enough to restrict \eqref{eqn:gr mult n gp} to the subgroup $K_n ^M (A, I^r)$ and observe that Lemma \ref{lem:KS} gives generators of the relative group, which apparently belongs to $z_{{\rm v}, \geq r} ^n (X, n)$ by definition.
\end{proof}

We also have the induced map mod $I^{m+1}$ for each $m \geq 1$:

\begin{prop}\label{prop:gr mult n lv m}
For integers $m, n  \geq 1$, the group homomorphism \eqref{eqn:gr mult n gp} induces the group homomorphisms
\begin{equation}\label{eqn:gr mult n gp lv m}
gr_{{\rm d}} : K_n ^M (A/I^{m+1}) \to \CH^n _{{\rm d}} (X / (m+1), n).
\end{equation}
In addition, for $1 \leq r \leq m+1$, the map \eqref{eqn:gr mult n gp v} induces
$$
 gr_{{\rm v}, \geq r}: K_n ^M (A/ I^{m+1}, I^r) \to \CH_{{\rm v}, \geq r} ^n (X/ (m+1), n).
 $$
\end{prop}

\begin{proof}
Since $A$ is a local ring, the natural homomorphism $A^{\times} \to ( A/I^{m+1} )^{\times}$ is surjective (see, e.g. Hartshorne-Polini \cite[Lemma 5.2]{HP}). Thus we have the natural surjective group homomorphism $K_n ^M (A) \to K_n ^M (A/I^{m+1})$. For \eqref{eqn:gr mult n gp lv m}, we want to show that $\ker ( K_n ^M (A) \to K_n ^M (A/I^{m+1}))$ is mapped to $0$ in $\CH^n_{{\rm d}} (X/(m+1), n)$ under $gr_{{\rm d}}$.

Since the above kernel is generated by the elements of the form $\{ 1 + \tau , a_2, \cdots, a_n \}$, where $\tau \in I^{m+1}$, $a_i \in A^{\times}$ for $2 \leq i \leq n$ by Lemma \ref{lem:KS}, it is enough to prove that 
$$
 \Gamma_{(1 + \tau, a_2, \cdots, a_n)} \equiv 0 \ \ \mbox{ in } \CH^n_{{\rm d}} (X/(m+1), n).
 $$
But, this is immediate because this is equivalent to the empty scheme under the mod $I^{m+1}$-equivalence by definition. The second map follows similarly. We shrink details.  This proves the proposition.
\end{proof}

\begin{cor}\label{cor:graph final}
Let $m, n, r \geq 1$ be integers such that $1 \leq r \leq m+1$. 
When $X= \Spec (A)$ is an integral henselian local $k$-scheme of dimension $1$ with the residue field $k=A/I$, there exist the graph homomorphisms
\begin{equation}\label{eqn:graph final 1}
gr_{{\rm d}}: \widehat{K}_n ^M (A) \to \CH_{{\rm d}} ^n (X, n),
\end{equation}
$$
gr_{{\rm v}, \geq r}: \widehat{K} ^M _n (A, I^r) \to \CH_{{\rm v}, \geq r} ^n (X, n),
$$
and
\begin{equation}\label{eqn:graph final 3}
gr_{{\rm d}}: \widehat{K}_n ^M (A/I^{m+1}) \to \CH_{{\rm d}} ^n (X/ (m+1), n),
\end{equation}
$$
gr_{{\rm v}, \geq r}: \widehat{K}_n ^M (A/ I^{m+1}, I^r) \to \CH_{{\rm v}, \geq r} ^n (X/ (m+1), n).
$$
\end{cor}

\begin{proof}
By Lemma \ref{lem:fpf}-(4), the target cycle class groups have the push-forward maps (trace maps) for finite extensions of fields. Hence from the graph maps on $K_n ^n (-)$ we deduce the induced maps from $\widehat{K}_n ^M (-)$ by the universal property of $\widehat{K}_n ^M (-)$ by M. Kerz \cite{Kerz finite}.
\end{proof}

\subsection{The special case $n=1$}\label{sec:regulator n=1}

In \S \ref{sec:regulator n=1}, we suppose that $X= \Spec (A)$ is an integral \emph{regular} henselian local $k$-scheme of dimension $1$ with the residue field $k=A/I$, and the function field $\mathbb{F}$. In this case, $A$ is a DVR, and furthermore $A/ I^{m+1} \simeq k[[t]]/(t^{m+1}) = k_{m+1}$.

\medskip

 Under this assumption, the goal of \S \ref{sec:regulator n=1} is to prove that when $n=1$ the graph homomorphisms in Corollary \ref{cor:graph final} are actually isomorphisms of groups. For this purpose, we construct the ``regulator" homomorphisms $\phi_{{\rm d}}: z^1 _{{\rm d}} (X, 1) \to \widehat{K}^M_1 (A)$ and $\phi_{{\rm v}, \geq r}: z^1 _{{\rm v}, \geq r} (X, 1) \to (1+I^r)^{\times}$, and we show that they induce the desired inverses of $gr_{{\rm d}}$ and $gr_{{\rm v}, \geq r}$, respectively.

\medskip

Let $Z \in z^1 _{{\rm d}} (X, 1)$ be an integral cycle. Since $\overline{Z} \cap \{ y_1 = \infty \} = \emptyset$ by Lemma \ref{lem:proper int face *}, we may regard $\overline{Z}$ as a closed subscheme in $X \times \mathbb{A}^1$. So, there is a prime ideal $P \subset A[ y_1]$ of height $1$ such that $\overline{Z}= \Spec (A [ y_1] / P)$. We will stick to this assumption in what follows in \S \ref{sec:regulator n=1}.

\begin{lem}\label{lem:norm Witt}
Let $Z \in z^1 _{{\rm d}} (X, 1)$ be an integral cycle. 
Then:
\begin{enumerate}
\item There exists a unique monic polynomial $f \in A [y_1]$ in $y_1$ such that $P= (f)$. 

When $d \geq 1$ is the degree of $f$ in $y_1$, write
\begin{equation}\label{eqn:f_1 finite}
f = y_1 ^d  - a_{d-1} y_1 ^{d-1} + \cdots + (-1)^{d-1} a_1 y_1 + (-1)^d a_0,
\end{equation}
for some $a_0, \cdots, a_{d-1} \in A$.
\item The image $\bar{y}_1$ in $A [y_1]/(f)$ of $y_1$, is a unit in the ring, and it is integral over $A$.
\item Consider the finite extension $\mathbb{F} (\bar{y}_1)$ of $\mathbb{F}$ and let 
$$
{\rm Nm}:= {\rm N}_{ \mathbb{F} (\bar{y}_1)/ \mathbb{F}}:  \mathbb{F}( \bar{y}_1) ^{\times} \to  \mathbb{F}^{\times}
$$
be the norm map associated to the field extension.

Then ${\rm Nm} (\bar{y}_1) = a_0 \in A^{\times}$.
\item 
Suppose in addition that $Z$ is in $z^1 _{{\rm v}, \geq r} (X, 1)$. 

Then ${\rm Nm} (\bar{y}_1) = a_0$ actually belongs to the subgroup $(1+I^r)^{\times} \subset A^{\times}$.
\end{enumerate}
\end{lem}

\begin{proof}
(1) Since $A$ is a regular local ring, it is a UFD. Thus so is $A[y_1]$. Since $P \subset A[y_1]$ is a height $1$ prime ideal of a UFD, there exists some $f \in A[y_1]$ such that $P = (f)$ (see e.g. Matsumura \cite[Theorem 20.1, p.161]{Matsumura}). By Lemma \ref{lem:proper int face *}, the intersection of $\overline{Z}$ with the extended face $\{ y_1 = \infty\}$ is empty. Thus we deduce that the leading coefficient of $f$ in $y_1$ is a unit in $A$. After scaling by this invertible leading coefficient, we may assume $f$ is a monic polynomial in $A[y_1]$ in $y_1$. The uniqueness of $f$ is apparent, as $f$ is now monic in $y_1$. This proves (1). 

\medskip

(2) Since $\overline{Z}$ does not intersect with the extended face $\{ y_1 = 0\}$ by Lemma \ref{lem:proper int face *}, we have $a_0 \in A^{\times}$. On the other hand, \eqref{eqn:f_1 finite} implies that modulo $P= (f)$, 
$$
\bar{y}_1 \left( \frac{  \bar{y}_1^{d-1} - \cdots + (-1)^{d-1} a_1)}{ (-1)^{d-1} a_0 ^{-1}} \right)= 1
$$
 in $A[ y_1]/(f)$. That $\bar{y}_1$ is integral over $A$ is apparent by (1). This proves (2).

\medskip

(3) By the construction of the norm map for finite extension of fields, we have ${\rm Nm} (\bar{y}_1) = a_0$, which is in $A^{\times}$ by construction.

\medskip

(4) Now we suppose that ${Z}$ is an integral strict vanishing cycle of order $ r \geq 1$. This means the equation $(f \mod I^r) = 0$ in $(A/I^r) [y_1]$ has $\{y_1=1\}$ as the only solution. Thus by Hensel's lemma, $f$ factors into the product of linear polynomials $(y_1 - \alpha_i)$ for $\alpha _i \in A$, where $\alpha_i \equiv 1 \mod I^r$. In particular, $a_0 \equiv \prod \alpha_i \equiv 1 \mod I^r$ so that $a_0 \in (1 + I^r)^{\times}$, proving (4).
\end{proof}

Using Lemma \ref{lem:norm Witt} we define the following:

\begin{defn}\label{defn:norm cycle}
For the integral cycle $Z$ as in Lemma \ref{lem:norm Witt}, we say the graph cycle $\Gamma_{a_0} \subset X \times \square^n$ given by the equation $\{ y_1 = a_0\}$, where $a_0 := {\rm Nm} (\bar{y}_1)$, is the \emph{norm cycle} of $Z$. Since the minimal polynomial $f$ that defines $\overline{Z}$ is uniquely determined by $Z$, so is the norm cycle.
\qed
\end{defn}

\begin{defn}\label{defn:cycle Witt}
Define the homomorphism of abelian groups 
$$
\phi_{{\rm d}}: z^1_{{\rm d}} (X, 1) \to A^{\times}
$$
$$
{Z} \mapsto {\rm Nm} ( \bar{y}_1),
$$
by sending each integral cycle ${Z}$ to the norm of $\bar{y}_1$, where $\bar{y}_1$ is the image of $y_1$ in the coordinate ring $A[y_1]/P$ of ${Z}$. Here, by Lemma \ref{lem:norm Witt}-(3), it belongs to $A^{\times}$. 

Furthermore, its restriction to the subgroup $z^1_{{\rm v}, \geq r}(X, 1) \subset z^1 _{{\rm d}} (X, 1)$ induces
$$
\phi_{{\rm v}, \geq r}: z^1 _{{\rm v}, \geq r} (X, 1) \to (1+I^r)^{\times}
$$
by Lemma \ref{lem:norm Witt}-(4). If $r=1$, we write $\phi_{{\rm v}} = \phi_{{\rm v}, \geq 1}$.
\qed
\end{defn}

\begin{lem}\label{lem:gr surj lv 1}
Let ${Z} \in z^1 _{{\rm d}} (X, 1)$ be an integral cycle and let $\bar{y}_1$ be as in Lemma \ref{lem:norm Witt}-(2).
\begin{enumerate}

\item Then there exists a cycle $W \in z^1 _{{\rm d}} (X, 2) $ such that 
$$
{\Gamma}_{{\rm Nm} (\bar{y}_1)} - Z = \partial W,
$$
where $\Gamma_{a}$ for $a \in A^{\times}$ is the graph cycle associated to $\{y= a\}$.
In particular, the graph homomorphism $gr_{{\rm d}} : K^M _1 (A) \to \CH^1 _{{\rm d}} (X, 1)$ is surjective.

\item If in addition $Z \in z^1 _{{\rm v}, \geq r} (X, 1)$, then the above $W$ can be chosen in $z^1 _{{\rm v}, \geq r} (X, 2)$. In particular, the graph homomorphism $gr_{{\rm v}, \geq r}: (1+I^r)^{\times} \to \CH^1 _{{\rm v}, \geq r} (X, 1)$ is also surjective.
\end{enumerate}
\end{lem}

\begin{proof}
(1) The argument is motivated by the construction in B. Totaro \cite[p.185]{Totaro}, while we need to do several more tasks.

For $f (y_1) \in A[ y_1]$ that defines $\overline{Z}$ as in \eqref{eqn:f_1 finite} of Lemma \ref{lem:norm Witt}, consider the closed subscheme of $\square_X^2$ defined by the polynomial in $A[y_1, y_1]$
\begin{equation}\label{eqn:gr surj lv 1 0} 
{W}_{f}: \ \ \ f (y_1) - (y_1 -1)^{d-1} (y_1 - a_0) y_2.
\end{equation}
Let $\overline{W}_f \subset \overline{\square}_X ^2$ be the closure. 

\medskip

So as to check that $W_f$ satisfies the conditions $(GP)_*$ and $(SF)_*$, let's first inspect the codimension $1$ extended faces $\overline{W}_f \cap (X \times \overline{F})$ of $W_f$ for the faces $F \subset \square^2$. When $F \subset  \square^2$ is a codimension $1$ face:
\begin{enumerate}
\item [(a)] For $F= \{ y_1 = 0 \}$: from \eqref{eqn:gr surj lv 1 0}, we deduce $(-1)^d a_0 - (-1)^{d-1} (-a_0) y_2 = 0$, which implies that $y_2 = 1$. 
Hence
$$
\overline{W}_f \cap (X \times \overline{F})= \{ y_1 = 0, y_2 = 1 \} \subset X \times \overline{\square}^2.
$$
Identifying $F\simeq \square_X ^1$, we deduce $\partial_1 ^0 W_f = \{ y = 1 \} = 0$ because $1 \not \in \square^1 = \mathbb{P}^1 \setminus \{ 1 \}$.
\item [(b)] For $F= \{ y_1 = \infty \}$: from \eqref{eqn:gr surj lv 1 0}, we deduce $1-y_2 = 0$. Hence
$$
\overline{W}_f \cap (X \times \overline{F}) = \{ y_1 = \infty, y_2 = 1 \} \subset X \times \overline{\square}^2.
$$
In particular $\partial_1 ^{\infty} W_f = \{ y = 1 \} =0$.
\item [(c)] For $F= \{ y_2 = 0 \}$: from \eqref{eqn:gr surj lv 1 0}, we deduce $f(y_1) = 0$, i.e. 
$$
\overline{W}_f \cap ( X \times \overline{F}) = \overline{Z} \times \{ y_2 = 0 \} \subset X \times \overline{\square}^2.
$$
We deduce that $\partial_2 ^0 W_f = Z$.
\item [(d)] For $F= \{ y_2 = \infty \}$: from \eqref{eqn:gr surj lv 1 0}, we deduce that $(y_1 -1)^{d-1} (y_1 - a_0) = 0$.  Hence
$$
\overline{W}_f \cap (X \times \overline{F}) = (d-1) \{ y_1 = 1, y_2 = \infty\} + \{ y_1 = a_0, y_2 = \infty \}.
$$
We deduce that $\partial _2 ^{\infty} W_f = (d-1) \{ y_1 = 1\} + \{ y_1 = a_0 \} = \Gamma_{a_0}= \Gamma_{{\rm Nm} (\bar{y}_1)}$.
\end{enumerate}

The codimension $ 2$ extended faces of $W_f$ are the codimension $1$ extended faces of the above codimension $1$ extended faces of $\overline{W}_f$. By inspection, we see that they are all empty. Here, we used Lemma \ref{lem:proper int face *} for some cases since $Z$ and $\Gamma_{a_0}$ satisfy the condition $(SF)_*$; it is given for $Z$, while for the graph cycle $\Gamma_{a_0}$ it is proven in Lemma \ref{lem:graph adm 0}.

\medskip

The above extended face computations in particular imply that $W_f$ satisfies the condition $(GP)_*$.

\medskip

To check the condition $(SF)_*$ for $W_f$, we compute $\overline{W}_f \cap ( \{ p \} \times \overline{F})$ for all faces $F \subset \square^2$ based on the above extended face computations. We saw in the above that the codimension $\geq 2$ extended faces of $W_f$ are empty. Hence, it is enough to check $(SF)_*$ for $\overline{F}= \overline{\square}^2$ and the codimension $1$ extended faces.

For the codimension $1$ extended faces, we see that $\overline{W}_f \cap ( \{ p \} \times \overline{F})$ are the cycles in (a) $\sim$ (d) mod $I$, and they are either empty or $0$-cycles in $\overline{\square}_k ^2$, thus of codimension $\geq 3$ in $X \times \overline{\square}^2$. Hence the intersection is proper. For $\overline{F}= \overline{\square}^2$, the intersection $\overline{W}_f \cap ( \{ p \} \times \overline{\square}^2)$ is given by a single equation \eqref{eqn:gr surj lv 1 0} mod $I$ in $k[y_1, y_2]$, so its codimension in $\overline{\square}_k ^2$ is $1$, thus its codimension in $X \times \overline{\square}^2$ is $2$. Thus the intersection is proper on $X \times \overline{\square}^2$, and we checked that $W_f$ satisfies $(SF)_*$, and $W_f \in z_{{\rm d}} ^1 (X, 2)$.

\medskip

The above codimension $1$ face computations also imply that 
$$
\partial {W}_{f} = {\Gamma}_{{\rm Nm} (\bar{y}_1)} - {Z}
$$
so that ${Z} \equiv {\Gamma}_{{\rm Nm} (\bar{y}_1)}$ in $\CH^1 _{{\rm d}} (X, 1)$. Since ${\Gamma}_{{\rm Nm} (\bar{y}_1)} \in {\rm Im} (gr_{{\rm d}})$ by definition, we deduce that ${Z} \in {\rm Im} (gr_{{\rm d}})$ as well. This proves that $gr_{{\rm d}}$ is surjective.

\medskip

(2) When $Z\in z^1 _{{\rm v}, \geq r} (X, 1)$, we show that the cycle $W_f$ defined in \eqref{eqn:gr surj lv 1 0} is in fact a vanishing cycle of order $\geq r$. By (1), we know that $W_f \in z_{{\rm d}} ^1 (X, 2)$ already, so we just check its vanishing order.

Indeed, that $Z \in z_{{\rm v}, \geq r}^1 (X, 1)$ implies that $a_0 \equiv 1 \mod I^r$ by Lemma \ref{lem:norm Witt}-(4). Thus the equation \eqref{eqn:gr surj lv 1 0} mod $I^r$ is equivalent to
\begin{equation}\label{eqn:lv 1 *}
(y_1 - 1)^d - (y_1 -1)^d y_2 = (y_1 -1)^d (1- y_2) = 0 \mod I^r.
\end{equation}
Apparently $y_1 -1$ and $1- y_2$ are not zero-divisors of $(A/ I^r) [ y_1, y_2]$, so we have $\{y_1 = 1\}$ and $\{y_2 = 1\}$ as the only solutions of \eqref{eqn:lv 1 *}, and they both give the empty scheme in $\square^2$ because $1 \not \in \square = \mathbb{P}^1 \setminus \{ 1 \}$. Thus ${W}_{f}$ is a vanishing cycle of order $\geq r$. This implies that $gr_{{\rm v}, \geq r}$ is surjective.
\end{proof}

The following lemma is at the heart of the proof for $n=1$:

\begin{lem}\label{lem:phi partial 0 lv 1}
$\phi _{{\rm d}} (\partial z^1 _{{\rm d}} (X, 2)) = 0.$ In particular, $\phi_{{\rm v}, \geq r} ( \partial z^1 _{{\rm v}, \geq r} (X, 2)) = 0$ as well.
\end{lem}

\begin{proof}
Let $W \in z^1 _{{\rm d}} (X, 2)$ be an integral cycle. We prove that $\phi_{{\rm d}} (\partial W) = 0$.

The closure $\overline{W} \subset \overline{\square}_X ^2$ of $W$ gives a projective flat surjective morphism $\overline{W} \to X$ of relative dimension $1$. This is flat because $X$ is regular of dimension $1$ (Lemma \ref{lem:basic Milnor 1}-(1)). Hence the base change $\overline{W}_{\eta} \to \eta$ gives an integral projective curve over the field $\mathbb{F}$, where $\eta \in \Spec (A)$ is the generic point. Let  $\mathbb{K}:= \mathbb{F}( \overline{W}_{\eta})$, the rational function field of the curve $\overline{W}_{\eta}$ over $\mathbb{F}$.

Consider the normalization of $\overline{W}$ composed with the closed immersion
\begin{equation}\label{eqn:Suslin 0}
\nu : Y:= \overline{W} ^N \to \overline{W} \hookrightarrow  \overline{\square}_{X} ^2.
\end{equation}

The composite $Y \to X$ gives a normal fibered surface in the sense of Q. Liu \cite[Ch.8 Definition 3.1, p.347]{Liu}. Taking the generic fibers of \eqref{eqn:Suslin 0} over $\eta$, we obtain 
$$
\nu_{\eta} : Y_{\eta}   \to \overline{W} _{\eta} \hookrightarrow \overline{\square}_{\mathbb{F}}^2,
$$
and $Y_{\eta}$ is an integral \emph{regular} projective curve over $\mathbb{F}$, whose function field is also $\mathbb{K} = \mathbb{F} ( Y_{\eta} )$. We denote $\nu_{\eta}$ also simply by $\nu$ when no confusion arises.

\medskip

The coordinates $y_1, y_2 \in \overline{\square}_{\mathbb{F}}^2$ define rational functions on $\overline{\square}_{\mathbb{F}} ^2$. Consider the pull-backs $\nu^* (y_i) \in \mathbb{K}$ on the curve $ Y_{\eta}$ for $i=1,2$. For simplicity, we will still write $y_1, y_2$ for these pull-backs on $Y_{\eta}$. They define the symbol $\{ y_1, y_2 \} \in K_2 ^M (\mathbb{K})$. 

We observe that we have the one-to-one correspondences: 
\begin{enumerate}
\item [(i)] the closed points $\mathfrak{p}$ of the curve $Y_{\eta}$ over $\mathbb{F}$
\item [(ii)] the discrete valuations $\partial_{\mathfrak{p}}$ on $\mathbb{K}$ over $\mathbb{F}$
\item [(iii)] the integral horizontal divisors $\overline{ \{ \mathfrak{p} \}}$ on the normal surface $Y = \overline{W}^N$
\item [(iv)] the integral horizontal divisors on the surface $\overline{W}$
\end{enumerate}
Here the correspondences between (i) and (ii) and between (iii) and (iv) are classical, while the correspondence between (i) and (iii) comes from Q. Liu \cite[Ch.8 Proposition 3.4, Definition 3.5, p.349]{Liu}, for instance.

Recall that (e.g. from J. Milnor \cite[Lemma 2.1]{Milnor IM}) that the discrete valuation $\partial_{\mathfrak{p}}$ associated to a given closed point $\mathfrak{p} \in Y_{\eta}$ defines the boundary homomorphism, denoted by the same notation
\begin{equation}\label{eqn:boundary Milnor K}
\partial_{\mathfrak{p}} : K_2 ^M (\mathbb{K}) \to \kappa (\mathfrak{p})^{\times},
\end{equation}
where $\kappa (\mathfrak{p})$ is the residue field of $\mathfrak{p}$, thus a finite extension of the field $\mathbb{F}$. When $\pi_{\mathfrak{p}}$ is a uniformizing parameter of the discrete valuation $\partial_{\mathfrak{p}}$ and $\partial_{\mathfrak{p}} (u) = 0$ for $u \in \mathbb{K}^{\times}$, then $\partial_{\mathfrak{p}} ( \{ \pi_{\mathfrak{p}} ^n, u \}) = (\bar{u})^n$ where $\bar{u}$ is the image of $u$ in $\kappa (\mathfrak{p})^{\times}$, and this requirement uniquely determines \eqref{eqn:boundary Milnor K}.

For the symbol $\{ y_1, y_2 \} \in K^M_2 ( \mathbb{K})$, all but finitely many closed points $\mathfrak{p}$ of $Y_{\eta}$ satisfy $\partial_{\mathfrak{p}} \{ y_1, y_2 \} = 1$, the identity of the group $\kappa (\mathfrak{p})^{\times}$. Consider the set
$$ 
\mathcal{F}_W:= \{ \mathfrak{p} \in Y_{\eta} \ | \ \partial_{\mathfrak{p}} \{ y_1, y_2 \} \not = 1.\}.
$$

Here, by the definition of \eqref{eqn:boundary Milnor K}, we observe immediately that:
\medskip

\textbf{Claim 1:} \emph{Under the correspondence between {\rm (i)} and {\rm (iv)}, the members of the set $\mathcal{F}_W$ correspond to the integral horizontal divisors appearing in the non-empty faces $\partial_i ^{\epsilon} W$ over $i \in \{ 1, 2 \}$ and $\epsilon \in \{ 0, \infty \}$.}

\medskip

We let $Z(\mathfrak{p})$ be the irreducible component of a nonempty face for $\mathfrak{p} \in \mathcal{F}_W$ under the correspondence in Claim 1. Since $W$ intersects properly with all faces, including the codimension $2$ faces, the irreducible cycle $Z(\mathfrak{p})$ belongs to a unique codimension $1$ face of $W$. Hence, for each $\mathfrak{p} \in \mathcal{F}_W$, there exists a unique pair $(i, \epsilon)$ where $i \in \{1, 2 \}$ and $\epsilon \in \{ 0, \infty \}$, that remembers the face to which the component $Z (\mathfrak{p})$ belongs. Under this, we let $n_{\mathfrak{p}}:= \partial_{\mathfrak{p}} ( y_i) \in \mathbb{N}$. In this case, for $i' \in \{ 1, 2 \} \setminus \{ i \}$, we must have $\partial_{\mathfrak{p}} (y_{i'}) = 0$.

\medskip

Recall from Definition \ref{defn:cycle Witt} that we had $\phi_{{\rm d}} (Z (\mathfrak{p})) \in A^{\times}$. Here, if we let $f \in A [y]$ be the monic minimal polynomial in $y$ of $\overline{Z(\mathfrak{p})}$ of the form \eqref{eqn:f_1 finite}, then by definition
$$ 
\phi_{{\rm d}} (Z (\mathfrak{p})) = {\rm Nm} (\bar{y}),
$$
for the norm map ${\rm Nm} : \mathbb{F} (\bar{y})^{\times} = \kappa (\mathfrak{p})^{\times} \to \mathbb{F}^{\times}$. Since both the residue field $\kappa (\mathfrak{p})$ and the norm map ${\rm Nm}$ change as $\mathfrak{p}$ runs over the set $\mathcal{F}_W$, we denote the norm map by ${\rm Nm}_{\mathfrak{p}}$.

\medskip

\textbf{Claim 2:}  \emph{
For each $\mathfrak{p} \in \mathcal{F}_W$, we have the equality
\begin{equation}\label{eqn:Claim2-0}
{\rm Nm}_{\mathfrak{p}} ( \partial_{\mathfrak{p}} \{ y_1, y_2 \}) = \left( \phi_{{\rm d}} ( Z (\mathfrak{p})) \right) ^{n _{\mathfrak{p}}} \ \ \mbox{ in } A^{\times}.
\end{equation}}

For the proof, for a fixed $\mathfrak{p} \in \mathcal{F}_W$, without loss of generality we may assume $i=1$ and $\epsilon = 0$, as the other cases are all similar. Here, $y_2$ is now the coordinate function for $Z(\mathfrak{p})$ so that by the definition of the boundary map, we have
\begin{equation}\label{eqn:Claim2-1}
 \partial_{\mathfrak{p}} \{y_1, y_2 \}= \bar{y}_2 ^{ n_{\mathfrak{p}}},
\end{equation}
where $\bar{y}_2$ is the image of $y_2$ in the residue field $\kappa (\mathfrak{p})$. Now
\begin{equation}\label{eqn:Claim2-2}
{\rm Nm}_{\mathfrak{p}} ( \bar{y}_2 ^{n_{\mathfrak{p}}}) = \left( {\rm Nm}_{\mathfrak{p}} (\bar{y}_2) \right) ^{ n_{\mathfrak{p}}} = \left( \phi_{{\rm d}} ( Z (\mathfrak{p})) \right) ^{n _{\mathfrak{p}}} \ \ \mbox{ in } \mathbb{F}^{\times},
\end{equation}
so that \eqref{eqn:Claim2-1} and \eqref{eqn:Claim2-2} together imply the equality \eqref{eqn:Claim2-0}, \emph{a priori} in $\mathbb{F}^{\times}$. Since $\phi_{{\rm d}} (Z(\mathfrak{p})) \in A^{\times}$ in general (recall Lemma \ref{lem:norm Witt} and Definition \ref{defn:cycle Witt}), the equality shows that the value in fact belongs to $A^{\times} \subset \mathbb{F}^{\times}$.

\medskip

Returning back to the proof of the lemma, recall again that if $\mathfrak{p} \in Y_{\eta}$ does not belong to $\mathcal{F}_W$, then $\partial_{\mathfrak{p}} \{ y_1, y_2 \} = 1$ by definition. Hence we see that 
$$
\phi_{{\rm d}} (\partial {W} ) = \phi_{{\rm d}} ( - \partial_1 ^{0} {W} + \partial_1 ^{\infty} {W} + \partial_2 ^0 {W} - \partial_2 ^{\infty} {W}) 
$$
$$ 
= \prod_{ \mathfrak{p} \in \mathcal{F}_W} ( \phi_{{\rm d}} ( Z (\mathfrak{p})))^{n_{\mathfrak{p}}} =^{\dagger}  \prod _{ \mathfrak{p} \in \mathcal{F}_W} {\rm Nm}_{\mathfrak{p}}  (\partial_{\mathfrak{p}} \{ y_1, y_2 \}) = \prod_{\mathfrak{p} \in Y_{\eta}} {\rm Nm}_{\mathfrak{p}} (\partial_{\mathfrak{p}} \{ y_1, y_2 \})=1,
$$
where $= ^{\dagger}$ holds by Claim 2, and the last quantity in $A^{\times}$ is equal to $1$ by the Suslin-Weil reciprocity \cite{Suslin}. Here, this reciprocity theorem applies because the cardinality $|\mathbb{F}| = \infty$. This finishes the proof.
\end{proof}

\begin{cor}
The maps $\phi_{{\rm d}}$ and $\phi_{{\rm v}, \geq r}$ induce group homomorphisms
$$
\phi_{{\rm d}}: \CH^1 _{{\rm d}} (X, 1) \to K^M_1 (A),
$$
$$
\phi_{{\rm v}, \geq r}: \CH^1 _{{\rm v}, \geq r} ( X, 1) \to (1+I^r)^{\times} \ \ \mbox{ for } r \geq 1.
$$
\end{cor}

\begin{proof}
This follows immediately from Lemma \ref{lem:phi partial 0 lv 1}.
\end{proof}

\begin{lem}\label{lem:gr inj lv 1}
The composites of the group homomorphisms
$$
\phi_{{\rm d}} \circ gr_{{\rm d}}: K^M_1 (A) \to \CH^1 _{{\rm d}}(X, 1) \to K^M_1 (A),
$$
$$
\phi_{{\rm v}, \geq r} \circ gr_{{\rm v}, \geq r}: (1+I^r)^{\times} \to \CH^1 _{{\rm v}, \geq r} (X, 1) \to (1+I^r)^{\times}, \ \ \mbox{ for } r \geq 1,
$$
 are the identity maps of $K_1 ^M (A)$ and $(1+I^r)^{\times}$, respectively.

In particular, the graph homomorphisms $gr_{{\rm d}}$ and $gr_{{\rm v}, \geq r}$ are injective.
\end{lem}

\begin{proof}
It is enough to prove the statement for $gr_{{\rm d}}$ only. For each $f \in K^M_1 (A)$, the cycle $Z_f \in z^1 _{{\rm d}} (X, 1)$ is given by the irreducible monic polynomial $y_1 - f \in A [ y_1]$ of degree $1$ in $y_1$. Here ${\rm Nm} (\bar{y}_1) = f$ so that $\phi_{{\rm d}} ( Z_f) = f$ by definition. Thus $\phi_{{\rm d}} ( gr_{{\rm d}} (f)) = f$, as desired.
\end{proof}

We now deduce part of the main theorems for $n=1$:

\begin{thm}\label{thm:main 0}
Let $k$ be a field. 
Let $X= \Spec (A)$ be an integral regular henselian local $k$-scheme of dimension $1$ with the residue field $k= A/I$. Then
$$
 gr_{{\rm d}}: K^M_1 (A) \to \CH^1 _{{\rm d}} (X, 1),
 $$
$$
 gr_{{\rm v}, \geq r}: (1+I^r)^{\times} \to \CH^1 _{{\rm v}, \geq r} (X, 1)
 $$
are isomorphisms of abelian groups.
\end{thm}

\begin{proof}
The surjectivity and the injectivity follow from Lemmas \ref{lem:gr surj lv 1} and \ref{lem:gr inj lv 1}, respectively.
\end{proof}

\begin{remk}
When $A= k[[t]]$, for $gr_{{\rm v}}= gr_{{\rm v}, \geq 1}$, we note that $(1+ I)^{\times}=\mathbb{W}(k)$ and it is a commutative ring, the ring of the big Witt vectors, where the summation operation of the ring is induced from the product of the formal power series, while the product operation of the ring is a somewhat exotic. One may wonder whether this isomorphism $gr_{{\rm v}}$ of abelian groups in Theorem \ref{thm:main 0} can be promoted to an isomorphism of \emph{rings}. This can be done, but we need to describe / define an appropriate ring structure on $\CH^1 _{{\rm v}} (X, 1)$ first. This is discussed in \S \ref{sec:ring Witt}, although strictly speaking it is not necessary for the purpose of this article.
\qed
\end{remk}

We can improve the isomorphisms of Theorem \ref{thm:main 0} for mod $I^{m+1}$:

\begin{thm}\label{thm:main 0 mod m}

Let $k$ be a field. Let $X= \Spec (A)$ be an integral regular henselian local $k$-scheme of dimension $1$ with the residue field $k= A/I$. 
Let $ m, r \geq 1$ be integers such that $1 \leq r \leq m+1$. Then the induced homomorphisms
$$
gr_{{\rm d}}: \widehat{K}_1 ^M (A / I ^{m+1}) \to \CH_{{\rm d}} ^1 (X/ (m+1), 1)
$$
$$
gr_{{\rm v}, \geq r} : ( 1 + I^r) ^{\times} / (1+ I^{m+1}) ^{\times} \to \CH_{{\rm v}, \geq r} ^1 (X/ (m+1), 1)
$$
are isomorphisms, where we recall $A/ I^{m+1} \simeq k_{m+1}$.
\end{thm}

\begin{proof}
It is enough to show that the graph cycle of the form
$$
Z= \{ y_1 = 1 + t^{m+1}f \},
$$
where $t \in I$ is a uniformizer of the DVR $A$ and $f \in A$, vanishes in the cycle class groups mod $I^{m+1}$. Indeed,
$$
{Z} \times_X X_{m+1} \equiv \{ y_1 = 1 \} = \emptyset
$$
in $X_{m+1} \times {\square}^n$. Thus we have the result.
\end{proof}

\subsection{The ring structure}\label{sec:ring Witt}
In \S \ref{sec:ring Witt}, we treat the special case when $X= \Spec (k[[t]])$. For $ m \geq 0$, let $X_{m+1} = \Spec (k_{m+1})$. We give the description of the concrete ring structure on $\CH^1 _{{\rm v}} (X, 1)$. This is not needed for the rest of the article and the reader who wishes not to read the details may skip \S \ref{sec:ring Witt} entirely.

We briefly recall the definitions on the big Witt vectors $\mathbb{W}_m (R)$. One may use K. R\"ulling \cite{R} as a practical reference. 

\subsubsection{The ring of big Witt vectors}
When $R$ is a commutative ring with unity, let $\mathbb{W}(R):= R^{\mathbb{N}}$ as a set. Define the ghost map
$
gh: \mathbb{W}(R) \to R^{\mathbb{N}},
$ by sending
$
gh ( a_n) = (w_n)$, where $w_n = \sum_{d | n } d a_d ^{ \frac{n}{d}}.$

The target $R^{\mathbb{N}}$ of $gh$ is a commutative ring with unity, with the coordinate-wisely defined ring structure. Then there exists a unique functorial ring structure on $\mathbb{W} (R)$ such that the ghost map $gh$ becomes a functorial homomorphism for all commutative rings $R$ with unity.

An alternative description of the ring structure on $\mathbb{W}(R)$ is to identify $\mathbb{W} (R)$ with $ (1 + t R[[t]])^{\times}$, where the sum of the underlying abelian group of $\mathbb{W} (R)$ corresponds to the product of the formal power series, and the product structure on $\mathbb{W}(R)$ is given by some nontrivial expressions on the level of the power series. See Remark \ref{remk:star op} for its idea. We denote the product structure by $\star$.

This description depends on the following important property of the elements of $(1+ t R[[t]])^{\times}$ (see e.g. S. Bloch \cite[I-\S 1-1, p.192]{Bloch crys}):

\begin{prop}\label{prop:Witt elements}
Each element $x \in (1+ t R[[t]])^{\times}$ can be uniquely expressed as an infinite product of the form
\begin{equation}\label{eqn:Witt infinite}
 x = \prod_{i=1} ^{\infty} (1 - \alpha_i t^i),
 \end{equation}
where $\alpha_i \in R$.
\end{prop}

A minor difference between Proposition \ref{prop:Witt elements} and \emph{loc.cit.} is that we don't put the power $(-1)$, namely, we use $(1- \alpha_i t^i)$ instead of $(1-\alpha_i t^i)^{-1}$ of \emph{ibid.}

\begin{remk}\label{remk:star op} 
Under the Proposition \ref{prop:Witt elements}, the product structure $\star$ on $\mathbb{W} (R)$ which is commutative, associative, and distributive over the summation is determined uniquely by requiring
\begin{equation}\label{eqn:star detail}
(1- a t^m) \star (1- b t^n) = (1- a^{ \frac{n}{r} } b ^{ \frac{m}{r}} t^{\frac{mn}{r}} )^r,
\end{equation}
for all $m, n \geq 1$ and $a, b \in R$, where $r= {\rm gcd} (m,n)$. See S. Bloch \cite[Proposition (I.1), p.192]{Bloch crys}. 
\qed
\end{remk}

In general, finding the sequence of elements $\alpha_1, \alpha_2, \cdots$ in \eqref{eqn:Witt infinite} can be done by elementary but tedious calculations. Furthermore, even if $x$ has only finitely many terms, a polynomial in $t$, under this factorization in \eqref{eqn:Witt infinite} we could possibly have infinitely many non-vanishing entries $\alpha_i$.

The additive group $\mathbb{W} (R)$ has a decreasing filtration of subgroups 
$$
 \mathbb{W} (R) \supset U^1 \supset U^2 \supset U^3 \supset \cdots,
 $$
where in terms of the identification $\mathbb{W}(R)= (1 + tR[[t]])^{\times}$, we have $U^m= (1 + t^{m+1} R[[t]])^{\times}$ with $U^0 = \mathbb{W}(R)$. The quotient is written as
\begin{equation}\label{eqn:U m}
\mathbb{W}_m (R) := \mathbb{W} (R) / U^m.
\end{equation}

\begin{lem}
The subgroup $U^m \subset \mathbb{W} (R)$ is an ideal. In particular, there exists a unique ring structure on $\mathbb{W} _m (R)$ such that the canonical surjection $\mathbb{W} (R) \to \mathbb{W}_m (R)$ is a ring homomorphism.
\end{lem}

\begin{proof}
We sketch the proof. 
In the light of Proposition \ref{prop:Witt elements} and the distributive law, it is enough to prove that for any elements of the form $1 - a t^n$ with $a \in R$, and $u \in U^m$, we have $(1- at^n) \star u \in U^m$.

Here, the element $u$ can also be written as $u = \prod_{i=m+1} ^{\infty} (1 - \alpha_i t^i), \ \ \ \mbox{ for some } \alpha_i \in R,$ so that it is enough to prove that $(1- a t^n) \star (1 - \alpha_i t^i) \in U^m$ for $i \geq m+1$. This is apparent because we have $(1- a t^n) \star (1 - \alpha_i t^i)  \in U^m,$ by \eqref{eqn:star detail} in Remark \ref{remk:star op}.
\end{proof}

\medskip

We saw that we have the group isomorphism
$$ 
gr_{{\rm v}}: \mathbb{W} (k) = (1 + t k[[t]])^{\times} \to \CH^1 _{{\rm v}} (X, 1).
$$
Since $\mathbb{W} (k)$ is a ring, we can induce the ring structure on $\CH^1 _{{\rm v}} (X, 1)$ via the isomorphism. Here, we want to have a closer look at the product operation.

\subsubsection{Modulus filtration and modulus topology}

\begin{defn}\label{defn:mod m}
For $m \geq 0$, let
$$
M^m := z^1_{{\rm v}, \geq m+1} ( X, 1) = \{ Z \in z^1 _{{\rm v}} (X, 1) \ | \ Z \mbox{ : integral, } Z \times_{X} X_{m+1} = \emptyset \}.
$$
We have a decreasing filtration
$$
\cdots \subset M^2 \subset M^1 \subset M^0 = z_{{\rm v}} ^1 (X, 1).
$$
\qed
\end{defn}

\begin{defn}
Let $G_{{\rm v}} ^1 (X, 1) \subset z^1_{{\rm v}} (X, 1)$ be the subset of the integral graph cycles, i.e the integral subschemes of the form $\{y_1 = a \}$ for some $a \in \mathbb{W} (k)$. Let $\Gamma_{{\rm v}} ^1 (X, 1) \subset z^1_{{\rm v}} (X, 1)$ be the subgroup generated by $G_{{\rm v}} ^1 (X, 1)$. 

From the filtration $M^{\bullet}$ on $z^1_{{\rm v}} (X, 1)$, we have the induced filtrations
$$ 
N^{\bullet} G^1_{{\rm v}} (X, 1) := M^{\bullet} \cap G^1 _{{\rm v}} (X, 1),
$$
$$
N ^{\bullet}\Gamma_{{\rm v}} ^1 (X, 1):= M^{\bullet} \cap \Gamma_{{\rm v}} ^1 (X, 1).
$$
The former is a set filtration, while the latter is a group filtration. 

Note that $N^m \Gamma_{{\rm v}} ^1 (X, 1)$ is the subgroup of $z^1 _{{\rm v}} (X, 1)$ generated by the integral cycles of the form $\{ y_1 = a \}$, where $a \in U^m$ in terms of the subgroup $U^m$ in \eqref{eqn:U m}.
\qed
\end{defn}

We have a few basic results:

\begin{lem}\label{lem:star M N}
Let $Z \in z^1_{{\rm v}} (X, 1)$ be an integral cycle and let ${\rm Nm} (Z)$ be its norm cycle (see Definition \ref{defn:norm cycle}). 

For an integer $m \geq 1$, if $Z \in M^m$, then ${\rm Nm} (Z) \in N^m$.
\end{lem}

\begin{proof}
Since $Z\in M^m$, we have 
$$
 Z \times _{X} X_{m+1} = \emptyset.
 $$
This means that inside the ring $k[[t]][y_1]$, the ideal generated by $t^{m+1}$ and $f$ satisfies $(t^{m+1}, f) = (1)$. In particular, there exist some $r= r( t, y_1), s = s (y, y_1) \in k[[t]] [ y_1]$ such that
\begin{equation}\label{eqn:star M N}
 r t^{m+1} + s f = 1.
 \end{equation}

Recall that we had $f (0) = (-1)^d a_0$. Thus evaluating \eqref{eqn:star M N} at $y_1=0$, we get 
\begin{equation}\label{eqn:star M N 1}
r( t, 0) t^{m+1} +s ( t, 0) (-1)^d a_0 = 1.
\end{equation}

Evaluating \eqref{eqn:star M N 1} at $t=0$, we deduce that $s (0, 0) (-1)^d a_0 = 1$. Hence putting it back to \eqref{eqn:star M N 1}, we deduce that
$$ 
s( t, 0) ^{-1} = (-1)^d s'
$$
for some $s' \in \mathbb{W}(k)$. Here from \eqref{eqn:star M N 1}, we have 
$$
 f (0) = (-1)^d a_0 = s (t, 0)^{-1} ( 1 - r (t, 0) t^{m+1}) = (-1)^d s' ( 1 - r (t,0) t^{m+1}),
 $$ 
so that
$$
 a_0 = s' (1 - r (t,0) t^{m+1}) \in ( 1 + t^{m+1} k[[t]])^{\times} = U^m.
 $$

This proves that ${\rm Nm} (Z) \in N^m$.
\end{proof}

\begin{lem}\label{lem:separable}
We have
$$
 \bigcap _{m \geq 0} N^m G^1 _{{\rm v}} (X, 1) = \emptyset.
 $$
\end{lem}

\begin{proof}
We show that there is no integral nonempty cycle $Z \in \Gamma^1 _{{\rm v}} (X, 1)$ such that $Z \in N^m$ for all $m \geq 0$.

Indeed, suppose $Z$ is integral but nonempty in $\Gamma^1 _{{\rm v}} (X, 1)$. Let $f:= y_1 - a_0 \in k[[t]][y_1]$ be the equation as in Lemma \ref{lem:norm Witt} that defines $Z$. Here, $a_0 \in U^m$ for all $m \geq 0$ implies that in fact $a_0 = 1 $ in $\mathbb{W} (k)$. However, this shows that $Z$ is defined by $y_1 = 1$, which gives an empty scheme in $\square_X ^1$ because $\square= \mathbb{P}^1 \setminus \{ 1 \}$, a contradiction. This proves the lemma.
\end{proof}

\begin{defn}
We will call the decreasing filtrations $N^{\bullet}$ on $G_{{\rm v}} ^1 (X, 1)$ and $\Gamma^1_{{\rm v}} (X, 1)$, the \emph{modulus filtrations}.

Since we have a natural bijection between $G_{{\rm v}} ^1 (X, 1)$ and $\mathbb{W} (k)$ as a set, this filtration defines a Hausdorff $t$-adic metric topology on $G_{{\rm v}} ^1 (X, 1)$. We call it the \emph{modulus topology} on $G_{{\rm v}}^1 (X, 1)$.  
\qed
\end{defn}

We summarize this as the following:

\begin{prop}
The modulus topology on $G_{{\rm v}} ^1 (X, 1)$ is induced by the $t$-adic metric topology, and this topological space is complete, i.e. all Cauchy sequences converge.
\end{prop}

\begin{proof}
The only additional remark we need to make is that $G_{{\rm v}} ^1 (X, 1)$ is homeomorphic to $\mathbb{W}(k)$ as topological spaces with the $t$-adic metric, and the latter is complete because all power series in $t$ converges in the set $\mathbb{W} (k)$ by definition.
\end{proof}

\begin{prop}
The filtrations $N^{\bullet}$ on $\Gamma_{{\rm v}} ^1 (X, 1)$ and $M^{\bullet}$ on $z^1 _{{\rm v}} (X, 1)$ induce the same filtrations on $\CH^1 _{{\rm v}}(X, 1)$ via the natural homomorphisms for $m \geq 1$
$$
\xymatrix{ M^{m}  \ar@{^{(}->}[r] &  z^1 _{{\rm v}} (X,1)  \ar@{>>}[r]   & \CH^1 _{{\rm v}} (X, 1) \\
N^{m}  \ar@{^{(}->}[r]  \ar@{^{(}->}[u] & \Gamma ^1 _{{\rm v}} (X, 1) \ar@{^{(}->}[u] \ar@{>>}[ru] &  }
$$
\end{prop}

\begin{proof}
The map $z^1_{{\rm v}} (X, 1) \to \CH^1 _{{\rm v}} (X, 1)$ is surjective by definition, while $\Gamma^1 _{{\rm v}} (X, 1) \to \CH^1 _{{\rm v}} (X, 1)$ is surjective by Lemma \ref{lem:gr surj lv 1}. To show that the images of $M^{m}$ and $N^{m}$ in $\CH^1 _{{\rm v}} (X, 1)$ coincide, it is enough to show that for each integral cycle $Z \in M^m$, we have
\begin{equation}\label{eqn:equiv to graph m}
Z  =  \Gamma + \partial W
\end{equation}
for some cycle $\Gamma \in N^m$ and $W \in z^1 _{{\rm v}} (X, 2)$. 

In fact, we saw in Lemma \ref{lem:norm Witt} that there exists a unique monic polynomial $ f \in k[[t]] [y_1] $ in $y_1$ that defines the ideal of $\overline{Z}$. Expressing $f$ as $f= y_1 ^d - a_{d-1} y_1 ^{d-1} + \cdots + (-1)^d a_0$, with $a_0 \in k[[t]]^{\times}$ as in \eqref{eqn:f_1 finite}, we saw $a_0 \in \mathbb{W} (k)$. Here $a_0$ is uniquely determined by $Z$, and $Z$ is equivalent $\Gamma_{a_0}$ in $\CH^1 _{{\rm v}} (X, 1)$, where $\Gamma_{a_0}$ is the graph cycle given by the equation $\{ y_1 = a_0 \}$ by Lemma \ref{lem:gr surj lv 1}. Furthermore, that $Z \in M^m$ implies that precisely that $a_0 \in U^m$, thus $\Gamma_{a_0} \in N^m$ by Lemma \ref{lem:star M N}. Hence taking $\Gamma := \Gamma_{a_0}$, we have \eqref{eqn:equiv to graph m}, as desired.
\end{proof}

\begin{defn}
For each $m \geq 0$, we let $\overline{M}^m \subset \CH^1 _{{\rm v}} (X, 1)$ be the common image of $M^m$ and $N^m$ in the Chow group of strict vanishing cycles. This $\overline{M}^m$ is a decreasing filtration, and by Lemma \ref{lem:separable}, it also satisfies the property that 
$$
\bigcap_{m \geq 0} \overline{M}^m = 0
$$
in $\CH^1 _{{\rm v}} (X, 1)$.

The topology defined by $\overline{M}^{\bullet}$ on $\CH^1 _{{\rm v}} (X, 1)$ is a Hausdorff metric topology given by the $t$-adic metric, and it is complete because the natural set map $G_{{\rm v}} ^1 (X, 1) \to \CH^1 _{{\rm v}} (X, 1)$ induces a homeomorphism.
\qed
\end{defn}

\subsubsection{The Witt product structure}

Let $ Z \in z^1 _{{\rm v}} (X, 1)$ be an integral cycle. We saw in Lemma \ref{lem:norm Witt} that for some $a_0 \in \mathbb{W}(k) = (1+ tk[[t]])^{\times}$ uniquely determined by $Z$, we have $Z\equiv \Gamma_{a_0}$ in  $\CH^1 _{{\rm v}} (X, 1)$. 

For each $a_0 \in \mathbb{W}(k)$, by Proposition \ref{prop:Witt elements}, there exist a unique sequence of elements $\alpha_i \in k$ for $i \geq 1$ such that
$$
a_0 = \prod_{i=1} ^{\infty} (1 - \alpha_i t^i).
$$

Hence for each $m \geq 1$, we have $$a_0 \equiv \prod_{i=1} ^m (1 - \alpha_i t^i) \mod U^m.$$ In terms of their corresponding graph cycles, by the additivity of $\CH^1 _{{\rm v}} (X, 1)$, we have 
$$
 \Gamma_{a_0} = \sum_{i=1} ^m \Gamma_{(1- \alpha_i t^i)} \mod \overline{M} ^m
 $$
so that taking $m \to \infty$, by the completeness of $\CH^1_{{\rm v}} (X, 1)$ in the topology, we have the equality
\begin{equation}\label{eqn:Z decomposition}
 Z \equiv \Gamma_{a_0} \equiv \sum_{i=1} ^{\infty} \Gamma_{ (1- \alpha_i t^i)}, \ \ \ \mbox{ in } \CH^1 _{{\rm v}} (X, 1),
 \end{equation}
Using this idea, we can perform the ring operation of two cycle classes in $\CH^1_{{\rm v}} (X, 1)$ represented by two integral cycles in $z^1_{{\rm v}} (X, 1)$.

\medskip

Indeed, if $Z_1, Z_2 \in z^1 _{{\rm v}} (X, 1)$ are two integral cycles, express
$$ 
Z_1 \equiv \sum_{i=1} ^{\infty} \Gamma_{(1 - \alpha_i t^i)}, \ \ \ \ Z_2 \equiv \sum_{j=1} ^{\infty} \Gamma_{ ( 1- \beta_j t^j)},
$$
for some uniquely determined sequences $\alpha_i, \beta_j \in k$. Then there exists a unique way to define the product $Z_1 \star Z_2$ in $\CH^1_{{\rm v}} (X, 1)$ that is distributive over the sum, and such that
$$
\Gamma_{(1 - \alpha_i t^i)} \star \Gamma_{ (1- \beta_j t^j)} = \Gamma_{ (1 - \alpha_i t^i) \star ( 1- \beta_j t^j)},
$$
where the second $\star$ in the above is the product structure on $\mathbb{W}(k)$. See Remark \ref{remk:star op}.

From the associativity and the commutativity of $\star$ on $\mathbb{W}(k)$, we deduce the corresponding properties for the $\star$ on $\CH^1 _{{\rm v}} (X, 1)$. On the other hand, one notes that $gr_{{\rm v}}$ maps $U^m$ to $\overline{M}^m$ for each $m \geq 0$. This shows that $gr_{{\rm v}}$ is a filtered ring homomorphism.

We summarize the discussions as:
\begin{thm}\label{thm:n=1 m}
Let $k$ be a field and let $X= \Spec (k[[t]])$. Then $(\CH^1 _{{\rm v}} (X, 1), +, \star)$ is a commutative ring and the graph homomorphism $gr_{{\rm v}}: \mathbb{W} (k) \to \CH^1 _{{\rm v}} (X, 1)$ is a ring isomorphism. Furthermore, for each integer $m \geq 1$, the induced graph homomorphism $gr_{{\rm v}}: \mathbb{W}_m (k) \to \CH^1_{{\rm v}} (X/(m+1), 1)$ is also a ring isomorphism. 
 \end{thm}

\section{The geometric presentations}\label{sec:local case}

In \S \ref{sec:local case}, let $X= \Spec (A)$ be an integral henselian local $k$-scheme of dimension $1$, with the maximal ideal $I \subset A$. Suppose the residue field is $k = A/ I$ and let $\mathbb{F}= {\rm Frac} (A)$, the function field. Here, we \emph{do not yet} suppose that $A$ is regular until we get to \S \ref{sec:triangular}.

\medskip

In \S \ref{sec:local case}, we prove a set of important results. After discussing the definitions and properties of compactified projections (\S \ref{sec:compact proj}) and the constructions of regulator maps (\S \ref{sec:5 regulator}) via the Gersten conjecture of M. Kerz \cite{Kerz finite}, from \S \ref{sec:triangular}, where we suppose $X$ is regular, we develop a ``reduction of degree" argument, a variant of an argument in \cite{Park MZ}, to show that every cycle class in $\CH^n_{{\rm d}} (X, n)$ has a representative given by a sum of the graph cycles. Out of these, we deduce that the groups $\CH_{{\rm d}} ^n (X, n)$ and $\CH_{{\rm v},\geq r} ^n (X, n)$ (resp. $\CH_{{\rm d}} ^n (X/(m+1), n)$ and $\CH_{{\rm v}, \geq r} ^n (X/ (m+1), n)$ compute the Milnor $K$-groups $\widehat{K}_n ^M (A)$ and $\widehat{K}_n ^M (A, I^r)$ (resp. $\widehat{K}_n ^M (A/I^{m+1})$ and $\widehat{K}_n ^M (A/I^{m+1}, I^r)$), answering the main theorems.

\subsection{Compactified projections}\label{sec:compact proj}

We discuss a convenient technical tool. 

Let $n \geq 1$ and let $Z \in z^n _{{\rm d}} (X, n)$ be an integral cycle. Since $Z \to X$ may not be finite, for a nonempty subset $J \subset \{ 1, \cdots, n \}$, and the corresponding projection map
$$
 pr_J: \square_X ^n \to \square_X ^{|J|}, 
$$
that ignores the coordinates $y_i$ for $i \not \in J$, the projection $pr_J (Z)$ of $Z$ may not be closed in $\square_X ^{|J|}$. 

However the morphism $\overline{Z} \to X$ is finite and surjective by Corollary \ref{cor:proper int t=0}, so we can improve the situation by projecting after compactifications. Let
$$ 
\widehat{pr}_J: \overline{\square}_X ^n \to \overline{\square}_X ^{|J|}
$$
be the projection that ignores the coordinates $y_i$ for $i \not \in J$. Then for $\overline{Z}$, we have the induced finite surjective morphisms. 
$$
 \overline{Z} \to \widehat{pr}_J (\overline{Z}) \to X.
$$
This encourages the following more general definition:

\begin{defn}\label{defn:comp proj 0}
Let $W \subset \square_X ^n$ be an integral closed subscheme and let $\overline{W}$ be its Zariski closure in $\overline{\square}_X ^n$. Take the projection
$$
\overline{W}^{(J)}:= \widehat{pr}_J (\overline{W}) \subset  \overline{\square}_X ^{ |J|},
$$
which is an integral closed subscheme. Taking its restriction to the open subscheme $\square^{|J|}_X$, we let
$$
W^{(J)}:= \widehat{pr}_J (\overline{W}) |_{ \square^{|J|}_X}.
$$
 This is a closed subscheme of $ \square^{|J|}_X$. We call them the \emph{compactified projections of $W$} to the coordinates of $J$.
 
 In case $J = \{ 1, \cdots i \}$ for some $1 \leq i \leq n$, we simply write $\overline{W}^{(i)}$ and $W^{(i)}$. 
\qed
\end{defn}

For the cycles in $z^n_{{\rm d}} (X, n)$ and $z^n_{{\rm d}} (X, n+1)$, we show in Lemmas \ref{lem:adm compact proj} and \ref{lem:adm comp proj 2} below that the compactified projections behave well, respectively. The first lemma is needed for \S \ref{sec:local case}, while the second one is needed in \cite{Park vanishing}.

\begin{lem}\label{lem:adm compact proj}
Let $Z \in z^n _{{\rm d}} (X, n)$ be an integral cycle and let $J \subset \{ 1, \cdots, n\}$ be a nonempty subset. Then we have $ Z^{(J)} \in z^{|J|}_{{\rm d}} (X, |J|).$
\end{lem}

\begin{proof}
We prove the lemma for $ J = \{ 1, \cdots,  i\}$ for some $1 \leq i \leq n$, i.e. ${Z}^{(i)} \in z^i _{{\rm d}} (X, i)$. The general case follows by applying suitable permutations of the coordinates.

\medskip

We begin with:

\medskip

\textbf{Claim :} \emph{For each nonempty proper face $F \subset \square^i$, we have $\overline{Z}^{(i)} \cap (X \times \overline{F} ) = \emptyset$.}

\medskip

Note that $\overline{Z}^{(i)} \cap (X \times \overline{F})$ is the image of $\overline{Z} \cap (X \times \overline{G})$ under the projection $\widehat{pr}_J$, where $G \subset \square^{n}$ is the face defined by the same set of equations as $F$, but in the larger space $\square^n$. Since $Z$ satisfies $(SF)_*$, by Lemma \ref{lem:proper int face *}, the latter $\overline{Z} \cap (X \times \overline{G})$ is empty. Hence the former $\overline{Z}^{(i)} \cap (X \times \overline{F})$ is empty as well, proving the Claim.

\medskip

This Claim in particular implies that $(GP)_*$ holds for $Z^{(i)}$.

\medskip

To check the property $(SF)_*$ for $Z^{(i)}$, we need to check that the intersection $\overline{Z} ^{(i)} \cap ( \{ p \} \times \overline{F})$ is proper on $X \times \overline{\square}^i$ for each face $F \subset \square^i$. However $\overline{Z}^{(i)} \cap ( \{ p \} \times \overline{F}) \subset \overline{Z}^{(1)} \cap (X \times \overline{F})$, and by the Claim, this is empty for each proper face $F$ in $\square^n$, so we consider the remaining case $F= \square^n$. 

Since $Z$ satisfies $(SF)_*$, it satisfies $(DO)$ by Lemma \ref{lem:SF DF}. Hence the morphisms $\overline{Z} \to \widehat{pr}_J (\overline{Z}) \to X$ are finite and dominant, thus finite surjective, and in particular, the morphism $\overline{Z}^{(i)} \to X$ is surjective.

Since $\{ p \} \times  \overline{\square}^n$ is an integral divisor on $X \times \overline{\square}^n$, that the intersection $\overline{Z}^{(i)} \cap ( \{ p \} \times \overline{\square}^n)$ is proper on $X \times \overline{\square}^n$ is equivalent to that $\overline{Z}^{(i)}  \not \subset  \{ p \}\times \overline{\square}^n$.
 
Toward contradiction, suppose $\overline{Z}^{(i)} \subset  \{ p \}\times \overline{\square}^n$. This means the image of the map $\overline{Z} ^{(i)} \to X$ is concentrated at the closed point $\{ p \}$, but this contradicts the above observation that the morphism $\overline{Z}^{(i)} \to X$ is surjective. This proves that $Z^{(i)}$ satisfies $(SF)_*$. Hence $Z ^{(i)} \in z_{{\rm d}} ^{i} (X, i)$. 
\end{proof}

\begin{cor}\label{cor:v compact proj}
 Let $J \subset \{ 1, \cdots, n \}$ be a nonempty subset. Then:
\begin{enumerate}
\item the specialization to $p$ gives the map
$$
{\rm ev}_p : z^{|J|}_{{\rm d}} (X, |J|) \to z^{|J|} ( k, |J|).
$$
\item Let $Z \in z^n_{{\rm v}} (X, n)$ be an integral cycle and suppose that there is some $i \in J$ such that $y_i$ is a vanishing coordinate for $Z$. Then we have $Z^{(J)} \in z^{|J|}_{{\rm v}} (X, |J|)$. 
\item For $Z \in z_{{\rm v}} ^n (X, n)$, the image of $Z ^{(J)}$ under ${\rm ev}_{p}$ of $(1)$ is $0$.
\end{enumerate}
\end{cor}

\begin{proof}
(1) follows immediately from Lemma \ref{lem:specialization}.

For (2), when $Z \in z^n_{{\rm v}} (X, n)$ and $y_i \equiv 1$ at a point in $\overline{Z} \cap (\{ p \} \times \square^n)$ (see Remark \ref{remk:type i}), then the same holds for $\overline{Z}^{ (J)} \cap (\{ p \} \times \square^{|J|})$. Hence $y_i$ is a vanishing coordinate for $Z^{ (J)}$, so by Lemma \ref{lem:adm compact proj}, we have $Z^{(J)} \in z^{|J|} _{{\rm v}} (X, |J|)$. 

From (2), we deduce (3) immediately.
\end{proof}

The following is used in \cite{Park vanishing}, but we put it in this article, as this seems to be a more appropriate location:

\begin{lem}\label{lem:adm comp proj 2}
Let $W \in z_{{\rm d}} ^n (X, n+1)$ be an integral cycle. Let $J$ be a nonempty subset of $ \{ 1, \cdots, n+1\}$ such that $\overline{W} \to \widehat{pr}_J (\overline{W})$ is quasi-finite. Then $W ^{(J)} \in z_{{\rm d}} ^{ |J|-1 } ( X, |J|)$.

In addition to the above, if $W \in z_{{\rm v}} ^n (X, n+1)$, and there is some $i \in J$ with $1 \leq i \leq n+1$ such that $y_i$ is a vanishing coordinate for $W$, then $W^{ (J)} \in z_{{\rm v}} ^{ |J|-1} (X, |J|)$. 
\end{lem}

\begin{proof}
Without loss of generality, we may assume that $J = \{ 1, \cdots, i \}$ for some $1 \leq i \leq n$. Furthermore, by induction, we may even assume that $i=n$ so that $J= \{ 1, \cdots, n\}$ and $|J|=n$. Since $\overline{W} \to \widehat{pr}_J (\overline{W}) = \overline{W}^{(J)}$ is quasi-finite, the morphism $\overline{W} \to \overline{W}^{(J)}$ is finite surjective, being projective and quasi-finite. In particular, $W ^{ (J)}$ has the codimension $n-1$ in $X \times \square^n$. 

We also note that by Lemma \ref{lem:SF DF} $\overline{W} \to X$ is a dominant projective (thus surjective) morphism, which factors through $\overline{W}^{(J)}$, so that $ \overline{W}^{(J)} \to X$ is surjective as well. 

\medskip

In order to check the conditions $(GP)_*$ and $(SF)_*$ for $W^{(J)}$, we first inspect its extended faces. 

Let $F \subset \square^n$ be a face. Note that $\overline{W} ^{(J)} \cap (X \times \overline{F})$ is the image of $\overline{W} \cap (X \times \overline{G})$ under the projection $\widehat{pr}_J$, where $G \subset \square^{n+1}$ is the face defined by the same set of equations as $F$, but in the larger space $\square^{n+1}$. Since $W$ satisfies $(GP)_*$, the intersection $\overline{W} \cap (X \times \overline{G})$ is proper on $X \times \overline{\square}^{n+1}$, while $\overline{W} \to \overline{W}^{(J)}$ is finite so that $\dim \ \overline{W} ^{(J)} \cap (X \times \overline{F}) \leq \dim \ \overline{W} \cap (X \times \overline{G})$. Since ${\rm codim}_{\overline{\square}^n} \overline{F} + 1 = {\rm codim}_{\overline{\square}^{n+1}} \overline{G}$, this implies that $\overline{W}^{(J)} \cap (X \times \overline{F})$ is proper on $X \times \overline{\square}^n$, proving $(GP)_*$ for $W^{(J)}$. 

\medskip

The argument for the condition $(SF)_*$ for $W^{(J)}$ is similar: we need to check that the intersection $\overline{W}^{ (J)} \cap ( \{ p \} \times \overline{F})$ is proper on $X \times \overline{\square}^n$ for each face $F \subset \square^n$. Here, $\overline{W} ^{(J)} \cap (\{ p \} \times \overline{F})$ is the image of $\overline{W} \cap (\{ p \} \times \overline{G})$, where $G \subset \square^{n+1}$ is defined as before by the same set of equations as those for $F$. Since $\overline{W} \to \overline{W} ^{(J)}$ is finite, we have $\dim \ \overline{W}^{(J)} \cap  ( \{ p \} \times \overline{F}) \leq \dim \ \overline{W} \cap ( \{ p \} \times \overline{G})$ so that by dimension counting again we check that the intersection $\overline{W}^{ (J)} \cap ( \{ p \} \times \overline{F})$ is indeed proper on $X \times \overline{\square}^n$, proving $(SF)_*$ for $W^{(J)}$. Thus $W^{(J)} \in z_{{\rm d}} ^{ |J|-1} (X, |J|)$.

\medskip

The second part is apparent because $W^{ (J)}  \cap ( \{ p \} \times \square^n)$ is equal to the image of $W \cap ( \{ p \} \times \square^{n+1})$, thus it is empty, showing that $W^{ (J)}$ is a pre-vanishing cycle if $W$ is a pre-vanishing cycle.
\end{proof}

\subsection{The higher regulators}\label{sec:5 regulator}
As before, $X= \Spec (A)$ is an integral henselian local $k$-scheme of dimension $1$ with the residue field $k= A/I$ and the function field $\mathbb{F}$. In \S \ref{sec:5 regulator} we will construct certain homomorphisms
$$
\phi_{{\rm d}}: z^n_{{\rm d}} (X, n) \to \widehat{K}_n ^M (A),
$$
such that the boundary $\partial z^n _{{\rm d}} (X, n+1)$ is mapped to $0$. When $X$ is in addition regular, later we will check that the induced map on $\CH^n_{{\rm d}} (X, n)$ gives the inverse of the graph map $gr_{{\rm d}}$ of Corollary \ref{cor:graph final}. We will check that it furthermore respects the mod $I^{m+1}$-equivalence relation.

\medskip

The construction of $\phi_{{\rm d}}$ builds on the Gersten conjecture for the improved Milnor $K$-theory proven by M. Kerz \cite{Kerz finite}. An analogous idea was previously used in \cite{Park MZ} in a model based on formal schemes. We will adjust it suitably to fit it into our situation.

\begin{defn}
Let $n \geq 1$ be an integer. Let $Z \in z^n _{{\rm d}} (X, n)$ be an integral cycle and let $\overline{Z}$ be the Zariski closure in $X \times \overline{\square}^n$. The morphism $\overline{Z} \to X$ is finite surjective by Corollary \ref{cor:proper int t=0}. Let $\overline{Z}_{\eta}$ be the generic fiber over the generic point $\eta \in X$, which gives a finite extension $\mathbb{F} \hookrightarrow \kappa (\overline{Z}_{\eta})$ of fields.

For the coordinate functions $y_1, \cdots, y_n \in \overline{\square}_X ^n$, denote by $\bar{y}_1, \cdots, \bar{y}_n$ their restrictions on $\overline{Z}_{\eta}$. They define the Milnor symbol $\{ \bar{y}_1, \cdots, \bar{y}_n \} \in K_n ^M ( \kappa (\overline{Z}_{\eta}))$. Hence applying the Bass-Tate--Kato norm ${\rm Nm}:= {\rm N}_{\kappa (\overline{Z}_{\eta})/ \mathbb{F}}: K_n ^M (\kappa (\overline{Z}_{\eta})) \to K_n ^M (\mathbb{F})$ (see \cite{BassTate} and \cite{Kato}), we obtain an element $\phi_{{\rm d}} (Z):={\rm Nm} ( \{ \bar{y}_1, \cdots, \bar{y}_n \}) \in K_n ^M (\mathbb{F}).$ Since $\mathbb{F}$ is a field, we have $\widehat{K}_n ^M (\mathbb{F}) = K_n ^M (\mathbb{F})$.

Extending $\mathbb{Z}$-linearly, we thus obtain a homomorphism
$$
\phi_{{\rm d}}: z^n _{{\rm d}} (X, n) \to \widehat{K}_n ^M ( \mathbb{F}).
$$
Later in Theorem \ref{thm:gr_n iso summary}, we will prove that the image of this map actually belongs to the subgroup $\widehat{K}_n ^M (A)$ of $\widehat{K}_n ^M (\mathbb{F})$. That this is indeed a subgroup follows from the Gersten conjecture for $\widehat{K}_n ^M (-)$ resolved by M. Kerz \cite{Kerz finite}.
\qed
\end{defn}

\begin{lem}\label{lem:regulator graph}
For $a_1, \cdots, a_n \in A^{\times}$, the image $\phi_{{\rm d}} (\Gamma_{(a_1, \cdots, a_n)} )$ of the graph cycle $\Gamma_{(a_1, \cdots, a_n)}$ associated to the Milnor symbol $\{ a_1, \cdots, a_n \} \in \widehat{K}_n ^M (A)$ is $\{ a_1, \cdots, a_n \}$ in $\widehat{K}^M_n (\mathbb{F})$. In particular, $\phi_{{\rm d}} (\Gamma_{(a_1, \cdots, a_n)} ) \in \widehat{K}_n ^M (A)$.
\end{lem}

\begin{proof}
Let $Z= \Gamma_{(a_1, \cdots, a_n) }$. The natural map $\overline{Z} \to X$ is of degree $1$, and it induces the identity map $\mathbb{F} \to \kappa (\overline{Z}_{\eta})= \mathbb{F}$ of the residue fields of the generic fibers. Hence by definition of the norm map, we deduce the lemma. 
\end{proof}

An important property is the following reciprocity result:

\begin{prop}
The composite homomorphism
$$
 z^n _{{\rm d}} (X, n+1) \overset{\partial}{\to} z^n _{{\rm d}} (X, n) \overset{\phi_{{\rm d}}}{\to} \widehat{K}_n ^M (\mathbb{F})
 $$
vanishes. In particular, we have the induced homomorphism
\begin{equation}\label{eqn:phi d 0}
 \phi_{{\rm d}} : \CH_{{\rm d}} ^n (X, n) \to \widehat{K}_n ^M (\mathbb{F}).
\end{equation}
\end{prop}

\begin{proof}
We follow the outline of \cite[Sect. 3]{Totaro} with various adjustments for our situation.

Let $W \in z^n_{{\rm d}} (X, n+1)$ be an integral cycle. Let $\overline{W}$ be its Zariski closure in $X \times \overline{\square}^{n+1}$. Let $\overline{C}:= \overline{W}_{\eta}$ and $C:= W_{\eta}$ be the fibers over the generic point $\eta \in X$. Here $\overline{C}$ is an integral projective curve over $\mathbb{F}$ and $C \subset \overline{C}$ is an open subscheme. Let $\overline{D}\to \overline{C}$ and $D \to C$ be their normalizations. In particular, $\overline{D}$ is a regular integral projective curve over $\mathbb{F}$. Let $\mathbb{K}:= \mathbb{F}(\overline{D})$ be their common function field. 

The natural morphism $\overline{D} \to \overline{\square}_{\eta}^{n+1} = \square_{\mathbb{F}} ^{n+1}$ is given by $(n+1)$ rational functions $y_1, \cdots, y_{n+1} \in \mathbb{K}$ on $\overline{D}$. By the given proper intersection condition $(GP)_*$ of $W$ with the faces, no $y_i$ is identically equal to $0$ or $\infty$, and any closed point $x \in D$ with $y_i (x) = 0$ or $\infty$ has $y_j (x) \not \in \{ 0, \infty\}$ for all $j \not = i$. 

For each closed point $x \in \overline{D}$, let $\partial_x: \widehat{K}_{n+1} ^M (\mathbb{K}) \to \widehat{K}_n ^M (\kappa (x))$
be the associated boundary map of Milnor \cite{Milnor IM}, where $\kappa (x)$ is the residue field of $x$. 

Since $\mathbb{K} \supset \mathbb{F}$, it is an infinite field. Thus Suslin's Weil reciprocity \cite{Suslin} applies, and for the Milnor symbol $\{ y_1, \cdots, y_{n+1} \} \in \widehat{K}_{n+1} ^M (\mathbb{K})$, we have the identity
$$
\sum_{x \in \overline{D}} {\rm N}_{ \kappa (x)/ \mathbb{F}} ( \partial _x  \{ y_1, \cdots, y_{n+1} \}) = 0 \ \ \mbox{ in } \widehat{K}_n ^M (\mathbb{F}),
$$
where the sum is over all closed points of $\overline{D}$ and the group operation in $\widehat{K}_n ^M (\mathbb{F})$ is written additively.

If $x \in \overline{D} \setminus D$, then one of $y_i$ has $y_i (x) = 1$ so that $\partial_x \{ y_1, \cdots, y_{n+1} \} = 0$. Thus the above sum can be rewritten in terms of the closed points over $D$ only, i.e.
\begin{equation}\label{eqn:Suslin n}
\sum_{x \in {D}} {\rm N}_{ \kappa (x)/ \mathbb{F}} ( \partial _x \{ y_1, \cdots, y_{n+1} \} )= 0 \ \ \mbox{ in } \widehat{K}_n ^M (\mathbb{F}),
\end{equation}
On the other hand, by the definition of $\partial_x$, we have the commutative diagram
$$
\xymatrix{
z^n _{{\rm d}} (X, n+1) \ar[d] \ar[rrr] ^{\partial} & & &  z^n _{{\rm d}} (X, n) \ar[d] ^{\phi_n} \\
\widehat{K}_{n+1} ^M (\mathbb{K}) \ar[rrr] ^{ \sum_x {\rm N}_{\kappa (x)/ \mathbb{F}}  ( \partial y \ ( - ) ) } & &  & \widehat{K}_n ^M (\mathbb{F}),
}
$$
where the left vertical arrow associates $W$ to the Milnor symbol $\{ y_1, \cdots, y_{n+1} \} \in \widehat{K}_{n+1} ^M (\mathbb{K})$. Hence $(\phi_{{\rm d}} \circ \partial ) (W)$ is equal to the left hand side of \eqref{eqn:Suslin n}, which is $0$ in $\widehat{K}_n ^M (\mathbb{F})$. This proves the proposition.
\end{proof}

To proceed further, we need to resolve some additional technical puzzles to be addressed in the next subsections. We will come back to our discussions in \S \ref{sec:conseq strong graph}.

\subsection{Triangular generators}\label{sec:triangular}
Now let $X= \Spec (A)$ be an integral \emph{regular} henselian local $k$-scheme of dimension $1$ with the residue field $k=A/I$ and the function field $\mathbb{F}$.

Let $Z \in z^n_{{\rm d}} (X, n)$ be an integral cycle and let $\overline{Z}$ be its Zariski closure in $X \times \overline{\square}^n$. By Corollary \ref{cor:proper int t=0}, we see that $\overline{Z} \to X$ is finite. By Lemma \ref{lem:proper int face *}, the intersections of $\overline{Z}$ with $\{y_i = \infty\}$ for $1 \leq i \leq n$, are empty. Thus in fact $\overline{Z} \subset X \times \mathbb{A}^n$.

We use this observation in what follows, which is largely motivated by \cite[Proposition 3.3.1]{Park MZ}, but not without a few necessary modifications:

\begin{prop}\label{prop:triangular}
Let $n \geq 2$ be an integer. 
Let $Z \in z^n _{{\rm d}} (X,n )$ be an integral cycle.

Then there are polynomials in $A[ y_1, \cdots, y_n]$ of the form
$$
\tuborg
P_1 (y_1) \in A  [y_1], \\
\vdots\\
P_n (y_1, \cdots, y_n) \in A[y_1, \cdots, y_n],
\sluttuborg
$$
such that
\begin{enumerate}
\item the ideal $(P_1, \cdots, P_n)$ in $A [ y_1, \cdots, y_n]$ defines $\overline{Z}$,
\item for each $1 \leq i \leq n$, the polynomial $P_i$ is monic in $y_i$. When $i=1$, the constant term of $P_1$ is a unit in $A^{\times}$, while for $2 \leq i \leq n$ and the image $\bar{P}_i (y_i)$ of $P_i$ in the ring $\left( \frac{ A [ y_1, \cdots, y_{i-1}]}{ (P_1, \cdots, P_{i-1})} \right) [ y_i]$, its constant term is a unit in $\left( \frac{ A [ y_1, \cdots, y_{i-1}]}{ (P_1, \cdots, P_{i-1})} \right)^{\times}$,
\item for $2 \leq i \leq n$, when $P_i$ is seen as a polynomial in $y_i$ with the coefficients in $A [y_1, \cdots, y_{i-1}]$, all coefficients have their $y_j$-degrees $< \deg_{y_j} P_j$ for $1 \leq j < i$.
\end{enumerate}

\end{prop}

\begin{proof}
Since $\overline{Z} \to X$ is finite and surjective, the residue field of the generic fiber $\overline{Z}_{\eta}$ is a finite extension of $\kappa (\eta) = \mathbb{F}$, and it defines a closed point in $\square_{\mathbb{F}}^n$. Hence we obtain a triangular shaped set of polynomials
$$
\tuborg
p_1 (y_1) \in \mathbb{F} [ y_1], \\
p_2 (y_1, y_2) \in \mathbb{F} [ y_1, y_2], \\
\vdots \\
p_n (y_1, \cdots, y_n) \in \mathbb{F} [y_1, \cdots, y_n],
\sluttuborg
$$
which defines the closed point, such that each $p_i$ is monic in $y_i$ for $1 \leq i \leq n$, and when $2 \leq i \leq n$, if we regard $p_i$ as a polynomial in $y_i$ with its coefficients in $\mathbb{F} [ y_1, \cdots, y_{i-1}]$, then each coefficient has its $y_j$-degree strictly less than $\deg_{y_j} \ (p_j)$ for all $1 \leq j < i$. 

Here, $\overline{Z}$ is the Zariski closure of $\overline{Z}_{\eta}$ in $X \times \mathbb{A}^n$ as well. To obtain generators of $\overline{Z}$, for each $1 \leq i \leq n$, we clear the denominators of the coefficients in $\mathbb{F}$ of the terms of $p_i$ by the l.c.m. of them in $A $. We can do it because $A$ is a regular local ring, thus a UFD. The resulting polynomials will be called $P_1 (y_1) \in A [ y_1], \cdots, P_n (y_1, \cdots, y_n) \in A [ y_1, \cdots, y_n]$, and they satisfy the properties
\begin{enumerate}
\item [(a)] the ideal $(P_1, \cdots, P_n)$ in $A [ y_1, \cdots, y_n]$ defines $\overline{Z}$,

\item [(b)] for each $i$, the highest $y_i$-degree term of $P_i$ involves no other variables, and

\item [(c)] for $2 \leq i \leq n$, if we regard $P_i$ as a polynomial in $y_i$ with the coefficients in $A [y_1, \cdots, y_{i-1}]$, then each coefficient has its $y_j$-degree strictly less than $\deg_{y_j} \ (P_j)$ for all $1 \leq j < i$. 
\end{enumerate}

\medskip

The desired conditions (1) and (3) of the proposition follow from the above properties (a) and (c). We will show that the condition (2) can be also achieved.

\medskip

For $1 \leq i \leq n$, let $S_i:= A [y_1, \cdots, y_i]$. Let $\mathfrak{p}_n \subset S_n$ be the prime ideal that defines $\overline{Z}$ in $X \times \mathbb{A}^n$. We have the sequence of injective ring homomorphisms
$$
A \hookrightarrow S_1 \hookrightarrow \cdots \hookrightarrow S_{n-1} \hookrightarrow S_n,
$$
and let $\mathfrak{p}_i:= \mathfrak{p}_n \cap S_i$ for $1 \leq i \leq n-1$. By construction, we have $\mathfrak{p}_i= (P_1, \cdots, P_i)$, by we have again the sequence of injective ring homomorphisms
$$ 
A \hookrightarrow S_1 / \mathfrak{p}_1 \hookrightarrow \cdots \hookrightarrow S_{n-1}/ \mathfrak{p}_{n-1} \hookrightarrow S_n / \mathfrak{p}_n.
$$

Since $A \to S_n / \mathfrak{p}_n$ is finite and flat, and $A$ is a PID, all of the homomorphisms $A \to S_i / \mathfrak{p}_i$ for $1 \leq i \leq n$ are also finite and free because a submodule of a free submodule of finite rank over a PID is again free, by the fundamental theorem on finitely generated modules over a PID. 

For $1 \leq i \leq n$, the compactified projection $\overline{Z}^{(i)} \subset X \times \mathbb{A}^i$ is the integral closed subscheme given by $\Spec (S_i/ \mathfrak{p}_i)$. By construction, we have $\dim \ \overline{Z}^{(i)} = \dim \ \overline{Z}$ so that its codimension in $X \times \mathbb{A}^i$ is $i$, while $Z^{(i)} = \overline{Z}^{(i)} \cap (X \times \square^i)$.  

\medskip

By Lemma \ref{lem:adm compact proj}, we have $Z^{(i)} \in z^i _{{\rm d}} (X, i)$ for each $1 \leq i \leq n$. Coming back to the remaining property (2), first suppose $i=1$. In this case, the intersection of $\overline{Z}^{(1)}$ with the faces $\{ y_1 = \epsilon \}$ for $\epsilon \in \{ 0, \infty \}$ are empty (Lemma \ref{lem:proper int face *}). Thus, the leading coefficient and the constant term of $P_1 (y_1)$ are units in $A^{\times}$. Scaling by the leading coefficient of $P_1$, we may assume it is $1$, and in the process the constant term is still a unit. This proves (2), while the scaling does not disturb the pre-established properties (1) and (3).

Now suppose $i \geq 2$. Among the defining polynomials $P_1, \cdots, P_i$ of $\overline{Z} ^{(i)}$, the only one that involves the variable $y_i$ is $P_i$, and the highest $y_i$-degree term of $P_i$ does not involve any other variable. hence, the empty intersection of $\overline{Z}^{(i)}$ with the face $\{y_i = \infty \}$ means that the leading coefficient of $P_i$ in $y_i$ is a unit in $A^{\times}$, so after scaling it, we may assume it is $1$. Note that the pre-established properties (1) and (3) stay after scaling.

On the other hand, the empty intersection of $\overline{Z}^{(i)}$ with $\{y_i = 0 \}$ means the image $\overline{P}_i$ of $P_i$ in the quotient ring has a unit constant term. This proves (2).
\end{proof}

\subsection{A reduction argument for a pair of mod equivalent cycles}\label{sec:strong graph}

In \S \ref{sec:strong graph}, we present an argument of ``reduction to graph cycles" for a pair of mod $I^{m+1}$-equivalent cycles. An argument of this sort first appeared in \cite[\S 4.4]{Park MZ}. Since our cycles are defined in a bit different set-up, we give details following the outline there, with suitable modifications.

The central statement is the following:

\begin{prop}\label{prop:strong graph}
Let $X= \Spec (A)$ be an integral regular henselian local $k$-scheme of dimension $1$ with the residue field $k= A/I$ and the function field $\mathbb{F}$.

For each pair of integral cycles $Z_1, Z_2 \in z^n _{{\rm d}} (X, n)$ such that $Z_1 \sim_{I^{m+1}} Z_2$, we have the following:

\begin{enumerate}
\item There exist
\begin{enumerate}
\item [(i)] graph cycles $Z_1', Z_2'\in z^n _{{\rm d}} (X, n)$ such that $Z_1' \sim_{I^{m+1}} Z_2'$, and
\item [(ii)] cycles $W_1, W_2 \in z^n _{{\rm d}} (X, n+1)$ such that 
$$Z_i - Z_i ' = \partial W_i,$$
for $i=1,2$.
\end{enumerate}
\item For an integer $1 \leq r \leq m$, if we had $Z_{\ell} \in z_{{\rm v}, \geq r} ^n (X, n)$ for $\ell = 1,2$ from the beginning, then we can also take $W_{\ell} \in z_{{\rm v}, \geq r} ^n (X, n+1)$ in the part $(1)$.

\end{enumerate}
\end{prop}

Its proof occupies most of \S \ref{sec:strong graph}. We continue to use the notations in the proof of Proposition \ref{prop:triangular}, that we recall briefly again: for $1 \leq i \leq n$, we let $\widehat{pr}_i: X \times \overline{\square}^n \to X \times \overline{\square}^i$ be the projection $(y_1, \cdots, y_n) \to (y_1, \cdots, y_i)$. For notational convenience, for $i=0$, we let $\widehat{pr}_0: X \times \overline{\square}^n \to X$ be the canonical projection. We denote all similar projections with $\square^n$ or $\mathbb{A}^n$ instead of $\overline{\square}^n$ by the notations $pr_i$. 

When $Z \in z^n_{{\rm d}} (X, n)$ is an integral cycle, we saw that $\overline{Z} \to X$ is finite surjective (Lemma \ref{lem:basic Milnor 1}). We have the compactified projection  $\overline{Z}^{(i)} = \widehat{pr}_i (\overline{Z})$ and it is still finite surjective over $X$. 

\begin{defn}
For a given $Z$ as the above, for each $0 \leq i < j \leq n$, and the induced finite surjective morphism $f^{j/i} : \overline{Z}^{(j)} \to \overline{Z}^{(i)}$, we define $d^{j/i} (Z):= \deg ( f^{j/i})$.
\qed
\end{defn}

We observe that $Z$ is a graph cycle if and only if $d^{n/0} (Z) = 1$ if and only if $d^{i/ (i-1)} (Z) = 1$ for all $1 \leq i \leq n$. 

\medskip

Thus Proposition \ref{prop:strong graph} says that, firstly each integral cycle in $z^n_{{\rm d}} (X, n)$ can be transformed into a graph cycle modulo boundaries, and secondly this process respects the mod $I^{m+1}$-equivalence. 

\medskip

Let $Z_1, Z_1 \in z^n_{{\rm d}} (X, n)$ be two integral cycles such that $Z_1 \sim_{I^{m+1}} Z_2$. For each $\ell = 1,2$, by Proposition \ref{prop:triangular}, we have a system of polynomials for $\overline{Z}_{\ell}$
$$
\tuborg
P_{\ell} ^{ (1)} ( y_1) \in A [ y_1], \\
\vdots\\
P_{\ell} ^{(n)} (y_1, \cdots, y_n) \in A [ y_1, \cdots, y_n],
\sluttuborg
$$
that satisfies the properties in Proposition \ref{prop:triangular}. For the compactified projections $Z_{\ell} ^{(i)}$ for $1 \leq i \leq n$ as before, we immediately have $Z_{1} ^{(i)} \sim_{I^{m+1}} Z_{2} ^{(i)}$ for $1 \leq i \leq n$. 

\medskip 

For $\ell = 0,1$ and $0 \leq i \leq n$, let $R_{\ell} ^{(i)}$ be the coordinate rings of $\overline{Z}_{\ell} ^{(i)}$. More precisely we have
$$ 
R_{\ell} ^{(0)} = A, \ \ R_{\ell} ^{(i)} = \frac{ A [ y_1, \cdots, y_i]}{ (P_{\ell} ^{ (1)}, \cdots, P_{\ell} ^{(i)} ) } \ \ \mbox{ for } \ 1 \leq i \leq n.
$$
We let $\overline{R}_{\ell} ^{(i)}$ be the reduction mod $I^{m+1}$, namely, 
$$
\overline{R}_{\ell} ^{(i)}:= R_{\ell} ^{(i)} / (I^{m+1}) R_{\ell} ^{(i)}.
$$

\medskip

Define $\overline{P}_{\ell} ^{(1)} := P_{\ell} ^{(1)}$, and for $2 \leq i \leq n$, let $\overline{P}_{\ell} ^{(i)}$ be the image of $P_{\ell} ^{(i)}$ in $R_{\ell} ^{(i-1)}[y_i]$. Let $d_{\ell} ^i:= \deg_{y_i} \ \overline{P}_{\ell} ^{(i)}$ for $1 \leq i \leq n$. We note that (cf. \cite[Lemma 4.4.4]{Park MZ}):

\begin{lem}\label{lem:4.4.4}
Let $\ell = 1,2$ and $2 \leq i \leq n$. Let $(-1)^{ d_{\ell}^i} c_{\ell} ^{(i)} \in R_{\ell} ^{(i-1)}$ be the constant term of $\overline{P}_{\ell} ^{(i)}$. By Proposition \ref{prop:triangular}-(2), we have $c_{\ell} ^{(i)} \in (R_{\ell} ^{ (i-1)})^{\times}$. 

Then under the above, we have
\begin{enumerate}
\item $d_1 ^i = d_2 ^i.$
\item $\overline{R}_1 ^{(i-1)} = \overline{R}_2 ^{(i-1)}$.
\item In the common ring $\overline{R}^{(i-1)}$ of the above, we have $c_1 ^{(i)} \equiv c_2 ^{(i)}$.
\end{enumerate}
\end{lem}

\begin{proof}
The argument follows the outline of \cite[Lemma 4.4.4]{Park MZ}. We repeat the story: by the properties of Proposition \ref{prop:triangular} for $P_{\ell} ^{(i)}$, we have $d_{\ell} ^i = \deg _{y_i} \ \overline{P}_{\ell} ^{(i)} \geq 1$. Going further modulo $I^{m+1}$, we denote by the double over-lined $\overline{\overline{P}}_{\ell} ^{(i)}$ be the image of $\overline{P}_{\ell} ^{(i)}$ in $\overline{R}_{\ell} ^{(i-1)} [y_i]$. 

Since $Z_1 ^{(i-1)} \sim_{I^{m+1}} Z_2 ^{(i-1)}$, we have $\overline{R}_1 ^{(i-1)} = \overline{R}_2 ^{(i-1)}$, thus we deduce (2). We will denote the common ring by $\overline{R}^{(i-1)}$ in what follows.

\medskip

Since $\overline{P}_{\ell} ^{(i)} (y_i)$ is monic in $y_i$ in the ring $R_{\ell} ^{(i-1)}[y_i]$, its further image $\overline{\overline{P}}_{\ell} ^{(i)}  (y_i)$ in $\overline{R}^{(i-1)}[y_i]$ is also monic in $y_i$ with the same $y_i$-degrees. Hence we obtain
\begin{equation}\label{eqn:d equal 1}
\deg_{y_i} \ \overline{P}_{\ell} ^{(i)} = \deg _{y_i} \ \overline{\overline{P}}_{\ell} ^{(i)}. 
\end{equation}

Since ${Z}_1 ^{(i)} \sim_{I^{m+1}} {Z}_2 ^{(i)}$, they give the common closed subscheme of $\Spec (\overline{R}^{(i-1)} [y_i])$ given by their respective monic polynomials $\overline{\overline{P}}_{\ell} ^{(i)}  (y_i) \in \overline{R} ^{(i-1)}[y_i]$. In other words, we have
\begin{equation}\label{eqn:d equal 1'}
 \overline{\overline{P}}_1 ^{(i)} = \overline{\overline{P}}_{2} ^{(i)} \ \ \mbox{ in } \overline{R} ^{(i-1)} [y_i].
 \end{equation}
In particular, we have 
\begin{equation}\label{eqn:d equal 2}
\deg_{y_i} \  \overline{\overline{P}}_1 ^{(i)} = \deg_{y_i} \ \overline{\overline{P}}_2 ^{ (i)}.
\end{equation}

Combining \eqref{eqn:d equal 1} and \eqref{eqn:d equal 2}, we deduce (1).

\medskip

On the other hand, the equality \eqref{eqn:d equal 1'} implies the equality of the constant terms $(-1)^{ d_1 ^i} c_1 ^{(i)} \equiv (-1)^{ d_2 ^{i}} c_2 ^{(i)}$ in $\overline{R}^{(i-1)}$. Since $d_1 ^i = d_2 ^i$ by (1), we deduce the equality $c_1 ^{(i)} \equiv c_2 ^{(i)} $ in $\overline{R} ^{(i-1)}$, proving (3).
\end{proof}

\begin{defn}
Based on Lemma \ref{lem:4.4.4}-(1), we let
$$
d_i:= d_1 ^i = \deg_{y_i} \ \overline{P}_1 ^{(i)} = \deg_{y_i} \ \overline{P}_2 ^{(i)} = d_2 ^i.
$$
By definition $d_i = d^{i/ (i-1)} ({Z}_{\ell})$ for $\ell = 1,2$, and it is now independent of $\ell$. We will call the vector $(d_1, \cdots, d_n) ^t \in \mathbb{N}^n$, the common \emph{degree vector} of the pair $(Z_1, Z_2)$. 

If $Z$ is an integer multiple of an integral cycle, we define its associated $d_i$ to be the degree for the underlying integral cycle.
\qed
\end{defn}

Before we discuss a crucial procedure, Lemma \ref{lem:4.4.6} below, we consider:

\begin{defn}
Suppose $n \geq 1$. For the set $\mathbb{N}= \{ 1, 2, \cdots \}$ of the natural numbers, consider the lexicographic order on $\mathbb{N}^n$. Its members are written as 
$$
(a_1, \cdots, a_n)^t = \begin{bmatrix} a_1 \\ \vdots \\ a_n \end{bmatrix}$$
for $a_i \in \mathbb{N}$. The set $\mathbb{N}^n$ is totally ordered, and $(1, \cdots, 1)^t$ is the smallest element. For example, when $n=3$, we have
$$
 (1, 1, 1)^t < (1, 1, 2)^t < (1, 2, 1)^t < (2, 1, 1)^t < (2, 1, 3)^t< \cdots,
 $$
in case the reader wants to see a few smallest members of $\mathbb{N}^3$ in this order. 
\qed
\end{defn}

The key ingredient in \S \ref{sec:strong graph} is to prove the following reduction process (cf. \cite[Lemma 4.4.6]{Park MZ}):

\begin{lem}\label{lem:4.4.6}
Let $Z_1, Z_2 \in z^n_{{\rm d}} (X, n)$ be a pair of integral cycles such that $Z_1 \sim_{I^{m+1}} Z_2$. Let
$\mathbf{d}:=(d_1, d_2, \cdots, d_n)^t \in \mathbb{N}^n$
be the common degree vector for the pair $(Z_1, Z_2)$. 

\begin{enumerate}
\item If $(1, \cdots, 1)^t < \mathbf{d}$, then there exist cycles $C_{Z_{\ell}} \in z^n _{{\rm d}} (X, n+1)$ and positive integer multiples of integral cycles $Z_1', Z_2' \in z^n_{{\rm d}} (X,n)$ for $\ell = 1,2$ such that
\begin{enumerate}
\item [(i)] $Z_{\ell} - Z_{\ell}' = \partial (C_{Z_{\ell}})$ for $\ell  = 1,2$,
\item [(ii)] $Z_1 ' \sim_{I^{m+1}} Z_2'$, and 
\item [(iii)] the degree vector $\mathbf{d}'$ for the pair $(Z_1', Z_2')$ satisfies $\mathbf{d}' < \mathbf{d}$ with respect to the lexicographic ordering on $\mathbb{N}^n$.
\end{enumerate}

\item For some $1 \leq r \leq m+1$, in case $Z_{\ell} \in z^n _{{\rm v}, \geq r} (X, n)$, then we can choose cycles $C_{Z_{\ell}} \in z^n _{{\rm v}, \geq r} (X, n+1)$ so that we have $Z_1', Z_2' \in z^n_{{\rm v}, \geq r} (X, n)$ as well.
\end{enumerate}
\end{lem}

\begin{proof}
The argument below largely follows the outline of \cite[Lemma 4.4.6]{Park MZ} with some improvements and modifications. It may look a bit technical, but its basic insight still comes from \cite{Totaro}. We prove (1) first, and deduce (2) from it rather immediately at the end of the argument.

\medskip

Suppose that $0 \leq i \leq n$ is the \emph{smallest} integer such that for all $j > i$, we have $d_j = 1$. In case no member of $d_j$ is $1$, we let $i=n$. Since $(1, \cdots, 1)^t < \mathbf{d}$, we have at least $i \geq 1$.

\medskip

Since $i$ is fixed, to simplify the notations, we write $R_{\ell}$ for the coordinate ring $R_{\ell} ^{(i-1)}$ of $\overline{Z}_{\ell} ^{(i-1)}$, and write $\widetilde{R}_{\ell}$ for the coordinate ring $R_{\ell} ^{(i)}$ of $\overline{Z}_{\ell} ^{(i)}$. These are henselian local domains finite and flat over $A$ (Lemmas \ref{lem:henselian} and \ref{lem:basic Milnor 1}). We have the inclusion $R_{\ell} \hookrightarrow \widetilde{R}_{\ell}$ corresponding to the projection $f_{\ell} ^{ i/ (i-1)}: \overline{Z}_{\ell} ^{(i)} \to \overline{Z}_{\ell} ^{(i-1)}$. 

\medskip

By the given assumption that $d_j = 1$ for all $j > i$, the images $\overline{P}_{\ell} ^{(j)}$ of $P_{\ell} ^{(j)} (y_1, \cdots, y_n)$ in $\widetilde{R}_{\ell} [ y_{i+1}, \cdots, y_n]$ are given by
\begin{equation}\label{eqn:446}
\overline{P}_{\ell} ^{(j)} (y_{i+1}, \cdots y_j) = y_j - c_{\ell} ^{(j)},
\end{equation}
where we recall that $(-1)^{d_j} c_{\ell} ^{ (j)} = - c_{\ell} ^{(j)}$ is the constant term of $\overline{P}_{\ell} ^{(j)}$.

On the other hand, recall that $(-1)^{d_i} c_{\ell} ^{(i)}$ is the constant term of the polynomial $\overline{P}_{\ell} ^{(i)} \in R_{\ell} [ y_i]$, and we had $c_{\ell} ^{(i)} \in R_{\ell} ^{\times}$. We first introduce a scheme $C_{\ell}$ which will be used in the middle of our eventual construction of $C_{Z_{\ell}}$ that comes below.

Let $(y_i, y_i')$ be the coordinates of $\square^2$ in $\Spec (R_{\ell}) \times \square^2$, and consider the closed subscheme $C_{\ell} \subset \Spec (R_{\ell}) \times \square^2$ given by the polynomial (cf. \cite[Lemma 2]{Totaro})

$$
Q_{Z_{\ell}} (y_i, y_i'):= \overline{P}_{\ell} ^{(i)} (y_i) - (y_i -1)^{d_i -1} ( y_i - c_{\ell} ^{ (i)}) y_i ' \in R_{\ell} [ y_i, y_i'].
$$
Let $\overline{C}_{\ell}$ be its Zariski closure in $\Spec (R_{\ell}) \times \overline{\square}^2$. 

\medskip

We first check that $C_{\ell} \in z_{{\rm d}} ^1 (\Spec (R_{\ell}), 2)$. To do this, let's inspect its extended faces $\overline{C}_{\ell} \cap (\Spec (R_{\ell})\times \overline{F})$ for proper faces $F \subset \square^2$. Let $F \subset \square^2$ be a codimension $1$ face. Then by direct calculations, in the space $ \overline{\square}^2_{R_{\ell}}$ the extended face $\overline{C} _{\ell} \cap (\Spec (R_{\ell}) \times \overline{F})$ is equal to the following:
\begin{enumerate}
\item [(a)] ($F= \{ y_i = 0\}$:) $ = \{ y_i = 0 , (-1)^{d_i} c_{\ell} ^{(i)} (1- y_i') =0\} = \{ y_i = 0, y_i ' = 1 \}$.

\item [(b)] ($F= \{ y_i = \infty \}$:) $=^{\dagger} \{ y_i = \infty, 1- y_i ' = 0 \}$.

\item [(c)] ($F= \{ y_i ' = 0 \}$:) $=\{ \overline{P}_{\ell} ^{(i)} (y_i) = 0, y_i' = 0 \} $.

\item [(d)] ($F=\{ y_i ' = \infty \}$:) $=\{ ( y_i -1)^{d_i -1} ( y_i - c_{\ell} ^{(i)}) = 0, y_i ' = \infty \}  =  (d_i-1) \{ y_i = 1 , y_i ' = \infty \} + \{ y_i = c_{\ell} ^{ (i)}, y_{i}' = \infty \}$.

\end{enumerate}
where $\dagger$ uses that the polynomial $\overline{P}_{\ell} ^{(i)} (y_i) \in R_{\ell} [ y_i]$ is monic in $y_i$. When $F \subset \square^2$ is a codimension $2$ face, we can compute $\overline{C} _{\ell} \cap (\Spec (R_{\ell}) \times \overline{F})$ by taking the intersections of two of those in (a) $\sim$ (d), but one notes that they are all empty by inspection. In particular, we deduce that $C_{\ell}$ satisfies the condition $(GP)_*$.

\medskip

To check the condition $(SF)_*$ for $C_{\ell}$, we need to inspect $\overline{C}_{\ell} \cap ( \{ p\} \times \overline{F})$ for faces $F \subset \square^2$, where $p\in \Spec (R_{\ell})$ is the closed point, where we use Lemma \ref{lem:henselian} to see that there is a unique closed point. If $F \subset \square^2$ is of codimension $2$, by the above, we have $\overline{C}_{\ell} \cap ( \{ p\} \times \overline{F}) = \emptyset$. If $F \subset \square^2$ is of codimension $1$, by inspecting the extended faces in (a) $\sim$ (d), one deduces that the intersection is proper on $\overline{\square}^2_{R_{\ell}}$. 

It remains to check the case when $F= \square^2$. Note from (a) $\sim$ (d) that $\overline{C}_{\ell}$ has an extended codimension $1$ face that is already surjective over $X$ (e.g. for $\{ y_i' = \epsilon \}$), so that $\overline{C}_{\ell} \to X$ is surjective. Hence $\overline{C}_{\ell} \not \subset \{ p \} \times \overline{\square}^2$, in particular, the intersection $\overline{C}_{\ell} \cap ( \{ p \} \times \overline{\square}^2)$ is proper on $\overline{\square}^2_{R_{\ell}}$. Hence we checked that $C_{\ell}$ satisfies the condition $(SF)_*$, proving that $C_{\ell} \in z_{{\rm d}} ^1 (\Spec (R_{\ell}), 2)$.

\medskip

Once again from the calculations in (a) $\sim$ (d), restricted onto the open subscheme $\Spec (R_{\ell}) \times \square^2$, we deduce that 
$$
\partial C_{\ell} = - \{ \overline{P}_{\ell} ^{(i)} (y_i) = 0 \} + [c_{\ell} ^{(i)} ]  \in z^1 _{{\rm d}} (\Spec (R_{\ell}) , 1).
$$

\medskip

For $1 \leq j \leq n$, let $\alpha_{\ell} ^{(j)}$ be the image of $y_i$ in $A [ y_1, \cdots, y_n]/ \mathcal{I} (\overline{Z}_{\ell})$, where $ \mathcal{I} (\overline{Z}_{\ell})$ is the ideal of $\overline{Z}_{\ell}$. 

For $j>i$, since we supposed that $d_j = 1$, we have
$$
A [ y_1, \cdots, y_n]/ \mathcal{I} (\overline{Z}_{\ell}) = \widetilde{R}_{\ell},
$$
so that under this identification, $\alpha_{\ell} ^{(j)}$ coincides with $c_{\ell} ^{(j)}$ for $j > i$.

\medskip

Let $(Y_1, \cdots, Y_{i-1}, Y_i, Y_i', Y_{i+1}, \cdots, Y_n)$ be the coordinates of $\square^{n+1}$ in $\Spec (R_{\ell}) \times \square^{n+1}$. Define the cycle $\widetilde{C}_{Z_{\ell}} \in z^n _{{\rm d}} (\Spec (R_{\ell}), n+1)$ given by the system of polynomials
$$
\tuborg
P_{\ell} ^{(1)} (Y_1), \\
\vdots\\
P_{\ell} ^{(i-1)} (Y_1, \cdots, Y_{i-1}), \\
Q_{Z_{\ell}} (Y_i, Y_i'), \\
P_{\ell} ^{(i+1)} (Y_1, \cdots, Y_{i+1}), \\
\vdots\\
P_{\ell} ^{(n)} (Y_1, \cdots, Y_n)
\sluttuborg
$$
in $R_{\ell} [ Y_1, \cdots, Y_{i-1}, Y_i, Y_i', Y_{i+1}, \cdots, Y_n]$. Taking \eqref{eqn:446} into account, we can express $\widetilde{C}_{Z_{\ell}}$ in terms of the concatenation notations as
$$
\widetilde{C}_{Z_{\ell}} = Z_{\ell} ^{(i-1)} \boxtimes C_{\ell} \boxtimes [ c_{\ell} ^{(i+1)}] \boxtimes \cdots \boxtimes [ c_{\ell} ^{(n)}].
$$
For the finite surjective morphism $\pi_{\ell}= f_{\ell} ^{ (i-1)/ 0}: \Spec (R_{\ell}) \to X$, and the associated push-forward map $\pi_{\ell, *}$ as in Lemma \ref{lem:fpf}-(1), define
$$
C_{Z_{\ell}} := (-1)^i \pi_{\ell, *} \left( \widetilde{C}_{Z_{\ell}} \right) \in z^n _{{\rm d}} (X, n+1).
$$

Since the boundary operators $\partial$ commute with the push-forward operators (Lemma \ref{lem:fpf}), the above allows us to compute $\partial C_{Z_{\ell}}$ by a straightforward calculation, namely
\begin{equation}\label{eqn:process final}
\partial C_{Z_{\ell}} = (-1)^i \pi_{\ell, *} ( \partial \widetilde{C}_{Z_{\ell}}) = Z_{\ell} -f_{\ell, *} ^{ i/0} ( \widetilde{Z}_{\ell}'),
\end{equation}
where $\widetilde{Z}_{\ell} \subset \Spec ( \widetilde{R}_{\ell}) \times \square^n$ is defined by the polynomials in $\widetilde{R}_{\ell} [ Y_1, \cdots, Y_n]$ 
$$
\widetilde{Z}_{\ell}' : \ \ \{ Y_1 - \alpha_{\ell} ^{(1)}, \cdots, Y_{i-1} - \alpha_{\ell} ^{ (i-1)}, Y_i - c_{\ell} ^{(i)}, \cdots, Y_n - c_{\ell} ^{ (n)} \}.
$$
Since all $\alpha_{\ell} ^{ (j)}$ for $j<i$ and $c_{\ell} ^{ (j)}$ for $j \geq i$ are in $\widetilde{R}_{\ell}$, the above is an integral cycle. Hence, letting $Z_{\ell}':= f_{\ell, *} ^{i/0} (\widetilde{Z}_{\ell} ')$, the equation \eqref{eqn:process final} gives for $\ell = 1,2$
$$
\partial C_{Z_{\ell}} = Z_{\ell} - Z_{\ell} ',
$$
proving (1)-(i).

\medskip

Let's check that the cycles $Z_{\ell}'$ satisfy the remaining desired properties. By definition this is a positive integer multiple of an integral cycle, so we can define their degree vectors. Since $Z_1 \sim_{I^{m+1}} Z_2$, by Lemma \ref{lem:4.4.4} we know that $Z_1 ^{(i-1)} \sim_{I^{m+1}} Z_2 ^{(i-1)}$ so that $c_1 ^{(i)} \equiv c_2 ^{(i)}$ in $R_1/ I^{m+1} R_1 = R_2 / I^{m+1} R_2$. We also already know that $Z_1 ^{(i)} \sim_{I^{m+1}} Z_2 ^{(i)}$ so that $c_1 ^{(j)} \equiv c_2 ^{(j)}$ in $\widetilde{R}_1/ I^{m+1} \widetilde{R}_1 = \widetilde{R}_2 / I^{m+1} \widetilde{R}_2$. Hence $Z_1' \sim_{I^{m+1}} Z_2'$, proving (1)-(ii).

In particular, $Z_1', Z_2'$ have the common degree vector $\mathbf{d}' = (d_1 ', \cdots, d_n')$, where $d_j' = d^{ j/(n-1)} (Z_{\ell}')$ for both $\ell = 1,2$. 

Since $c_{\ell} ^{(i)} \in R_{\ell}$, and $(Z_{\ell}') ^{(i)} = Z_{\ell} ^{ (i-1)} \boxtimes [ c_{\ell} ^{(i)}]$, the coordinate rings of $(\overline{Z}_{\ell}' )^{(i)}$ and $(\overline{Z}_{\ell}')^{(i-1)}$ are both $R_{\ell}$. Thus $d^{i/(i-1)} (Z_{\ell} ') = d_i = 1$. Here, the degrees $d_j$ in $\mathbf{d}$ with $j<i$ do not change, while by the minimality of $i$, we have $d_i \not = 1$. Hence we have
$$\mathbf{d}' < \mathbf{d}$$
in the lexicographic order on $\mathbb{N}^n$, proving (1)-(iii), and we have the desired cycles $C_{Z_{\ell}}$ and $Z_{\ell}'$ for $\ell = 1,2$.

\medskip

For (2), we now suppose $Z_{\ell}$ for $\ell = 1,2$ are strict vanishing cycles of order $ \geq r$. Since $Z_{\ell}$ have vanishing coordinates and they must match up as they are mod $I^{m+1}$-equivalent, after reindexing the variables, we may assume that $y_1$ is a common vanishing coordinate such that $P_{\ell} ^{(1)} (y_1) = 0$ modulo $I^{m+1}$ has only $y_1 = 1$ as the solutions for $\ell = 1, 2$, so that the cycles $\widetilde{C}_{Z_{\ell}}$ are automatically strict vanishing cycles of order $\geq r$.

Thus being push-forwards of strict vanishing cycles of order $\geq r$, we deduce from Lemma \ref{lem:fpf} that $C_{Z_{\ell}} \in z^n _{{\rm v}, \geq r} (X, n+1)$. This proves (2).
\end{proof}

\medskip

Now we get to:

\begin{proof}[Proof of Proposition \ref{prop:strong graph}]
For the given pair $(Z_1, Z_2)$, with $\mathbf{d} > (1, \cdots, 1)^t$ in $\mathbb{N}^n$, we can apply Lemma \ref{lem:4.4.6} repeatedly. Here, the set $\mathbb{N}^n$ is a well-ordered set with the lexicographic order, so after a finite number of applications of the Lemma \ref{lem:4.4.6}, we will get to a pair of mod $I^{m+1}$-equivalent graph cycles $(Z_1', Z_2')$. 

Thus, taking the sum of the all relations, after the cancellations of all the intermediate cycles whose degree vectors between $(1, \cdots, 1)^t$ and $\mathbf{d}$, we eventually obtain $\partial C_{\ell} = Z_{\ell} - Z_{\ell}'$ for $\ell = 1,2$, with the degree vectors of the pair $(Z_1', Z_2')$ are both equal to $(1, \cdots, 1)^t$ and $Z_1' \sim_{I^{m+1}} Z_2'$. This finishes the proof.
\end{proof}

\subsection{The comparison isomorphisms}\label{sec:conseq strong graph}
We continue to suppose that $X= \Spec (A)$ is an integral regular henselian local $k$-scheme of dimension $1$ with the residue field $k= A/I$, and the function field $\mathbb{F}= {\rm Frac} (A)$. We now finish the proof of Theorem \ref{thm:intro main v d}.

\begin{cor}\label{cor:gr_n surj}
 $gr_{{\rm d}}: \widehat{K}_n ^M (A) \to \CH^n_{{\rm d}} (X, n)$ of \eqref{eqn:graph final 1} is surjective. 
\end{cor}

\begin{proof}
Let $Z \in z^n_{{\rm d}} (X, n)$ be an integral cycle. We apply the Proposition \ref{prop:strong graph}-(1) with $Z_1 = Z_2 = Z$. Then $Z$ is equivalent to a graph cycle $Z' \in z^n_{{\rm d}} (X, n)$ modulo a boundary, so that the classes of $Z$ and $Z'$ in $\CH^n_{{\rm d}} (X, n)$ are equal. However, the latter belongs to the image of $gr_{{\rm d}}$ by definition. Thus $gr_{{\rm d}}$ is surjective.
\end{proof}

\begin{cor}\label{cor:gr_n surj m}
$gr_{{\rm d}}: \widehat{K}_n ^M (A/I^{m+1}) \to \CH^n_{{\rm d}} (X/ (m+1), n)$ of \eqref{eqn:graph final 3} is surjective.
\end{cor}

\begin{proof}
It follows from the commutativity of the diagram
$$
\xymatrix{ \widehat{K}_n ^M (A) \ar[rr] \ar[d] & & \CH_{{\rm d}} ^n (X, n) \ar[d] \\
\widehat{K}_n ^M (A/I^{m+1}) \ar[rr] & & \CH_{{\rm d}} ^n (X/ (m+1), n),}
$$
where the vertical maps are the canonical surjections, and the top horizontal map is surjective by Corollary \ref{cor:gr_n surj}. 
\end{proof}

\begin{thm}\label{thm:gr_n iso summary}
For the map $\phi_{{\rm d}} : \CH_{{\rm d}} ^n (X, n) \to K_n ^M (\mathbb{F})$ in \eqref{eqn:phi d 0}, we have:
\begin{enumerate}
\item The image of $\phi_{{\rm d}}$ belongs to the image of the injective homomorphism $\widehat{K}_n ^M (A)\to \widehat{K}_n ^M (\mathbb{F})$ induced by $A \hookrightarrow \mathbb{F}$, so that we can write $\phi_{{\rm d}}$ as
\begin{equation}\label{eqn:phi gr final}
\phi_{{\rm d}}: \CH^n_{{\rm d}} (X, n) \to \widehat{K}_n ^M (A).
\end{equation}
\item The above $\phi_{{\rm d}}$ in \eqref{eqn:phi gr final} satisfies $\phi_{{\rm d}} \circ gr_{{\rm d}} = {\rm Id}$. In particular $gr_{{\rm d}}$ is an isomorphism.

\item When $|k| > M_n$, $gr_{{\rm d}}: K^M_n (A) \to \CH_{{\rm d}} ^n (X, n)$ is also an isomorphism.
\end{enumerate}
\end{thm}

We remark again that the injectivity of $\widehat{K}_n ^M (A)\to \widehat{K}_n ^M (\mathbb{F})$ holds by the Gersten conjecture for $\widehat{K}_n ^M$ proven by M. Kerz \cite{Kerz finite}.

\begin{proof}
(1) By Corollary \ref{cor:gr_n surj}, the map $gr_{{\rm d}}: K_n ^M (A) \to \CH^n_{{\rm d}} (X, n)$ is surjective. In particular, $\CH_{{\rm d}} ^n (X, n)$ is generated by the graph cycles of the form $Z=\Gamma_{(a_1, \cdots, a_n)}$ for $a_i \in A^{\times}$. However, for such cycles, by Lemma \ref{lem:regulator graph}, we know that $\phi_{{\rm d}} (Z)$ belongs to the image of $\widehat{K}_n ^M (A) \to \widehat{K}_n ^M (\mathbb{F})$. Thus ${\rm im} (\phi_{{\rm d}}) = \widehat{K}^M_n (A)$.

\medskip

(2) is immediate by checking at the symbols $\{ a_1, \cdots, a_n \} \in \widehat{K}_n ^M (A)$ for $a_i \in A^{\times}$.

\medskip

(3) This holds because $K^M _n (A) = \widehat{K}^M_n (A)$ when $|k| > M_n$ by Theorem \ref{thm:Khat_univ}.
\end{proof}

We now go modulo $I^{m+1}$:

\begin{thm}\label{thm:local main 1} 
Let $k$ be an arbitrary field. Let $X= \Spec (A)$ be an integral regular henselian local $k$-scheme of dimension $1$ with the residue field $k= A/I$. Let $n, m\geq 1$ be integers. Then:

\begin{enumerate}
\item The graph homomorphisms
\begin{equation}\label{eqn:local main 1 0}
gr_{{\rm d}}:  \widehat{K}_n ^M (A/I^{m+1}) \overset{\simeq}{\to} \CH^n_{{\rm d}}( X/ (m+1), n),
\end{equation}
are isomorphisms, where we recall that $A/ I^{m+1} \simeq k_{m+1}$.
\item In case $|k|> M_n$, the homomorphisms
$$
gr_{{\rm d}}:  {K}_n ^M (A/I^{m+1}) \overset{\simeq}{\to} \CH^n_{{\rm d}}( X/ (m+1), n),
$$
are also isomorphisms.
\end{enumerate}

\end{thm}

\begin{proof}
(1) By Corollary \ref{cor:gr_n surj m}, the map $gr_{{\rm d}}$ of \eqref{eqn:local main 1 0} is surjective. On the other hand, Proposition \ref{prop:strong graph} and Theorem \ref{thm:gr_n iso summary} imply that we have $\phi_{{\rm d}}: \CH^n _{{\rm d}} (X/(m+1), n) \to K_n ^M (A/I^{m+1})$.

The composite $\phi_{{\rm d}}\circ gr_{{\rm d}}$ is the identity on $\widehat{K}_n ^M (A/I^{m+1})$. This proves that $gr_{{\rm d}}$ of \eqref{eqn:local main 1 0} is an isomorphism.

\medskip

(2) This holds because $K_n ^M (A) = \widehat{K}_n ^M (A)$ when $|k|> M_n$ by Theorem \ref{thm:Khat_univ}.
\end{proof}

We deduce the relative version of Theorems \ref{thm:gr_n iso summary} and  \ref{thm:local main 1} as well, using Proposition \ref{prop:strong graph}:

\begin{thm}\label{thm:local main 2}
Let $k$ be an arbitrary field. Let $X= \Spec (A)$ be an integral regular henselian local $k$-scheme of dimension $1$ with the residue field $k= A/I$. Let $n, m, r\geq 1$ be integers.  

\begin{enumerate}
\item The graph homomorphisms
$$
gr_{{\rm v}, \geq r } : \widehat{K}_n ^M (A, I^r)  \overset{\simeq}{\to} \CH^n_{{\rm v}, \geq r} (X, n)
$$
are isomorphisms. If $1 \leq r \leq m +1$, then
$$
gr_{{\rm v}, \geq r}:  \widehat{K}_n ^M (A/I^{m+1}, I^r) \overset{\simeq}{\to} \CH^n_{{\rm v}, \geq r}( X/ (m+1), n),
$$
are also isomorphisms.

\item If $|k|> M_n$, then the above holds with $\widehat{K}_n ^M$ replaced by $K_n ^M$.
\end{enumerate}

\end{thm}

\section{The Bass-Tate-Kato norms for the Artin $k$-algebras}\label{sec:applications}

We deduce some consequences of Theorems \ref{thm:gr_n iso summary}, \ref{thm:local main 1} and \ref{thm:local main 2}.

\subsection{The short exact sequences}

\begin{cor}\label{cor:ses final}
Let $k$ be an arbitrary field. Let $X= \Spec (A)$ be an integral regular henselian local $k$-scheme of dimension $1$ with the residue field $k = A/I$. Let $m, n, r \geq 1$ be integers. Then:

\begin{enumerate}
\item We have the short exact sequence
$$
0 \to \CH^n _{{\rm v}} (X, n) \to \CH^n_{{\rm d}} (X, n) \to \CH^n (k, n) \to 0.
$$
\item For $1 \leq r \leq m+1$, we have the short exact sequences
$$
 0 \to \CH^n_{{\rm v}, \geq r} (X/ (m+1), n) \to \CH^n_{{\rm d}} (X/ (m+1), n) \to \CH^n_{{\rm d}} (X/r, n) \to 0.
 $$
\end{enumerate}
\end{cor}

\begin{proof}
(1) We have the apparent short exact sequence
$$
0 \to \widehat{K}_n ^M (A, I) \to \widehat{K}_n ^M (A) \to K_n ^M (k) \to 0.
$$
Then (1) follows by combining Theorems \ref{thm:gr_n iso summary}, and Theorem \ref{thm:NST}.

\medskip

(2) By definition, we have the short exact sequence
$$
 0 \to \widehat{K}^M _n (A/I^{m+1}, I^r) \to \widehat{K}^M_n (A/I^{m+1}) \to \widehat{K}^M_n (A/I^r) \to 0.
 $$

The corollary then follows by combining the isomorphisms in Theorems \ref{thm:local main 1} and \ref{thm:local main 2}.
\end{proof}

\begin{remk}
Corollary \ref{cor:ses final} is stated separately due to the following reason: if we define $K_{\bullet}:=\ker ({\rm ev}: z^q _{{\rm d}} (X, \bullet) \to z^q (k, \bullet)$, where ${\rm ev}$ is the reduction mod $I$, then in general we can obtain $z^q _{{\rm v}} (X, \bullet) \subset K_{\bullet}$, and we can deduce the following part of the long exact sequence
$$
\CH_{{\rm v}} ^n (X, n) \to \CH_{{\rm d}} ^n (X, n)  \to \CH^n (k, n) \to 0.
$$
Here, the injectivity of the first arrow is not immediately deducible by looking at the connecting homomorphism $\CH^n (k, n+1) \to \CH_{{\rm v}} ^n (X, n)$. However, through Theorems \ref{thm:gr_n iso summary} and \ref{thm:local main 2}, we see that the first map must be an injection.
\qed
\end{remk}

\subsection{Norms and traces for the Milnor $K$-groups}
For a finite separable extension $k \hookrightarrow k'$ of fields, the base change $\pi :\Spec (k_{m+1}') \to \Spec (k_{m+1})$ is finite \'etale. Hence we know that there exists the norm maps ${\rm N}_{k'/k}: \widehat{K}^M_n (k_{m+1}') \to \widehat{K}^M_n (k_{m+1})$ by Theorem \ref{thm:Khat_univ}. Now our new results Theorems \ref{thm:local main 1} and \ref{thm:local main 2} allow us to improve it to any finite extension of fields:

\begin{defn}
Let $k \hookrightarrow k'$ be a finite extension of fields and let $\pi: \Spec (k') \to \Spec (k)$ be the associated morphism. Let $m, n , r \geq 1$ be integers such that $1 \leq r \leq m+1$. Let $X$ be an integral regular henselian local $k$-scheme of dimension $1$ with the residue field $k$.

Define the norm and the trace maps ${\rm N}_{k'/k}$, ${\rm Tr}_{k'/k}$ on $\widehat{K}_n ^M (k'_{m+1})$ and $\widehat{K}_n ^M (k'_{m+1}, (t^r))$, respectively, to be the composites 
$$
{\rm N}_{k'/k}: \widehat{K}_n ^M (k'_{m+1}) \simeq \CH_{{\rm d}} ^n (X_{k'} / (m+1), n)$$
$$
 \overset{\pi_*}{\to} \CH_{{\rm d}} ^n (X/ (m+1), n) \simeq \widehat{K}_n ^M (k_{m+1}),$$
and
$$
{\rm Tr}_{k'/k}: \widehat{K}_n ^M k_{m+1} ', (t^r)) \simeq \CH_{{\rm v}, \geq r} ^n (X_{k'}/ (m+1), n)$$
$$
 \overset{\pi_*}{\to} \CH_{{\rm v}, \geq r} ^n (X/ (m+1), n) \simeq \widehat{K}_n ^M (k_{m+1}, (t^r)),$$
where $\pi_*$ between the cycle class groups are the push-forwards given by Lemma \ref{lem:fpf}. 
\qed
\end{defn}

Since the push-forwards on cycles are transitive, we deduce that:

\begin{cor}
The norms ${\rm N}_{k'/k}$ and traces ${\rm Tr}_{k'/k}$ defined for all finite extensions $k \hookrightarrow k'$ are transitive. 
\end{cor}

\bigskip

\noindent\textbf{Acknowledgments.}  Part of this work was done while the author was on his sabbatical leave at the Center for Complex Geometry (CCG) of the Institute for Basic Science of South Korea in the years 2022-2023. The author thanks the Director Jun-Muk Hwang and the members of CCG for their hospitality. 

This work was supported by Samsung Science and Technology Foundation under Project Number SSTF-BA2102-03.

\end{document}